\documentclass{article}

\newcommand{\mc}[1]{\mathcal{#1}}

\usepackage{subcaption}

\usepackage{indentfirst}
\usepackage{xcolor}
\usepackage{verbatim}
\usepackage{tikz}
\usetikzlibrary{decorations.pathreplacing}
\usetikzlibrary{calc}

\usepackage{hyperref}
\usepackage{bm}
\usepackage{stmaryrd}
\usepackage[a4paper, total={6in, 8in}]{geometry}
\usepackage{amsmath}
\usepackage{amsthm}	
\usepackage{amsfonts}	
\usepackage{amssymb, bbm, units}
\usepackage{enumerate, accents}
\usepackage[nobysame, alphabetic, initials]{amsrefs}
\usepackage{ulem} 

\numberwithin{equation}{section}
\numberwithin{figure}{section}
 \usepackage[nodayofweek]{datetime}

\theoremstyle{plain}
\newtheorem{theorem}{Theorem}[section]
\newtheorem{lemma}[theorem]{Lemma}
\newtheorem{proposition}[theorem]{Proposition}
\newtheorem{corollary}[theorem]{Corollary}

\newtheorem{definition}[theorem]{Definition}
\newtheorem{remark}[theorem]{Remark}

\newtheorem{assumption}[theorem]{Assumption}

\newcommand{\cB}{\mathcal{B}}

\newcommand{\cF}{\mathcal{F}}
\newcommand{\cG}{\mathcal{G}}

\newcommand{\cK}{\mathcal{K}}
\newcommand{\cL}{\mathcal{L}}
\newcommand{\cM}{\mathcal{M}}

\newcommand{\cP}{\mathcal{P}}

\newcommand{\cR}{\mathcal{R}}
\newcommand{\cS}{\mathcal{S}}
\newcommand{\cT}{\mathcal{T}}

\newcommand{\AAA}{\mathbb{A}}

\newcommand{\CC}{\mathbb{C}}
\newcommand{\DD}{\mathbb{D}}
\newcommand{\EE}{\mathbb{E}}
\newcommand{\FF}{\mathbb{F}}
\newcommand{\GG}{\mathbb{G}}
\newcommand{\II}{\mathbb{I}}
\newcommand{\KK}{\mathbb{K}}

\newcommand{\NN}{\mathbb{N}}
\newcommand{\PP}{\mathbb{P}}
\newcommand{\RR}{\mathbb{R}}
\newcommand{\SSS}{\mathbb{S}}
\newcommand{\XX}{\mathbb{X}}
\newcommand{\ZZ}{\mathbb{Z}}

\newcommand{\mf}[1]{\mathbf{#1}}

\newcommand{\tA}{\tilde{A}}
\newcommand{\tD}{\tilde{D}}
\newcommand{\tM}{\tilde{M}}
\newcommand{\tQ}{\tilde{Q}}
\newcommand{\tX}{\tilde{X}}

\newcommand{\Lip}[1]{\mbox{Lip}_{#1}}
\newcommand{\Lipbar}[1]{\ensuremath{\overline{\mathrm{Lip}}_{#1}}}

\newcommand{\brac}[1]{\left(#1\right)}

\newcommand\numeq[1]%
  {\stackrel{\scriptscriptstyle(\mkern-1.5mu#1\mkern-1.5mu)}{=}}

\makeatletter
\newsavebox{\@brx}
\newcommand{\llangle}[1][]{\savebox{\@brx}{\(\m@th{#1\langle}\)}%
  \mathopen{\copy\@brx\mkern2mu\kern-0.9\wd\@brx\usebox{\@brx}}}
\newcommand{\rrangle}[1][]{\savebox{\@brx}{\(\m@th{#1\rangle}\)}%
  \mathclose{\copy\@brx\mkern2mu\kern-0.9\wd\@brx\usebox{\@brx}}}
\makeatother

\newcommand{\toP}{\stackrel{p}{\rightarrow}}
\newcommand{\tof}{\stackrel{f}{\rightarrow}}

\definecolor{darkgreen}{rgb}{0,0.35,0}

\newcommand{\Pol}[1]{\mathrm{#1}}

\newcommand{\LISF}{\Pol{LISF}}     
\newcommand{\TB}{\Pol{TB}}         

\newcommand{\FSL}{\Pol{FSL}}       
\newcommand{\SSL}{\Pol{SSL}}       

\newcommand{\JFSQ}{\Pol{JFSQ}}     

\newcommand{\JSSQ}{\Pol{JSSQ}}     

\newcommand{\RJSQ}{\Pol{RR}}     

\newcommand{\JFIQ}{\Pol{JFIQ}}     

\usetikzlibrary{arrows.meta, positioning, calc, fit, backgrounds}

\definecolor{cGray}{HTML}{E5E5E5}
\definecolor{cGrayBorder}{HTML}{555555}

\definecolor{cTeal}{HTML}{D8EEF7}
\definecolor{cTealBorder}{HTML}{0072B2}

\definecolor{cPink}{HTML}{FCE1C2}
\definecolor{cPinkBorder}{HTML}{E69F00}

\definecolor{cBlue}{HTML}{DDF0D7}
\definecolor{cBlueBorder}{HTML}{009E73}

\definecolor{cCoral}{HTML}{F7D8D8}
\definecolor{cCoralBorder}{HTML}{D55E00}

\definecolor{cGreen}{HTML}{E6DDF3}
\definecolor{cGreenBorder}{HTML}{6A3D9A}

\definecolor{cAmber}{HTML}{FFF1C7}
\definecolor{cAmberBorder}{HTML}{B8860B}

\definecolor{cProjBorder}{HTML}{777777}

\tikzset{
  base/.style={
    draw,
    rounded corners=4pt,
    align=center,
    font=\small,
    minimum height=1.2cm,
    minimum width=3.4cm,
    line width=0.5pt,
  },
  assumption/.style={base, fill=cGray, draw=cGrayBorder},
  prelim/.style={base, fill=cTeal, draw=cTealBorder},
  fairness/.style={base, fill=cPink, draw=cPinkBorder},
  proof/.style={base, fill=cBlue, draw=cBlueBorder},
  mainresult/.style={
    base,
    fill=cCoral,
    draw=cCoralBorder,
    minimum width=6.5cm,
    font=\small\bfseries
  },
  corollary/.style={base, fill=cGreen, draw=cGreenBorder},
  stationary/.style={base, fill=cAmber, draw=cAmberBorder},
  projected/.style={
    base,
    draw=cProjBorder,
    densely dashed,
    fill=none,
    text=cProjBorder,
    font=\small\itshape
  },
  arr/.style={-{Stealth[length=5pt]}, line width=0.5pt},
  darr/.style={
    -{Stealth[length=5pt]},
    line width=0.5pt,
    densely dashed,
    draw=black!55
  },
  separatorline/.style={gray, densely dash dot, line width=0.4pt},
  legendbox/.style={
    draw=cGrayBorder,
    rounded corners=3pt,
    fill=cGray!15,
    inner sep=6pt
  },
}
\tikzset{imported/.style={densely dashed, line width=0.9pt}}

\tikzset{
  arrowbase/.style={
    -{Stealth[length=5pt]},
    line width=0.65pt,
    line cap=round,
    line join=round
  },
  directarrow/.style={
    arrowbase
  },
  indirectarrow/.style={
    arrowbase,
    densely dashed
  },
  arrassumption/.style={
    directarrow,
    draw=cGrayBorder
  },
  arrprelim/.style={
    directarrow,
    draw=cTealBorder
  },
  arrfair/.style={
    directarrow,
    draw=cPinkBorder
  },
  arrproof/.style={
    directarrow,
    draw=cBlueBorder
  },
  arrcorollary/.style={
    directarrow,
    draw=cGreenBorder
  },
  arrstationary/.style={
    directarrow,
    draw=cAmberBorder
  },
  arrmain/.style={
    directarrow,
    draw=cCoralBorder
  },
  darrassumption/.style={
    indirectarrow,
    draw=cGrayBorder
  },
  darrprelim/.style={
    indirectarrow,
    draw=cTealBorder
  },
  darrfair/.style={
    indirectarrow,
    draw=cPinkBorder
  },
  darrproof/.style={
    indirectarrow,
    draw=cBlueBorder
  },
  darrcorollary/.style={
    indirectarrow,
    draw=cGreenBorder
  },
  darrstationary/.style={
    indirectarrow,
    draw=cAmberBorder
  },
  darrmain/.style={
    indirectarrow,
    draw=cCoralBorder
  },
  arr/.style={
    directarrow,
    draw=black
  },
  darr/.style={
    indirectarrow,
    draw=black!55
  },
  darrbg/.style={
    darrprelim
  },
  darrcorr/.style={
    darrcorollary
  }
}

\usepackage{booktabs}

\title{Load Balancing with Individually Heterogeneous Servers}
\author{Burak B\"{u}ke$^*$  \and Arpan Mukhopadhyay$^\dag$ \and \"{O}zge Tekin$^*$}
\date{%
    $^*$School of Mathematics, The University of Edinburgh\\
    $^\dag$Department of Computer Science, University of Warwick\\
}

\begin{document}

\maketitle
\begin{abstract}
In this paper, we analyze the diffusion limit of Join-the-Shortest-Queue (JSQ) load balancing policies for a system with many parallel queues where each queue is being served by a server with a potentially different service rate. Prior asymptotic analyses of JSQ policies assumed servers to have service rates from a finite set which does not capture scenarios in modern data centers where service rates can vary at the level of individual servers.
For systems with individually heterogeneous servers, tracking the empirical queue length distribution for each possible service rate becomes infeasible. To overcome this difficulty, we develop a framework based on measure-valued processes and provide a unified analysis of all JSQ-based load balancing policies under general tie-breaking rules. Our analysis identifies two key objects that distinguish the performance of different tie-breaking rules in the Halfin--Whitt regime, namely, the {\em limiting fairness process} which describes how idle servers are distributed across different service rates and the {\em limiting routing measure} which describes how arriving jobs are assigned across different service rates. In addition to characterizing the diffusion limits for different tie-breaking rules, we identify the tie-breaking rule which asymptotically minimizes the steady-state distributions of the diffusion-scaled total number of jobs and the diffusion-scaled number of waiting jobs. In proving these results, we develop crucial coupling-based sample-path comparisons which provide both policy-independent steady-state bounds and lower bounds to prove asymptotic optimality.
\end{abstract}
\section{Introduction}
\label{sec:Intro}

Load balancing is a central problem in modern data centers, which underpin cloud services and have become economically significant infrastructure for the broader digital economy \cite{ChenMoinzadehSongZhong2023}. A recent study \cite{McKinseyDataCenters2025} projects that meeting future data-center demand could require nearly \$7 trillion in worldwide investment by 2030. In the data-center environment, a stream of incoming jobs requiring service must be routed to one of many parallel servers each having a dedicated queue. The routing rule used to dispatch jobs balances the load across different servers and directly affects key performance measures such as delay and resource utilization. As these systems scale in size, designing effective routing/load balancing policies becomes both practically important and mathematically challenging.

A large part of the load balancing literature studies systems in which all servers are identical. In this homogeneous setting, the Join-the-Shortest-Queue (JSQ) policy, where each incoming job is routed to the server with the shortest queue, is known to have strong optimality properties \cite{Winston1977, Weber1978}. Since exact JSQ requires knowledge of all queue lengths, scalable alternatives such as the power-of-$d$-choices policy \cite{VDK96,Mitz96} and Join-the-Idle-Queue \cite{Lu2011,Stolyar2015} have been developed to reduce information requirements. We refer the reader to \cite{vdB2022} for a comprehensive survey. 

Despite the optimality guarantees of JSQ in the homogeneous setting, characterizing the performance of JSQ in the heavy-traffic regimes where the system load approaches its capacity has remained an important direction of research as it provides crucial insights into dimensioning the system to achieve target performance.
The Halfin--Whitt regime, first introduced in \cite{HalfinWhitt1981}, is an important example of heavy-traffic regimes, where both the number $n$ of servers and the total system load $n\lambda^n$ vary such that the gap between capacity and the load scales as $O(n^{1/2})$. In this traffic regime, Eschenfeldt and Gamarnik \cite{eschenfeldt2018join} established a process-level diffusion limit for the JSQ occupancy process over finite time intervals. They showed that a scaled process counting the number of idle servers and queues of length exactly two weakly converges to a two-dimensional reflected Ornstein--Uhlenbeck process, while the processes counting longer queues converge to a deterministic system decaying to zero. Mukherjee et al.~\cite{mukherjee2016universality} showed that the same diffusion limit is shared by a wider class of load balancing policies.  Using a generator-expansion framework based on Stein's method, Braverman \cite{Braverman2020} proved exponential ergodicity of the limiting diffusion and established the tightness of the diffusion-scaled pre-limit stationary distributions, which in turn implies the convergence of the stationary distributions of the pre-limit processes to that of the limiting diffusion. The steady-state convergence rate was later quantified in \cite{BraSSY2023}, where an $O(1/\sqrt n)$ bound was obtained for the finite-buffer JSQ system using the pre-limit generator comparison approach. Banerjee and Mukherjee \cite{BM19,BM20} analyzed the stationary behavior of the limiting diffusion, including tail asymptotics, scaling of extrema, and sensitivity with respect to the heavy-traffic parameter, and Hurtado-Lange and Maguluri \cite{HLM2022} studied other related heavy-traffic scaling regimes for load balancing systems. This line of work provides a powerful analytical foundation, but it relies on the simplifying assumption that servers are homogeneous.

Since real service systems are seldom composed of identical servers, a substantial literature has investigated load balancing under heterogeneous service rates. A line of work, in particular, has studied many-server systems where the servers can be grouped into a finite number of pools, each characterized by a different service rate, with all servers sharing jobs from a single queue. In this setting, Armony \cite{Armony2005} analyzed the fastest-server-first routing policy, where each job is assigned to the fastest server that becomes available earliest. The optimality of this policy was established in both fluid and Halfin--Whitt regimes. Armony and Ward \cite{AW2010} and Ward and Armony \cite{WA2013} further studied fairness-efficiency tradeoffs in such systems. To the best of our knowledge, the earliest work to study a setting in which each individual server could potentially have a different service rate was by Atar \cite{Atar2008}, who established a Halfin--Whitt diffusion limit for a shared-queue many-server system with i.i.d. random service rates. Atar and Shwartz \cite{AtarShwartz2008} showed that asymptotically efficient routing can be achieved using partial sampling of service-rate information. Atar, Shaki, and Shwartz \cite{AtarShakiShwartz2011} developed a blind routing policy that equalizes cumulative idleness across heterogeneous server pools. 

Load balancing systems with heterogeneous servers have been studied in a series of papers under the assumption of pool-based heterogeneity. In this setting, Bhambay and Mukhopadhyay \cite{BM2022} analyzed the Join-the-Fastest-of-the-Shortest-Queue (JFSQ) policy, which they refer to as the {\em speed-aware JSQ (SA-JSQ)} policy. Under this policy, each job is routed to the fastest server among those with the shortest queue. They showed that JFSQ asymptotically minimizes the average response time of jobs in the fluid limit. Most recently, Bhambay et al. \cite{bhambay2025asymptotic} studied the same setting in the Halfin--Whitt regime. For the JFSQ policy, they established a two-dimensional reflected Ornstein--Uhlenbeck diffusion limit, proved convergence of the stationary distributions, and established the asymptotic optimality of the JFSQ policy in minimizing both the diffusion-scaled total number of jobs and the diffusion-scaled number of waiting jobs. Cao and Zhong \cite{CaoZhong2025} proved asymptotic optimality under a long-run average delay-cost criterion in multi-station systems with statistically identical servers per station. Gardner et al. \cite{GHJM2021} designed heterogeneity-aware power-of-$d$ schemes, while Mukhopadhyay and Mazumdar \cite{MukhopadhyayMazumdar2016} showed that heterogeneity-unaware policies can suffer significantly degraded performance and even reduced stability. Together, these works show that heterogeneity is not merely a modeling refinement and the interaction between service rates and routing decisions can have a first-order effect on system performance. At the same time, they share a common structural restriction, namely that service rates take only finitely many distinct values. This assumption, while being analytically powerful, does not capture settings in which service rates vary at the level of individual servers.

This individual-level variability is increasingly visible in the large-scale systems that motivate load-balancing models. At the system level, modern data center platforms contain machines with different processor types, capabilities, and performance characteristics, and increasingly combine CPUs with specialized accelerators such as GPUs and FPGAs \cite{Verma2015Borg, CloudNativeSilicon2023, Duato2010, Huang2016, Um2024Metis, Subramanya2023Sia}. This kind of variability naturally motivates pool-based models, in which servers can be grouped into a finite number of classes. At the individual-processor level, however, effective service rates may differ even across nominally identical processors because of fabrication variability, workload-dependent frequency scaling, thermal throttling, and power-management constraints \cite{IntelTBM3_2025, Romanescu2008, IntelAVX_2025, IntelRAPLGuidance2025}. Empirical studies confirm substantial variation across such processors, with reported performance differences of up to 20\% and further amplification under power constraints \cite{Marathe2017}. Similar considerations also arise in human-operated service systems, where effective service rates may differ across agents because of experience, specialization, learning, or fatigue \cite{Gans2010}. These observations motivate modeling heterogeneity through the empirical distribution of service rates rather than through a finite collection of pool rates.

The presence of individual-level heterogeneity changes the analysis of the system in a fundamental way. With non-identical service rates, the total service capacity at each queue-length level depends not only on how many servers are present there, but also on the distribution of their service rates. Aggregate occupancy counts therefore no longer suffice to determine the system dynamics, and the analysis must instead track how queue lengths are distributed across service rates. This naturally leads to a measure-valued state descriptor. This shift from finite-dimensional occupancy processes to measure-valued dynamics is precisely what makes diffusion-limit analysis technically challenging. 

B\"uke \cite{Buke2022} identified diffusion-limit analysis for load-balancing systems with individual-level heterogeneity as an important open problem. B\"uke and Qin \cite{buke2023many} subsequently developed a general framework for many-server shared queue systems with heterogeneous service rates based on measure-valued stochastic processes. Central to their framework, and essential for the present paper, is the fairness process, which records how cumulative idleness is distributed across service rates under a given routing policy. They also characterize the limiting fairness process under common routing rules such as fastest-server-first, slowest-server-first, longest-idle-server-first, and random routing, showing how different policies lead to different idleness allocations across service rates. Within this framework, the fairness process records how cumulative idleness is allocated across service rates and provides the additional information required to describe the limiting system after state-space collapse. For the routing policies considered in \cite{buke2023many}, the limiting diffusion is one-dimensional and only the first moment of the limiting fairness process appears as a multiplier in the diffusion limit.  Building on this framework, B\"uke, dos Reis, and Platonov \cite{BdRP2025} extended the measure-valued approach to systems with strategic servers via mean-field game ideas. 

The work most closely related to the present paper is that of Liu and Ying \cite{LiuYing2025}, who studied a load balancing system with individual-server heterogeneity under the JFSQ policy and established asymptotically zero waiting time and waiting probability in heavy traffic. Their analysis covers both the sub- and super-Halfin--Whitt regimes and provides explicit steady-state bounds that depend on the system size and the degree of heterogeneity. While their focus is on the steady-state delay performance of the JFSQ policy, our work develops a general framework for analyzing both the transient and stationary behavior of load-balancing systems at the diffusion scale under general tie-breaking policies. In particular, our framework enables the comparison of routing policies in terms of diffusion-scaled queue lengths and the total number of jobs in the system.

In this paper, we study load-balancing systems with individual-level heterogeneous service rates in the Halfin--Whitt regime. Each server has its own service rate, and the empirical distribution of these rates is assumed to converge weakly to a limiting distribution. We develop a diffusion-limit framework for load-balancing systems with individual-level heterogeneous service rates in the Halfin--Whitt regime under a broad class of tie-breaking rules. The limiting system consists of a reflected stochastic differential equation for the diffusion-scaled idle-server process coupled with a family of measure-valued integral equations for the higher queue levels. The routing policy enters the limit only through two policy-dependent objects: the limiting fairness process, which describes how cumulative idleness is distributed across service rates, and the limiting routing measure, which describes how arriving jobs are assigned across service rates in the limit. In addition to the transient diffusion limit, we establish stationary approximation results by proving policy-uniform steady-state bounds and convergence of the stationary distributions of the pre-limit processes to that of the limiting diffusion. The resulting framework extends the homogeneous JSQ diffusion limit of \cite{eschenfeldt2018join} to systems with individual-level heterogeneity under general tie-breaking rules, generalizes the pool-based speed-aware analysis of \cite{bhambay2025asymptotic}, and extends the measure-valued methodology of \cite{buke2023many} to dedicated-queue load balancing systems.

\subsection{Contributions}
The contributions of this paper are summarized as follows.
\begin{enumerate}
    \item[(C1)] \textbf{A general diffusion limit framework for heterogeneous JSQ systems}\\
    We establish a general framework for deriving diffusion limits for heterogeneous JSQ systems under a broad class of tie-breaking policies. The key structural feature is that the diffusion limit depends on the routing policy only through two policy-dependent objects: the limiting fairness process and the limiting routing measure. Once the limiting fairness process and the limiting routing measure associated with a tie-breaking policy have been identified, the corresponding diffusion limit follows directly from our general limiting equations. A notable structural feature is that these two objects decouple, allowing diffusion limits to be derived in a mix-and-match fashion when a policy uses different rules for idle and busy servers. This enables a wide range of load-balancing policies to be analyzed within a unified framework, rather than requiring each policy to be studied from scratch. As applications of the general result, we derive explicit diffusion limits for JFSQ, JSQ with random tie-breaking and Join-the-Fastest-Idle-Queue (JFIQ), and we show that the pool-based formulation of \cite{bhambay2025asymptotic} is recovered as a special case whenever the limiting service-rate distribution has finite support.
    \item[(C2)] \textbf{Policy-uniform steady-state bounds and interchange of limits}\\
    We establish policy-uniform bounds on the stationary moments of the diffusion-scaled queue-length processes (Theorem~\ref{thm:stationary_bounds}). These bounds enable the stationary behavior of general tie-breaking policies to be analyzed within a single framework and provide the key ingredient for our stationary approximation results and the interchange of limits. In particular, we show that under any tie-breaking rule, the expected scaled number of idle servers and the expected scaled number of servers with at least two jobs remain $O(1)$, while the expected scaled number of servers with at least three jobs remains $O(n^{-1/2})$, as long as the dispatching policy routes each incoming job to the shortest queue. This extends Theorem~2 of \cite{bhambay2025asymptotic} from JFSQ to arbitrary tie-breaking policies and sharpens the higher-level stationary moment bound to $O(n^{-1/2})$.
    
    \item[(C3)] \textbf{Asymptotic optimality of JFSQ under individual-server heterogeneity}\\
We prove that JFSQ is asymptotically optimal in steady state under individual-level heterogeneity (Theorem~\ref{thm:asymptotical_optimality}). This shows that the optimality of routing to the fastest server among those with the shortest queue extends beyond finite-pool models to arbitrary limiting service-rate distributions. Specifically, the stationary diffusion-scaled total number of jobs and the stationary diffusion-scaled number of waiting jobs under JFSQ are asymptotically optimal in the Halfin--Whitt regime. This extends the pool-based optimality result of \cite{bhambay2025asymptotic} to arbitrary limiting service-rate distributions, establishing the asymptotic optimality of JFSQ for fully heterogeneous load-balancing systems.
\item[(C4)] \textbf{Auxiliary bounding systems}\\
As a central methodological tool, we develop a coupling-based sample-path comparison using the auxiliary systems Join-the-Shortest-Queue-Fastest-Serves-the-Longest (JSQ-FSL) and Join-the-Shortest-Queue-Slowest-Serves-the-Longest (JSQ-SSL), where servers and queues are decoupled and servers may move between queues in a preemptive fashion. While JSQ-FSL was introduced in \cite{bhambay2025asymptotic} to obtain sample-path lower bounds for pool-based systems, we introduce the complementary JSQ-SSL system, which provides sample-path upper bounds for the relevant queue-length processes. The comparison holds at the sample-path level for each fixed $n$ and is independent of the Halfin--Whitt scaling, yielding direct stochastic boundedness without requiring separate fluid-limit arguments or truncated reflected systems. These auxiliary systems provide the comparison framework underlying both the stationary bounds and the asymptotic optimality analysis.

\item[(C5)] \textbf{Fixed-point arguments via weighted Kantorovich--Rubinstein norms}\\
To establish the existence and uniqueness of solutions to the limiting measure-valued equations, we develop a contraction mapping argument based on switching between weighted Kantorovich--Rubinstein norms. This technique is motivated by the measure-valued setting, where the standard contraction argument in \cite{eschenfeldt2018join} no longer applies directly. Interestingly, the same norm-switching argument also simplifies the original proof of \cite{eschenfeldt2018join} by establishing contraction directly using the system equations, avoiding the introduction of an auxiliary system of equations and the subsequent uniqueness proof.
\end{enumerate}

The remainder of the paper is organized as follows. Section~2 introduces the notation and topological framework used throughout the paper. Section~3 describes the heterogeneous JSQ model and introduces the associated occupancy, fairness, routing, and departure processes. Section~4 states the main transient and stationary results. The subsequent sections develop the auxiliary comparison systems and establish the steady-state bounds, asymptotic optimality, and diffusion-limit results. Supporting technical arguments are collected in the appendices. For ease of reference, Appendix~\ref{app:dependency_map} provides a dependency map summarizing the logical structure of the transient and stationary analyses and the connections among the main results and supporting lemmas.

\section{Notation}
\label{sec:notation}

We use $\RR$, $\RR_+$, and $\RR_-$ to denote the set of real numbers, non-negative real numbers, and non-positive real numbers, respectively. Similarly, we use $\ZZ$, $\ZZ_+$, and $\NN$ to denote the set of integers, non-negative integers, and positive integers, respectively. 

For a real interval $[a,b]\subset\RR$ and a metric space $(\XX,d_{\XX})$, let $\CC_{\XX}[a,b]$, $\CC_{\XX}^b[a,b]$, $\DD_{\XX}[a,b]$, and $\GG_{\XX}[a,b]$ denote the spaces of $\XX$-valued functions that are continuous, bounded continuous, right-continuous, and left-continuous, respectively. We endow $\CC_{\XX}[0,\infty)$ and $\DD_{\XX}[0,\infty)$ with the topology of uniform convergence on compact sets and the Skorokhod $J_1$ topology, respectively. For all $b<\infty$, we equip the space $\GG_{\XX}[a,b]$ with the Skorokhod $J_1$ topology by mapping each $f\in\GG_{\XX}[a,b]$ to $g_f\in\DD_{\XX}[0,b-a]$, where $g_f(t):=f(b-t)$.

For a real-valued function $f$, we denote its supremum over a compact set $\KK$ by $|f|_\KK^*$. If $\KK=[0,T]$, we use the shorthand $|f|_T^*$. For $\RR^m$ valued functions $\mf{f}=(f_1,f_2,\ldots,f_m)$, we define
\[
|\mf{f}|_T^* := \sum_{i=1}^m |f_i|_T^*.
\]
For infinite dimensional vectors $\mf{x}=(x_1,x_2,\ldots)\in\RR^\infty$, we define $|\mf x|_\rho= \sum_{i=1}^\infty \rho^i|x_i|$ and for $\RR^\infty$ valued functions $\mf{f}=(f_1,f_2,\ldots)$, we use the notation
\[
|\mf{f}|_{\rho,T}^* := \sum_{i=1}^\infty \rho^i |f_i|_T^*, \quad 0 < \rho < 1.
\]

We also define the set of all $\XX$ valued $L$-Lipschitz functions on $[a,b] \subset \RR_+$ by
\begin{align*}
\Lip{\XX, L}[a,b] := \left\{ f \in \CC_{\XX}[a,b] : d_{\XX}(f(s),f(t)) \leq L|t-s| \ \text{for all } s,t \in [a,b] \right\}.
\end{align*}
For real valued functions, we further define
\[
\Lipbar{\RR, L}[a,b] := \Lip{\RR, L}[a,b] \cap \left\{ f \in \CC_{\RR}^b[a,b] : |f|_{[a,b]}^* \leq L \right\}.
\]

In this work, we assume that service rates take values in the interval $\SSS = [\mu_{\min}, \mu_{\max}]$, where $0 < \mu_{\min}\leq 1\leq \mu_{\max} < \infty$ with $\mu_{\min}<\mu_{\max}$. We consider measures on the measurable space $(\SSS,\cB)$ where $\cB$ corresponds to the Borel $\sigma$-algebra on $\SSS$. We write $\cM$ for the set of all measures, $\cM_f$ for the set of all measures with finite total mass, $\cM_s$ for the set of all signed finite measures.
We use the inner product notation for the integration of a measurable function 
$f$ with respect to a measure $\alpha \in \cM$, namely
\[
\langle f, \alpha \rangle = \int_{\SSS} f(x)\, d\alpha(x).
\]
We denote the zero measure and the Dirac measure assigning unit mass at $\mu$ by $0$ and $\delta_\mu$, respectively. The indicator function $\II(p)$ takes the value $1$ if proposition $p$ is true and $0$ otherwise, while $\iota(\cdot)$ denotes the identity function. Hence, for any measure $\alpha\in\cM$, $\langle \iota,\alpha\rangle$ corresponds to the first moment of $\alpha$. 

Weak convergence of finite measures on $\SSS$ can be metrized by the Kantorovich--Rubinstein metric (cf.~Chapter 8.3 of \cite{bogachev2007measure}). For any $L>0$ and $\alpha,\beta\in\cM$, we define the scaled Kantorovich--Rubinstein metric by
\begin{equation}
    \label{eq:measure_metric_definition}
    d_{\cM,L}(\alpha, \beta) = \sup\left\{|\langle f, \alpha \rangle - \langle f, \beta \rangle|:f\in \Lipbar{\RR, L} (\SSS)\right\}.
\end{equation}
The standard Kantorovich-Rubinstein norm assumes $L=1$ and $ d_{\cM,L}(\alpha, \beta) = L d_{\cM,1}(\alpha, \beta)$. We define the set of $\cM$-valued $L$-Lipschitz functions using the standard Kantorovich-Rubinstein norm as 
\[
\Lip{\cM,L}[a,b] = \left\{\alpha\in \CC_{\cM}[a,b]: d_{\cM,1}(\alpha(s),\alpha(t))\leq L|t-s| \mbox{ for all }s,t\in [a,b]\right\}.
\]
Allowing arbitrary values of $L$ provides additional flexibility in the contraction arguments developed later in the paper. For any $\alpha\in \DD_{\cM_s}[0,T]$, we also define $|\alpha|_{TV,t}^*:=\sup_{0\leq s\leq t}|\alpha(s)|_{TV}$ with $|\cdot|_{TV}$ denoting the total variation norm of a measure. The above definitions are extended to finite signed measures naturally. 

All random quantities in this paper are defined on a common probability space $(\Omega, \cF,\PP)$ unless stated otherwise. We use ``$\Rightarrow$'' to denote weak convergence of measures and convergence in distribution of random quantities, and ``$\toP$'' to denote convergence in probability. For any positive recurrent Markov process $X=\{X(t):t\geq 0\}$, we use $X(\infty)$ to denote a random variable with the stationary distribution of $X$.

\section{The Model Description and Related Stochastic Processes}\label{sec:model}

We consider a sequence of queueing systems, where the $n$th system receives external arrivals according to a Poisson process $A^n(t)$ with rate $n\lambda^n$ and consists of $n$ servers each serving its own queue of jobs in the first-come-first-served (FCFS) fashion. Upon arrival, each job is immediately routed to a server using a join-the-shortest-queue (JSQ) policy, where ties are broken according to a pre-specified rule $\pi$. The service times for jobs at server $k$ are i.i.d.~exponential random variables with rate $\mu_k^n \in \SSS$. For each $n\in \NN$, we denote the service rate vector by $\boldsymbol{\mu}^n = (\mu_k^n:1\leq k\leq n)$  and the departure process from the system by $D^{\pi,n}(t)$. We assume that arrival and service rates satisfy the following assumption.

\begin{assumption}\label{asm:arrival_service_rates}
Define the empirical service rate distribution of the $n$th system with service rate vector $\boldsymbol{\mu}^n$ by
\begin{equation*}
F_n(\mu) := n^{-1}\sum_{k=1}^n \II(\mu_k^n\leq \mu)\mbox{ and the sample mean by }\bar{\mu}^n:=n^{-1}\sum_{k=1}^n \mu_k^n.
\end{equation*}
There exists a probability measure with cumulative distribution function $F$ with support $\SSS$, mean $1$ and variance $\sigma_F^2$, together with a random variable $\nu$ with finite mean and variance such that $F_n\Rightarrow F$, $\nu^n:=n^{1/2}(\bar{\mu}^n-1)\Rightarrow \nu$ and $\EE[|\nu^n|]\to \EE[|\nu|]$. We also assume that $\varsigma^n := n^{1/2}(1-\lambda^n)\to \varsigma>0$ deterministically. 
\end{assumption}

Assumption~\ref{asm:arrival_service_rates} is sufficiently general and applies to a wide class of models. For example, the pool-level heterogeneity model considered in \cite{bhambay2025asymptotic} with $M$ pools corresponds to the case where $F$ is a discrete distribution with $M$ point masses and $\nu=0$ w.p.~1. This assumption also allows us to consider the case where service rates are i.i.d.~samples from a common distribution $F$ with variance $\sigma_F^2$, in which case the central limit theorem implies that $\nu$ is normally distributed with mean 0 and variance $\sigma_F^2$. In the remainder of the paper, we abuse the notation and use $F$ to denote both the cumulative distribution and the corresponding probability measure. 

For the $n$th system, we represent the state as $\mf{Q}^{\pi,n}(t):=(Q_{k,i}^{\pi,n}(t):1\le k\le n,\ i\in\NN)$ for all $t\ge0$, where
\[
Q_{k,i}^{\pi,n}(t):=
\begin{cases}
1, & \text{if there are $i$ or more jobs at server $k$ at time $t$,}\\
0, & \text{otherwise.}
\end{cases}
\]
We also define $Q_i^{\pi,n}(t):=\sum_{k=1}^n Q_{k,i}^{\pi,n}(t)$ for all $i\in\NN$ and $t\ge0$, and let $(\cF_t^n:t\ge0)$ denote the natural filtration associated with the system. To study the weak convergence of the system, we define the scaled processes
\begin{equation}
\label{eq:scaled_system}
X_{k,1}^{\pi,n}(t):= n^{-1/2}(Q_{k,1}^{\pi,n}(t) - 1), X_1^{\pi,n}(t) := \sum_{k=1}^n X_{k,1}^{\pi,n}(t), \text{ and } X_i^{\pi,n}(t) := n^{-1/2}Q_{i}^{\pi,n}(t),
\end{equation}
for all $1\leq k \leq n$ and $i\geq 2$. The process $X_1^{\pi,n}(t)$ denotes the negative of the diffusion-scaled number of idle servers in the $n$th system, whereas for $i\geq 2$, $X_i^{\pi,n}(t)$ denotes the scaled number of servers with $i$ or more jobs. We define the sequence $\mf{X}^{\pi,n}(t) := (X_i^{\pi,n}(t): i\in \NN)$. In this work the service rate vector $\boldsymbol{\mu}^n$ is allowed to be random. We assume that the service rates are observed at time 0 and remain constant over time. This is captured by the following assumption.
\begin{assumption}
    The service rate vector $\boldsymbol{\mu}^n$ is $\cF_0^n$-measurable. 
\end{assumption}



The aggregate quantities  $X_i^{\pi,n}(t)$ are insufficient to characterize the limiting dynamics in heterogeneous systems, since they do not retain information about the service-rate composition of the servers contributing to these counts. B\"{u}ke and Qin~\cite{buke2023many} propose using measure-valued processes to model the individual differences between servers and in this work, we adopt a similar approach. Two measure-valued processes play a central role in our analysis. The first is the fairness process introduced in~\cite{buke2023many}, which records how cumulative idleness is distributed across service rates and is needed to describe the effect of tie-breaking among idle servers. The second records the service-rate distribution of servers at higher queue levels and provides the additional information needed to describe the evolution of the queue-length process beyond the idle-server component. We define $\bar{\SSS}:=\SSS\cup\{0\}$ and let $\bar{\cB}$ be the Borel $\sigma$-algebra on $\bar{\SSS}$. By a slight abuse of notation, we also use $\cM_f$ to denote the space of finite measures on $(\bar{\SSS},\bar{\cB})$ whenever the underlying measurable space is clear from the context. For all $\epsilon\geq 0$, define
\begin{equation*}
    \tau_\epsilon^{\pi,n} = \inf\left\{t:\int_0^t X_1^{\pi,n}(s)ds <-\epsilon\right\}.
\end{equation*}
Then, for every $t\geq0$ and $\AAA_1\in\bar{\cB}$, and for every $i\geq2$ and $\AAA_2\in\cB$, define
\begin{align*}
    \eta^{\pi,n}(t)(\AAA_1) = \begin{cases}
       \displaystyle \frac{\sum_{k=1}^n\delta_{\mu_k^n}(\AAA_1)\int_0^t X_{k,1}^{\pi,n}(s)ds}{\int_0^t X_1^{\pi,n}(s)ds} &t> \tau_0^{\pi,n}\\
       \delta_0(\AAA_1) &t\leq\tau_0^{\pi,n}
    \end{cases}
    \mbox{ and }\xi_i^{\pi,n}(t)(\AAA_2) = n^{-1/2}\sum_{k=1}^n \delta_{\mu_k^n}(\AAA_2) Q_{k,i}^{\pi,n}(t).
\end{align*}
The definitions of these two measure-valued processes are structurally different. The probability measure-valued process $\eta^{\pi, n}\in \GG_{\cM_f}[0,\infty)$ is the fairness process defined in \cite{buke2023many} and corresponds to the \emph{self-normalized} distribution of the \emph{cumulative} idleness experienced by time $t$ among servers with different rates. In contrast, the process $\xi_i^{\pi, n}\in \DD_{\cM_f}[0,\infty)$ is non-normalized and takes values in the space of finite measures $\cM_f$ and corresponds to the \emph{instantaneous} distribution of servers with $i$ or more jobs. To simplify the notation, we define $\boldsymbol{\xi}^{\pi,n}(t)=(\xi_i^{\pi,n}(t): i\in \NN)$ with $\xi_1^{\pi,n}(t)=0$ and have the following usual assumption regarding the initial system state: 
\begin{assumption}\label{asm:xi_initial}
There exist random quantities $X_1(0)\in \RR_{-}$ and $\boldsymbol{\xi}(0)\in \cM_f^\infty$ such that $X_1^{\pi,n}(0)\Rightarrow X_1(0)$ and $\boldsymbol{\xi}^{\pi,n}(0)\Rightarrow \boldsymbol{\xi}(0)$ as $n\to \infty$, which in turn imply $\mf{X}^{\pi,n}(0)\Rightarrow \mf{X}(0)$. We also assume that there are initially finitely many jobs in the system, i.e., $\sum_{i=1}^\infty Q_i^{\pi,n}(0)<\infty$ for all $n\in \NN$.
\end{assumption}

The fairness process, $\eta^{\pi,n}$, has been studied in~\cite{buke2023many} in detail. B\"uke and Qin~\cite{buke2023many} show that the self-normalization causes a singularity at $\tau_0^{\pi,n}$ that prevents the processes $\eta^{\pi,n}$ from converging weakly when $\GG_{\cM_f}[0,\infty)$ is equipped with any of the four topologies introduced in~\cite{sko56}. To overcome this problem, they suggest using a modified definition for the limiting fairness process as follows. 
\begin{definition}\label{def:limiting_fairness}
    For any $\epsilon>0$, define the $\epsilon$-shifted fairness process as 
    \begin{equation}
    \cS_\epsilon\eta^{\pi,n}(t)(\AAA) = \begin{cases}
        \eta^{\pi,n}(t)(\AAA) &\mbox{if }t\geq \tau_{\epsilon}^{\pi,n}\\
        \delta_0(\AAA) &\mbox{if }t<\tau_\epsilon^{\pi,n}.
    \end{cases}\label{eq:shifted_process_definition}
    \end{equation}
    A probability measure-valued process $\eta^\pi \in \GG_{\cM_f}[0,\infty)$ is said to be the limiting fairness process for $(\eta^{\pi,n}:n\in \NN)$, if for all $\epsilon>0$, $\tau_\epsilon^{\pi,n}\Rightarrow \tau_\epsilon^\pi$ and $\cS_\epsilon\eta^{\pi,n}\Rightarrow \cS_\epsilon\eta^\pi$ as $n\to\infty$, where $\cS_\epsilon\eta^\pi$ is defined by replacing $\tau^{\pi,n}_\epsilon$ with $\tau_\epsilon^\pi$ in \eqref{eq:shifted_process_definition}. We denote this mode of convergence by $\eta^{\pi,n} \tof \eta^\pi$.
\end{definition}

To analyze the convergence of fairness processes under any tie-breaking rule $\pi$, we need the stochastic boundedness result in Lemma~\ref{lem:stochastic_boundedness}, which is also required for our main result. The proof of Lemma~\ref{lem:stochastic_boundedness} is provided in Appendix~\ref{app:proof_stochastic_boundedness}.

\begin{lemma}\label{lem:stochastic_boundedness}
    The sequence of processes $(\mf{X}^{\pi,n}:n\in \NN)$ is stochastically bounded under any tie-breaking rule $\pi$.
\end{lemma}

We next review the results on fairness processes from~\cite{buke2023many}. Although originally derived for $M/M/n$ systems with a single global queue, these results extend naturally to JSQ systems.
The stochastic boundedness of $(X_1^{\pi,n}:n\in \NN)$ naturally yields the following tightness result. The proof follows the same lines as in the proof of Theorem 2 in~\cite{buke2023many} and is provided in Appendix~\ref{app:fairness_measure_results} for completeness. 

\begin{lemma}[cf. Theorem 2 in~\cite{buke2023many}]\label{lem:fairness_tightness}
Under any tie-breaking policy $\pi$ among idle servers, the sequence of fairness processes $(\eta^{\pi,n}:n\in \NN)$ is tight in the sense of Definition~\ref{def:limiting_fairness}, i.e., for any subsequence, there exists a further subsequence that possesses a limiting fairness process.
\end{lemma}

Lemma~\ref{lem:fairness_tightness} acts as the first step in obtaining limits of fairness processes. The existence of a unique limiting fairness process follows if all subsequences converge to the same limit. The following two propositions from~\cite{buke2023many} characterize the limiting fairness processes for some common tie-breaking rules. The proofs extend with only minor modifications to JSQ systems and are therefore omitted.

\begin{proposition}[cf. Theorems 3 and 4 in~\cite{buke2023many}]\label{prop:priority_idleness}
    The limiting fairness processes for the JSQ systems, where ties among idle servers are broken by routing to the fastest server (JFSQ) or the slowest server (JSSQ), respectively, are given by
    \[
    \eta^{\JFSQ}(t)(\AAA) = \begin{cases}
        \delta_{\mu_{\min}}(\AAA) &\mbox{if }t> \tau_0^{\JFSQ}\\
        \delta_0(\AAA) &\mbox{if }t \leq \tau_0^{\JFSQ}
    \end{cases} \mbox{ and } 
    \eta^{\JSSQ}(t)(\AAA) = \begin{cases}
        \delta_{\mu_{\max}}(\AAA) &\mbox{if }t> \tau_0^{\JSSQ}\\
        \delta_0(\AAA) &\mbox{if }t\leq \tau_0^{\JSSQ}
    \end{cases}.
    \]
\end{proposition}

\begin{proposition}[cf. Theorems 6, 7 and 8 in~\cite{buke2023many}]\label{prop:totally_blind_idleness}
For $j\in \NN$, let $\tilde{\theta}_j^{\pi,n}$ denote the $j$th departure epoch of $D^{\pi,n}(t)$, and let $\phi_j^{\pi,n}$ denote the duration for which the server completing service at time $\tilde{\theta}_j^{\pi,n}$ remains idle. Similarly, for any $1\leq k \leq n$, define $\phi_{-k}^{\pi,n}$ to be the amount of time that server $k$ remains idle after time $0$ if it is initially idle, and $0$ otherwise. We say that a tie-breaking rule $\pi$ among idle servers is \emph{totally blind} if, for every $T>0$, it satisfies the following two properties:
\begin{enumerate}
    \item The conditional expectation of the scaled idle duration $\phi_j^{\pi,n}$ becomes asymptotically insensitive to the specific server, i.e.,
    \[
    \sup_{1\leq j\leq D^{\pi,n}(T)} n^{1/2}\,\EE\!\left[\left|\EE\!\left[\phi_j^{\pi,n}\mid\cF_{\tilde{\theta}_j^{\pi,n}-}^n\right]-\EE\!\left[\phi_j^{\pi,n}\mid\cF_{\tilde{\theta}_j^{\pi,n}}^n\right]\right|\bigm|\cF_{\tilde{\theta}_j^{\pi,n}-}^n\right] \stackrel{p}{\longrightarrow} 0.
    \]
    \item The unscaled expected residual idle time of any server becomes negligible:
    \[
    \sup_{0\leq t \leq T}\sup_{-n\leq j\leq D^{\pi,n}(t)} \EE\!\left[(\phi_j^{\pi,n}-t+\tilde{\theta}_j^{\pi,n})^+\mid\cF_{t}^n\right] \stackrel{p}{\longrightarrow} 0.
    \]
\end{enumerate}
The limiting fairness process under any totally blind tie-breaking rule is given as
\begin{align}\label{eq:totally_blind_fairness}
\eta^{\TB}(t)(\AAA) = \begin{cases}
        \frac{\int_\AAA \mu dF(\mu)}{\int_\SSS \mu dF(\mu)} &\mbox{if }t> \tau_0^{\TB}\\
        \delta_0(\AAA) &\mbox{if }t\leq \tau_0^{\TB}
    \end{cases}.
\end{align}
Moreover, the longest-idle-server-first (LISF) policy, which routes an arriving job to the server that has been idle the longest, and the random routing (RR) policy, which assigns jobs uniformly at random among idle servers, are both totally blind.
\end{proposition}

To understand Proposition~\ref{prop:totally_blind_idleness}, note that the cumulative idleness of servers in the set $\AAA$ up to time $t$ depends both on the number of times departures leave these servers idle and on the idling duration at each such epoch.  Totally blind policies are designed to equalize the idle times once a server becomes idle as follows: The main difference between $\cF_{\tilde{\theta}_j^{\pi,n}-}^n$ and $\cF_{\tilde{\theta}_j^{\pi,n}}^n$ is that the latter contains information about the service rate of the server completing a job at time $\tilde{\theta}_j^{\pi,n}$. Hence, condition~1 implies that the idling time of a server is asymptotically insensitive to this additional information. In other words, the policy does not substantially favor or penalize servers based on their rate after they become idle, motivating the term totally blind. Condition~2 ensures that the residual idle time is negligible and hence, instantaneous variations do not contribute significantly to the cumulative idleness distribution.

For example, under the LISF policy, any server becoming idle at time $\tilde{\theta}_j^{\LISF,n}$ remains idle until exactly $n-Q_{1}^{\LISF,n}(\tilde{\theta}_j^{\LISF,n})$ jobs have arrived, regardless of its service rate. In this case, we have
\begin{align*}
\EE\!\left[\phi_j^{\LISF,n} \mid \cF_{\tilde{\theta}_j^{\LISF,n}-}^n\right]
= \EE\!\left[\phi_j^{\LISF,n} \mid \cF_{\tilde{\theta}_j^{\LISF,n}}^n\right]
&= (n\lambda^n)^{-1} (n-Q_{1}^{\LISF,n}(\tilde{\theta}_j^{\LISF,n}))\\
&= n^{-1/2} (\lambda^n)^{-1} |X_1^{\LISF,n}(\tilde{\theta}_j^{\LISF,n})|,
\end{align*}
and
\begin{align*}
\sup_{0 \leq t \leq T} \sup_{-n \leq j \leq D^{\LISF,n}(t)} 
\EE\!\left[(\phi_j^{\LISF,n} - t + \tilde{\theta}_j^{\LISF,n})^+ \mid \cF_t^n\right]
\leq n^{-1/2} (\lambda^n)^{-1} |X_1^{\LISF,n}|_T^*.
\end{align*}
Hence, one expects servers that become idle more often to receive more idleness and the cumulative idleness of each server to be roughly proportional to its service rate. This property is reflected in the limiting fairness process in~\eqref{eq:totally_blind_fairness}, which possesses a density proportional to $\mu$ with respect to $F$ after $\tau_0^{\LISF}=\tau_0^{\TB}$. 

As a further observation, an interesting consequence of Propositions~\ref{prop:priority_idleness} and~\ref{prop:totally_blind_idleness} is that the limiting fairness processes depend only on the tie-breaking rule among idle servers and are independent of how ties are resolved when there are no idle servers. This observation enables us to characterize the limiting fairness processes using the tools developed for conventional $M/M/n$ queues.

To describe the evolution of the measure-valued state descriptors introduced above, we next define the measure-valued routing processes. To characterize the limiting dynamics of $\xi_i^{\pi,n}$, it is not sufficient to track only the number of arrivals routed to servers with $i$ or more jobs immediately before arrival. We must also record the service-rate distribution of the servers receiving these arrivals. This information is captured by the measure-valued routing processes $\alpha_i^{\pi,n}(t)$ for $i\in \NN$. We first define the counting processes of jobs routed to servers with $i$ or more jobs:
\[
A_i^{\pi,n}(t) := \int_{0}^t \II(Q_{i}^{\pi,n}(s-) = n)\, dA^n(s), \qquad i \in\NN,
\]
with the convention $A_0^{\pi, n}(t):= A^n(t)$. Let $\hat{A}_i^{\pi,n}(t):=n^{-1/2}A_i^{\pi,n}(t)$, and let $\tilde{\mu}_{i,j}^{\pi,n}$ denote the service rate of the server to which the job arriving at the $j$th arrival epoch of $A_i^{\pi,n}(t)$ is routed. We then define  $\alpha_i^{\pi,n}$ as follows. For any $m\in \ZZ_+$,
\[
\alpha_i^{\pi,n}(m)(\AAA) := \sum_{j=1}^{m} \delta_{\tilde{\mu}_{i,j}^{\pi,n}}(\AAA),
\]
which corresponds to the number of the first $m$ jobs routed to servers with $i$ or more jobs that are assigned to servers whose service rates are in $\AAA$. We extend $\alpha_i^{\pi,n}(t)$ to all $t>0$ by linear interpolation between integer times:
\begin{align}\label{eq:alpha_definition}
\alpha_i^{\pi,n}(t)(\AAA)
:= \sum_{j=1}^{\lfloor t\rfloor} \delta_{\tilde{\mu}_{i,j}^{\pi,n}}(\AAA)
+ (t-\lfloor t\rfloor)\delta_{\tilde{\mu}_{i,\lceil t \rceil}^{\pi,n}}(\AAA),
\end{align}
where $\sum_{j=1}^{0} \delta_{\tilde{\mu}_{i,j}^{\pi,n}}(\AAA)$ is understood as the zero measure. We also define the scaled process $\hat{\alpha}_i^{\pi,n}(t)=n^{-1/2}\alpha_i^{\pi,n}(t)$. We use the following Lipschitz continuity and tightness result to prove our main weak convergence result. 

\begin{lemma}\label{lem:alpha_continuity}
For any $i, n\in \NN$, $\alpha_i^{\pi,n} \in \Lip{\cM_f,1}[0,\infty)$. Moreover, the collection of scaled processes $(\hat{\alpha}_1^{\pi,n}(n^{1/2}\cdot):n\in \NN)$ is tight in $\CC_{\cM_f}[0,\infty)$.
\end{lemma}

\begin{proof}
For any $0\leq s\leq t$ and $f\in \Lipbar{\RR, 1}(\SSS)$
\begin{align*}
\left|
\langle f,\alpha_i^{\pi,n}(t)\rangle
-
\langle f,\alpha_i^{\pi,n}(s)\rangle
\right|
&=
\left|
\sum_{j=\lfloor s\rfloor+1}^{\lfloor t\rfloor}
\left\langle f,\delta_{\tilde{\mu}_{i,j}^{\pi,n}}\right\rangle
+
(t-\lfloor t\rfloor)
\left\langle f,\delta_{\tilde{\mu}_{i,\lceil t\rceil}^{\pi,n}}\right\rangle
-
(s-\lfloor s\rfloor)
\left\langle f,\delta_{\tilde{\mu}_{i,\lceil s\rceil}^{\pi,n}}\right\rangle
\right|\leq t-s.
\end{align*}
which shows $\alpha_1^{\pi,n}\in \Lip{\cM_f, 1}[0,\infty)$. To prove the tightness in $\CC_{\cM_f}[0,\infty)$, we need to show that for any $T>0$,  $(\hat{\alpha}_1^{\pi,n}(n^{1/2}\cdot):n\in \NN)$ is tight in $\CC_{\cM_f}[0,T]$, where the total masses of the measures are uniformly bounded by $T$ on $[0,T]$. Hence, using Jakubowski's criteria~\cite{Jakubowski1986} (cf. Theorem~\ref{thm:jakubowski}), we only need to prove that for any bounded continuous function $f\in \CC_{\RR}^b(\SSS)$, $(\langle f, \hat{\alpha}_1^{\pi,n}(n^{1/2}\cdot)\rangle: n\in \NN)$ is tight, i.e., for any $\epsilon>0$,
\begin{enumerate}
    \item  there exists a $K_\epsilon$ such that $\PP(|\langle f, \hat{\alpha}_1^{\pi,n}(n^{1/2}\cdot)\rangle|_T^*>K_\epsilon)<\epsilon$ for all $n\in \NN$
    \item there exists a $\delta>0$ such that $\PP\left(\sup_{\substack{0\leq t-s\leq \delta\\ 0\leq s\leq t\leq T}}n^{-1/2}|\langle f, \alpha_1^{\pi,n}(n^{1/2}t)\rangle - \langle f, \alpha_1^{\pi, n}(n^{1/2}s)\rangle|>\epsilon\right)\leq \epsilon$ for all $n\in \NN$.
\end{enumerate}
As $f$ is a bounded function with a bound $K_f$, the first condition is seen to be satisfied by choosing $K_\epsilon=K_fT$. Using the Lipschitz continuity of $\alpha_1^{\pi, n}$ established above, any $\delta<\epsilon/K_f$ satisfies the second condition. Hence, for every $T>0$, the collection $(\hat{\alpha}_1^{\pi,n}(n^{1/2}\cdot):n\in\NN)$ is tight in $\CC_{\cM_f}[0,T]$. Since $T>0$ was arbitrary, it is tight in $\CC_{\cM_f}[0,\infty)$.
\end{proof}

The balance equations also require us to keep track of departures according to service rate. We therefore introduce the following measure-valued departure processes. We define $D_i^{\pi,n}(t)$ as the number of departures up to time $t$ from servers with $i$ or more jobs, and its scaled version by $\hat{D}_i^{\pi,n}(t) := n^{-1/2} D_i^{\pi,n}(t)$. We further let $\hat{\mu}_{i,j}^{\pi,n}$ denote the service rate of the server completing service at the $j$th departure epoch of $D_i^{\pi,n}(t)$. Then, we define $\Delta_i^{\pi,n}(t)(\AAA)$ as the number of departures up to time $t$ from servers whose rates are in $\AAA$ when these servers have $i$ or more jobs, that is,
\begin{align}
    \Delta_i^{\pi,n}(t)(\AAA) := \sum_{j=1}^{D_i^{\pi,n}(t)} \delta_{\hat{\mu}_{i,j}^{\pi,n}}(\AAA).
    \label{eq:delta_definition}
\end{align}
We also define the diffusion-scaled departure measure $\hat{\Delta}_i^{\pi,n}(t) := n^{-1/2}\Delta_i^{\pi,n}(t)$ for $t \ge 0$ and $i \in \ZZ_+$. We emphasize that the processes $\Delta_i^{\pi,n}(t)$ are  right continuous, whereas the processes $\alpha_i^{\pi,n}(t)$ are Lipschitz continuous, as established in Lemma~\ref{lem:alpha_continuity}.

It is important to stress that the processes $\alpha_i^{\pi,n}$ are not indexed by the natural time of the system. Instead, they are initially indexed by the number of jobs routed to servers with queue length $i$ and are then converted into continuous-time processes by interpolation. This contrasts with all other processes considered in this work. For example, $\Delta_i^{\pi,n}$ is indexed by natural time, since the sum on the right-hand side of \eqref{eq:delta_definition} is taken over all departures that occur up to time $t$. This plays an important role in the argument for the existence and uniqueness of the fixed point in the proof of Theorem~\ref{thm:process_level_convergence}.

We now have all the ingredients needed to formulate the evolution equations. Based on the definitions above, we can write the following balance equations for the scaled total number of jobs at servers with $i$ or more jobs, $X_i^{\pi,n}(t)$, as:
\begin{align}\label{eq:aggregate_balance}
    X_i^{\pi,n}(t) &= X_i^{\pi,n}(0) + \left( \hat{A}_{i-1}^{\pi,n}(t) - \hat{A}_{i}^{\pi,n}(t) \right)- \left(\hat{D}_{i}^{\pi,n}(t) -\hat{D}_{i+1}^{\pi,n}(t)\right), \mbox{ for all }i\in \NN. 
\end{align}
In a similar fashion, we can also obtain the following balance equations in the weak form for the measure-valued processes $\xi_i^{\pi, n}(t)$:
\begin{align}\label{eq:measure_balance}
    \langle f, \xi_i^{\pi,n}(t)\rangle &= \langle f,\xi_i^{\pi,n}(0)\rangle + \langle f,\hat{\alpha}_{i-1}^{\pi,n}(A_{i-1}^{\pi,n}(t))\rangle-\langle f,\hat{\alpha}_{i}^{\pi,n}(A_{i}^{\pi,n}(t))\rangle-\langle f, \hat{\Delta}_{i}^{\pi,n}(t)\rangle + \langle f, \hat{\Delta}_{i+1}^{\pi,n}(t)\rangle
\end{align}
for all $i\geq 2$ and $f\in \Lipbar{\RR,L}(\SSS)$ for all $L\geq 0$.

The preceding discussion focuses on the transient evolution of load-balancing systems with finitely many servers. Since we also address stationary behavior, we conclude this section by establishing the positive recurrence of pre-limit load-balancing systems when there is excess capacity, i.e., $\varsigma^n + \nu^n>0$. We provide the proof in Appendix~\ref{app:positive_recurrent_finite}.

\begin{lemma}\label{lem:positive_recurrent_finite}
    Suppose $\nu^n+\varsigma^n>0$. Then, under any tie-breaking policy $\pi$ that makes $(\mf{Q}^{\pi,n}(t):t\geq 0)$ a Markov process, the JSQ system is positive recurrent.
\end{lemma}

\section{Main Results}\label{sec:main_results}

In this section, we present our main results, postponing the proofs to later sections. Our first result provides a general framework for deriving diffusion limits under arbitrary tie-breaking rules. Once the limiting fairness process $\eta^{\pi}$ and the weak limit of the routing processes $\alpha_1^{\pi}$ have been identified, the diffusion limit follows as the unique solution to a reflected stochastic differential equation coupled with a family of measure-valued integral equations. We first state this general result in Theorem~\ref{thm:process_level_convergence} and then illustrate its use by deriving diffusion limits for several specific tie-breaking policies. For all our process-level convergence results, we define $W(t)$ to be a standard Brownian motion.

\begin{theorem}\label{thm:process_level_convergence}
Suppose that under tie-breaking policy $\pi$, $\eta^{\pi, n}\tof \eta^{\pi}$ and $\hat{\alpha}_1^{\pi, n}(n^{1/2}\cdot)\Rightarrow \alpha_1^{\pi}$ jointly. Then $(X_1^{\pi, n}, \boldsymbol{\xi}^{\pi, n})\Rightarrow (X_1^{\pi},\boldsymbol{\xi}^{\pi}),$ where $(X_1^{\pi},\boldsymbol{\xi}^{\pi})$ is the unique solution of
\begin{align}
    \label{eq:main_diffusion_x1} X_1^{\pi}(t) & = X_1(0) - (\varsigma+ \nu) t + \sqrt{2}W(t) - \langle \iota, \eta^{\pi}(t)\rangle\int_0^t X_1^{\pi}(s)ds + \int_0^t \langle \iota, \xi_2^{\pi}(s)\rangle ds - U_1^{\pi}(t),\\
    \label{eq:main_diffusion_xi2}
    \langle f, \xi_2^{\pi}(t)\rangle & = \langle f, \xi_2(0)\rangle + \langle f, \alpha_1^{\pi}(U_1^{\pi}(t))\rangle - \int_0^t\langle f\times \iota, \xi_2^{\pi}(s)\rangle ds + \int_0^t \langle f\times\iota, \xi_3^{\pi}(s)\rangle ds ,\\
    \label{eq:main_diffusion_xi3}
     \langle f, \xi_i^{\pi}(t)\rangle & = \langle f, \xi_i(0)\rangle - \int_0^t\langle f\times \iota, \xi_i^{\pi}(s)\rangle ds +  \int_0^t \langle f\times\iota, \xi_{i+1}^{\pi}(s)\rangle ds,\mbox{ for all }i\geq 3, 
\end{align}
for any fixed $L>0$ and all $f\in\Lipbar{\RR,L}(\SSS)$, where $W$ is a standard Brownian motion defined on the same probability space as $\eta^\pi$ and $\alpha_1^\pi$, and may depend on these processes. Moreover, $U_1^\pi\in\DD_{\RR_+}[0,\infty)$ is a non-decreasing function with $U_1^\pi(0)=0$ and satisfies
\begin{align}\label{eq:main_reflection}
\int_0^\infty \II\bigl(X_1^\pi(s)<0\bigr) dU_1^\pi(s)=0.
\end{align}
\end{theorem}

An important feature of Theorem~\ref{thm:process_level_convergence} is that it combines an aggregate description for the idle-server process with a measure-valued description for the higher queue levels. The reason for this distinction lies in the different time scales on which these processes evolve. We see that $X_{1}^{\pi,n}(t)$ changes either when an arrival is routed to an idle server or when there is a departure from one of the servers with only one job. Both of these events occur at an $O(n)$ rate, implying that the distribution of idle servers changes at the same order. Hence, we need to use cumulative idleness, as in the definition of $\eta^{\pi,n}(t)$, to achieve the desired smoothness in the limit. On the other hand, this rapid evolution of the distribution allows us to obtain limiting fairness processes that admit simple expressions, as seen in Propositions~\ref{prop:priority_idleness} and \ref{prop:totally_blind_idleness}. In contrast, the events that change the state of servers with two or more jobs occur at an $O(n^{1/2})$ rate, and hence, we need to work with the evolution equations of the instantaneous distribution of jobs.

In contrast to the many-server systems with a single queue, Theorem~\ref{thm:process_level_convergence} has an interesting consequence on the fairness process. As initial condition implies that $n-Q_2^{n,\pi}(0)=O(n)$ servers with one or zero jobs, the singularity point $\tau_0^{\pi}$ for the limiting fairness process is observed at time 0 almost surely.

\begin{corollary}\label{cor:tau_0_continuity}
    Under any tie-breaking policy $\pi$, $\PP(\tau_0^{\pi}=0) =1$.
\end{corollary}

\begin{remark} 
The theorem does not impose any independence assumption between the Brownian motion $W$, the limiting fairness process $\eta^\pi$, and the limiting routing process $\alpha^\pi_1$. If $\eta^\pi$ and $\alpha^\pi_1$ are random, care must be taken to establish the joint convergence of these processes with the stochastic fluctuations that give rise to $W$. In all the examples considered below, $\eta^\pi$ and $\alpha^\pi_1$ are deterministic, and hence this joint convergence issue does not arise.\end{remark}

As a first application of the general framework established in Theorem~\ref{thm:process_level_convergence}, we consider the join-the-fastest-shortest-queue (JFSQ) policy, also referred to as the speed-aware JSQ system. JFSQ was studied by Bhambay et al.~\cite{bhambay2025asymptotic} in the context of pool-level heterogeneity. Corollary~\ref{cor:SA-JSQ} extends their diffusion limit to the setting of individual server-level heterogeneity specified in Assumption~\ref{asm:arrival_service_rates}, illustrating one of the main advantages of the measure-valued framework developed in this paper.



\begin{corollary}\label{cor:SA-JSQ}
    For the JFSQ system, $\eta^{\JFSQ,n}\tof \eta^{\JFSQ}$ as given in Proposition~\ref{prop:priority_idleness} and $\hat{\alpha}_1^{\JFSQ,n}(n^{1/2}\cdot)\Rightarrow \iota(\cdot)\times \delta_{\mu_{\max}}$. Moreover, if $\xi_i(0) = \delta_{\mu_{\max}}X_i(0)$ for all $i \geq 2$, then $\mf{X}^{\JFSQ, n} \Rightarrow \mf{X}^{\JFSQ}$, where $\mf{X}^{\JFSQ}$ is the unique solution of
    \begin{align}
        \nonumber X_1^{\JFSQ}(t) & = X_1(0) - (\varsigma + \nu)t + \sqrt{2} W(t)\\
        &\qquad\qquad\qquad- \mu_{\min} \int_0^t X_1^{\JFSQ}(s) \, ds + \mu_{\max} \int_0^t X_2^{\JFSQ}(s) \, ds - U_1^{\JFSQ}(t), \\
        X_2^{\JFSQ}(t) & = X_2(0) + U_1^{\JFSQ}(t) - \mu_{\max} \int_0^t X_2^{\JFSQ}(s) \, ds + \mu_{\max} \int_0^t X_3^{\JFSQ}(s) \, ds, \\
        X_i^{\JFSQ}(t) & = X_i(0) - \mu_{\max} \int_0^t X_i^{\JFSQ}(s) \, ds + \mu_{\max} \int_0^t X_{i+1}^{\JFSQ}(s) \, ds, \quad \text{for all } i \geq 3,
    \end{align}
    where $U_1^{\JFSQ}\in \DD_{\RR_+}[0,\infty)$ is a non-decreasing function with $U_1^{\JFSQ}(0)=0$ and satisfies 
    \begin{align*}
        \int_0^\infty \II(X_1^{\JFSQ}(s)<0)\,dU_1^{\JFSQ}(s) = 0.
    \end{align*}
\end{corollary}

As our second application of Theorem~\ref{thm:process_level_convergence}, we consider the random-JSQ system, where ties are broken uniformly at random. Unlike the JFSQ policy, whose limiting dynamics collapse to the fastest service rate, the random-JSQ policy retains the full service-rate distribution in the limiting state descriptor. Hence, the limiting dynamics cannot be described solely through aggregate queue-length processes and instead require the full measure-valued framework developed in this paper.

\begin{corollary}\label{cor:JSQ_random}
    For the random-JSQ system, $\eta^{\RJSQ, n}\tof \eta^{\TB}$ as given in Proposition~\ref{prop:totally_blind_idleness} and $\hat{\alpha}_1^{\RJSQ, n}(n^{1/2}\cdot)\Rightarrow \iota(\cdot)\times F$. Consequently, $(X_1^{\RJSQ, n},\boldsymbol{\xi}^{\RJSQ,n})\Rightarrow (X_1^{\RJSQ}, \boldsymbol{\xi}^{\RJSQ})$, where $(X_1^{\RJSQ}, \boldsymbol{\xi}^{\RJSQ})$ is the unique solution of 
    \begin{align}
    \nonumber X_1^{\RJSQ}(t) & = X_1(0) - (\varsigma+ \nu) t + \sqrt{2}W(t) \\
    &\qquad\qquad\quad- \frac{1+\sigma_F^2}{\bar{\mu}}\int_0^t X_1^{\RJSQ}(s)ds + \int_0^t \langle \iota, \xi_2^{\RJSQ}(s)\rangle ds - U_1^{\RJSQ}(t), \label{eq:JSQ_random1}\\
    \langle f, \xi_2^{\RJSQ}(t)\rangle & = \langle f, \xi_2(0)\rangle + \langle f, F\rangle U_1^{\RJSQ}(t) - \int_0^t\langle f\times \iota, \xi_2^{\RJSQ}(s)\rangle ds + \int_0^t \langle f\times\iota, \xi_3^{\RJSQ}(s)\rangle ds, \label{eq:JSQ_random2}\\
    \langle f, \xi_i^{\RJSQ}(t)\rangle & = \langle f, \xi_i(0)\rangle - \int_0^t\langle f\times \iota, \xi_i^{\RJSQ}(s)\rangle ds +  \int_0^t \langle f\times\iota, \xi_{i+1}^{\RJSQ}(s)\rangle ds,\mbox{ for all }i\geq 3, \label{eq:JSQ_random3}
\end{align}
for any fixed $L>0$ and all $f\in \Lipbar{\RR, L}(\SSS)$. Moreover $U_1^{\RJSQ}\in \DD_{\RR_+}[0,\infty)$ is a non-decreasing function with $U_1^{\RJSQ}(0)=0$ and satisfies
\begin{align*}
    \int_0^\infty \II(X_1^{\RJSQ}(s)<0)dU_1^{\RJSQ}(s) = 0.
\end{align*}
\end{corollary}

When the limiting service rate distribution is finite discrete, we can express the measure-valued integral equation in Corollary~\ref{cor:JSQ_random} as classical integral equations. A special case is the JSQ-system with pool-level heterogeneity where the ties are broken uniformly at random among all servers with the shortest queue regardless of their pool.

\begin{corollary}\label{cor:JSQ_random_pool} Suppose there exist $(\mu_1, \ldots, \mu_m)$ such that the limiting service rate distribution and the limiting initial distribution can be expressed as $F(\AAA) = \sum_{j=1}^m \gamma_j \delta_{\mu_j}(\AAA)$, and $\xi_i(0)(\AAA) = \sum_{j=1}^m X_{j,i}(0) \delta_{\mu_j}(\AAA)$, respectively, for all $\AAA\in \cB$ and $i\geq 2
$. For all $1\leq j\leq m$, let $\AAA_j$ be an open interval such that $\mu_j\in \AAA_j$ and $\mu_{j'}\notin\AAA_j$ for $j'\neq j$, and define $X_{j,i}^{\RJSQ,n}(t) = n^{-1/2}\sum_{k=1}^n\delta_{\mu_k^n}(\AAA_j)Q_{k,i}^{\RJSQ, n}(t)$. Then, for the random-JSQ system $(X_1^{\RJSQ, n}, X_{j,i}^{\RJSQ,n}:i\geq 2, 1\leq j\leq m)\Rightarrow (X_1^{\RJSQ}, X_{j,i}^{\RJSQ}:i\geq 2, 1\leq j\leq m)$ which is the unique solution of 
    \begin{align}
    \label{eq:JSQ_random_pool1}
    \nonumber X_1^{\RJSQ}(t) & = X_1(0) - (\varsigma+ \nu) t + \sqrt{2}W(t) \\
    &\qquad\qquad\quad- \frac{1+\sigma_F^2}{\bar{\mu}}\int_0^t X_1^{\RJSQ}(s)ds + \sum_{j=1}^m\mu_j\int_0^t X_{j,2}^{\RJSQ}(s) ds - U_1^{\RJSQ}(t),\\
    \label{eq:JSQ_random_pool2}X_{j,2}^{\RJSQ}(t) & = X_{j,2}(0) + \gamma_j U_1^{\RJSQ}(t) - \mu_j\int_0^t X_{j,2}^{\RJSQ}(s) ds + \mu_j\int_0^t X_{j,3}^{\RJSQ}(s)ds ,\mbox{ for all }1\leq j\leq m,\\
     \label{eq:JSQ_random_pool3} X_{j,i}^{\RJSQ}(t) & = X_{j,i}(0)  - \mu_j\int_0^t X_{j,i}^{\RJSQ}(s) ds + \mu_j\int_0^t X_{j,i+1}^{\RJSQ}(s)ds ,\mbox{ for all }1\leq j\leq m, i\geq 3,
    \end{align}  
where $U_1^{\RJSQ}\in \DD_{\RR_+}[0,\infty)$ is a non-decreasing function  with $U_1^{\RJSQ}(0)=0$ and satisfies
\begin{align*}
    \int_0^\infty \II(X_1^{\RJSQ}(s)<0)\,dU_1^{\RJSQ}(s) = 0.
\end{align*}    

\end{corollary}

A notable consequence of Theorem~\ref{thm:process_level_convergence} is that, in general, the limiting fairness process $\eta$ and the limiting routing process $\alpha_1$ decouple. This allows diffusion limits to be derived in a modular mix-and-match fashion when one tie-breaking rule is used while servers are idle and another is used when all servers are busy. Corollaries~\ref{cor:JSQ_random} and \ref{cor:JSQ_random_pool} therefore also hold when ties among idle servers are broken according to the LISF policy. As an example, Corollary~\ref{cor:JFIQ} presents the diffusion limit under the join-the-fastest-idle-server (JFIQ) policy, where jobs are routed to the fastest idle server when idle servers are available and ties are broken uniformly at random when all servers are busy. 

\begin{corollary}\label{cor:JFIQ}
    For the JFIQ system, $\eta^{\JFIQ, n}\tof \eta^{\JFSQ}$ as given in Proposition~\ref{prop:priority_idleness} and $\hat{\alpha}_1^{\JFIQ, n}(n^{1/2}\cdot)\Rightarrow \iota(\cdot)\times F$. Consequently, $(X_1^{\JFIQ, n},\boldsymbol{\xi}^{\JFIQ, n})\Rightarrow (X_1^{\JFIQ}, \boldsymbol{\xi}^{\JFIQ})$, where $(X_1^{\JFIQ}, \boldsymbol{\xi}^{\JFIQ})$ is the unique solution of 
    \begin{align}
    \label{eq:JSQ_JFIQ1}
    \nonumber
    X_1^{\JFIQ}(t) & = X_1(0) - (\varsigma+ \nu) t + \sqrt{2}W(t) \\
    &\qquad\qquad\qquad- \mu_{\min}\int_0^t X_1^{\JFIQ}(s)ds + \int_0^t \langle \iota, \xi_2^{\JFIQ}(s)\rangle ds - U_1^{\JFIQ}(t),\\
    \label{eq:JSQ_JFIQ2}
    \nonumber \langle f, \xi_2^{\JFIQ}(t)\rangle & = \langle f, \xi_2(0)\rangle + \langle f, F\rangle U_1^{\JFIQ}(t) \\
    &\qquad \qquad \qquad \qquad- \int_0^t\langle f\times \iota, \xi_2^{\JFIQ}(s)\rangle ds + \int_0^t \langle f\times\iota, \xi_3^{\JFIQ}(s)\rangle ds ,\\
     \label{eq:JSQ_JFIQ3}\langle f, \xi_i^{\JFIQ}(t)\rangle & = \langle f, \xi_i(0)\rangle - \int_0^t\langle f\times \iota, \xi_i^{\JFIQ}(s)\rangle ds +  \int_0^t \langle f\times\iota, \xi_{i+1}^{\JFIQ}(s)\rangle ds,\mbox{ for all }i\geq 3, 
\end{align}
for any fixed $L>0$ and all $f\in \Lipbar{\RR, L}(\SSS)$. Moreover $U_1^{\JFIQ}\in \DD_{\RR_+}[0,\infty)$ is a non-decreasing function with $U_1^{\JFIQ}(0)=0$ and satisfies
\begin{align*}
    \int_0^\infty \II(X_1^{\JFIQ}(s)<0)\,dU_1^{\JFIQ}(s) = 0.
\end{align*}
\end{corollary}

\begin{figure}[t]
    \centering
    \begin{subfigure}[t]{0.49\linewidth}
        \centering
        \includegraphics[height=5.8cm]{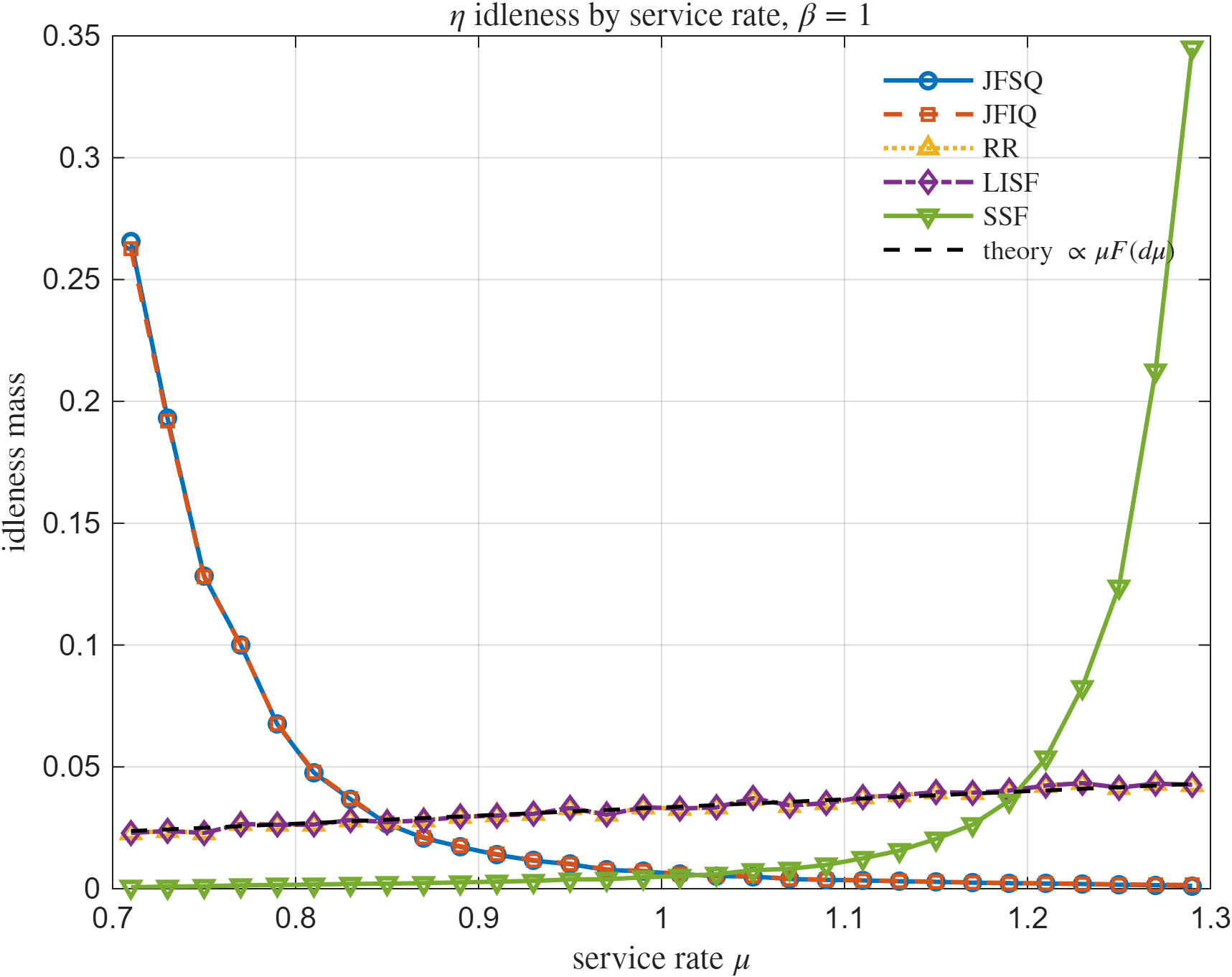}
        \caption{Limiting fairness measures.}
        \label{fig:eta_profile}
    \end{subfigure}
    \hfill
    \begin{subfigure}[t]{0.49\linewidth}
        \centering
        \includegraphics[height=5.8cm]{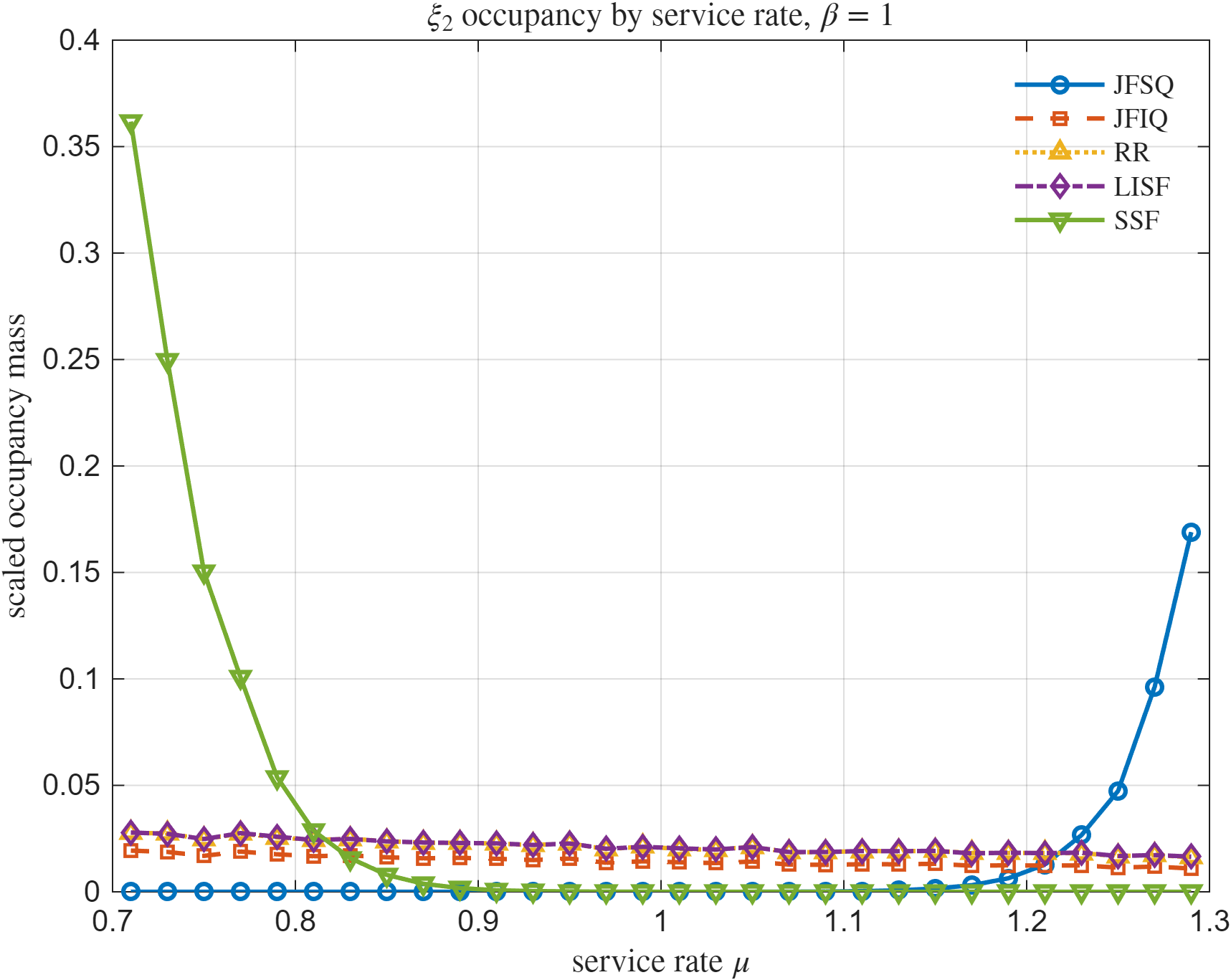}
        \caption{Limiting $\xi_2$ measures.}
        \label{fig:xi2_profile}
    \end{subfigure}
    \caption{Empirical measure-valued profiles for $\varsigma+\nu=1.00$.}
    \label{fig:measure_valued_profiles}
\end{figure}

Figure~\ref{fig:measure_valued_profiles} illustrates the limiting fairness measures and the stationary occupancy measure $\xi_2(\infty)$  across service rates at $\varsigma + \nu =1.00$. The simulations are based on $n=500$ servers whose service rates are sampled i.i.d. from the uniform distribution on ($[1-\epsilon, 1+\epsilon]$) with $\epsilon= 0.3$. The simulation horizon is $T= 3000$, with a warm-up period of $T_{\mathrm{wu}}=300$, and all reported averages are exact time averages over the post warm-up interval.

Figure~\ref{fig:eta_profile} shows the post-warm-up fairness measure $\eta$. JFSQ and JFIQ have almost identical idleness profiles, with idleness concentrated on the slowest servers. This is expected because both policies use the same idle-server rule which selects the fastest idle server. Conversely, JSSQ leaves more idleness on fast servers. RR and LISF are close to the totally blind benchmark, whose limiting idleness profile is proportional to $\mu dF(\mu)$. 

On the other hand, Figure~\ref{fig:xi2_profile} illustrates the 
time-averaged scaled occupancy profile $\xi_2$, which records the mass of servers with queues of length at least two across service-rate bins. Under JFSQ, this mass is concentrated near the fastest servers, while under JSSQ it is concentrated near the slowest servers. The JFIQ, RR and LISF profiles are more spread out, since these policies use random tie-breaking when all servers are busy.

Theorem~\ref{thm:process_level_convergence} establishes process-level approximations for heterogeneous load-balancing systems under general tie-breaking rules. The next question is whether these diffusion limits also characterize the corresponding steady-state behavior. A load-balancing system is positive recurrent only when there is excess capacity, and is null recurrent or transient otherwise. Since the service-rate vector may be random, the pre-limit system is positive recurrent only on service-rate environments with positive excess capacity. In the stationary results below, all stationary quantities are therefore considered conditionally on the event
$\{\varsigma^n+\nu^n>\varepsilon\}$.
Assuming such uniform excess capacity, we next establish uniform bounds on the steady-state number of idle servers and on the number of queues with $i$ or more jobs. These bounds justify the interchange of the many-server and steady-state limits, thereby connecting the process-level convergence results with steady-state performance approximations. Theorem~\ref{thm:stationary_bounds} not only generalizes Theorem~2 in \cite{bhambay2025asymptotic}, originally proved for JFSQ, to arbitrary tie-breaking policies, but also improves the bound on $\EE[X_i^{\pi,n}(\infty)]$ for $i\geq 3$ to $O(n^{-1/2})$. 

\begin{theorem}\label{thm:stationary_bounds}
Suppose that each server in the system has a finite buffer of length $b$ and Assumption~\ref{asm:arrival_service_rates} holds. Also assume that for $\varepsilon>0$ there exists $\varrho>0$ such that $\PP(\varsigma^n+\nu^n>\varepsilon)>\varrho$ for all $n\in \NN$. Then there exist constants $C_{1,\varepsilon},C_{2,\varepsilon},C_{3,\varepsilon}$, uniform in $n$ and over all tie-breaking policies where a stationary version exists, such that
\begin{align}
\label{eq:stationary_bound1}
\EE\left[|X_1^{\pi,n}(\infty)|\mid \varsigma^n+\nu^n>\varepsilon\right]&<C_{1,\varepsilon},\\
\label{eq:stationary_bound2}
\EE\left[|X_2^{\pi,n}(\infty)|\mid \varsigma^n+\nu^n>\varepsilon\right]&<C_{2,\varepsilon},\\
\label{eq:stationary_bound3}
\EE\left[|X_3^{\pi,n}(\infty)|\mid \varsigma^n+\nu^n>\varepsilon\right]&<C_{3,\varepsilon}n^{-1/2},
\end{align}
for all $n\in\NN$.
\end{theorem}

The bounds provided in Theorem~\ref{thm:stationary_bounds} do not naturally yield the interchangeability of limits. Our construction allows general tie-breaking policies, under which $\mf{Q}$ need not even be Markov. The next result establishes the tightness result required to prove the interchangeability of limits for specific policies.

\begin{corollary}\label{cor:interchange_limits}
  Suppose that the assumptions of Theorem~\ref{thm:stationary_bounds} hold. 
  Under any tie-breaking rule $\pi$ that yields a Markov 
  $(\mf{Q}^n(t), t\geq 0)$ for all $n$ and a Markov diffusion limit 
  \eqref{eq:main_diffusion_x1}--\eqref{eq:main_reflection}, the following hold 
  conditional on the event $\{\varsigma^n+\nu^n>\varepsilon\}$ where $\varepsilon>0$ is such that $\PP(\varsigma + \nu=\varepsilon) = 0$:
  \begin{enumerate}
      \item The sequence $((X_1^{\pi,n}(\infty),\mf{\xi}^{\pi,n}(\infty)),\, n\in\NN)$ is tight. 
      Moreover, if along a subsequence $n_k\to\infty$, 
      $(X_1^{\pi,n_k}(\infty),\mf{\xi}^{\pi,n_k}(\infty))
      \Rightarrow (X_1^\pi(0),\mf{\xi}^\pi(0))$, then the solution 
      $(X_1^\pi(t),\mf{\xi}^\pi(t))$ of 
      \eqref{eq:main_diffusion_x1}--\eqref{eq:main_reflection} with initial 
      condition $(X_1^\pi(0),\mf{\xi}^\pi(0))$ is stationary. In particular, 
      the law of $(X_1^\pi(0),\mf{\xi}^\pi(0))$ is an invariant distribution 
      of the limiting diffusion.
      \item $\xi_i^{\pi,n}(\infty)\toP 0$ for all $i\geq 3$.
  \end{enumerate}
  Hence, if the limiting diffusion 
  \eqref{eq:main_diffusion_x1}--\eqref{eq:main_reflection} has a unique 
  stationary distribution, then the stationary distributions of the pre-limit 
  systems converge to this distribution, and the steady-state and many-server 
  limits can be interchanged.
\end{corollary}

Our final result builds on the preceding transient and steady-state analysis to establish the asymptotic optimality of the JFSQ policy under individual server heterogeneity. This extends the corresponding result established for pool-level heterogeneity to the more general setting studied in this paper. We define
\begin{align}\label{eq:cumulative_Q}
Q_{+i}^{\pi,n}(t) = \sum_{j=i}^\infty Q_{j}^{\pi,n}(t) \mbox{ and }X_{+i}^{\pi,n}(t) = \sum_{j=i}^\infty X_{j}^{\pi,n}(t) \mbox{ for all $i,n\in \NN$ and $t\geq 0$}.
\end{align}
Specifically, $Q_{+1}^{\pi,n}(t)$ and $Q_{+2}^{\pi,n}(t)$ denote the total number of jobs in the system and the total number of jobs waiting in the queue, respectively, in the $n$th system at time $t\geq 0$.

\begin{theorem}\label{thm:asymptotical_optimality}
    Suppose that the assumptions of Theorem~\ref{thm:stationary_bounds} hold. Then, for any tie-breaking policy $\pi$ where the stationary  $(X_1^{\pi,n}(\infty), \boldsymbol{\xi}^{\pi,n}(\infty))$ exists for all $n\in \NN$, and for all $y>0$, we have 
    \begin{align}
       \label{eq:asymptotic_optimality1} \lim_{n\to \infty} \PP(X_{+1}^{\JFSQ, n}(\infty)>y\mid \varsigma^n+\nu^n>\varepsilon)&\leq \liminf_{n\to \infty} \PP(X_{+1}^{\pi, n}(\infty)>y\mid \varsigma^n+\nu^n>\varepsilon),\\
        \label{eq:asymptotic_optimality2}
        \lim_{n\to \infty} \PP(X_{+2}^{\JFSQ, n}(\infty)>y\mid \varsigma^n+\nu^n>\varepsilon)&\leq \liminf_{n\to \infty} \PP(X_{+2}^{\pi, n}(\infty)>y\mid \varsigma^n+\nu^n>\varepsilon).
    \end{align}
      conditional on the event $\{\varsigma^n+\nu^n>\varepsilon\}$ where $\varepsilon>0$ is such that $\PP(\varsigma + \nu=\varepsilon) = 0$.
\end{theorem}

Theorem~\ref{thm:asymptotical_optimality} complements the results of \cite{LiuYing2025} in the Halfin--Whitt regime. Liu and Ying~\cite{LiuYing2025} showed that JFSQ is asymptotically optimal with respect to the probability of waiting, which is implied by \eqref{eq:asymptotic_optimality1}. In addition, Theorem~\ref{thm:asymptotical_optimality} shows that JFSQ asymptotically minimizes the steady-state number of jobs in the system and the steady-state number of jobs waiting in the queue in the sense of stochastic ordering.

\begin{figure}[!ht]
    \centering
    \begin{subfigure}{\linewidth}
        \centering
        \includegraphics[width=\linewidth]{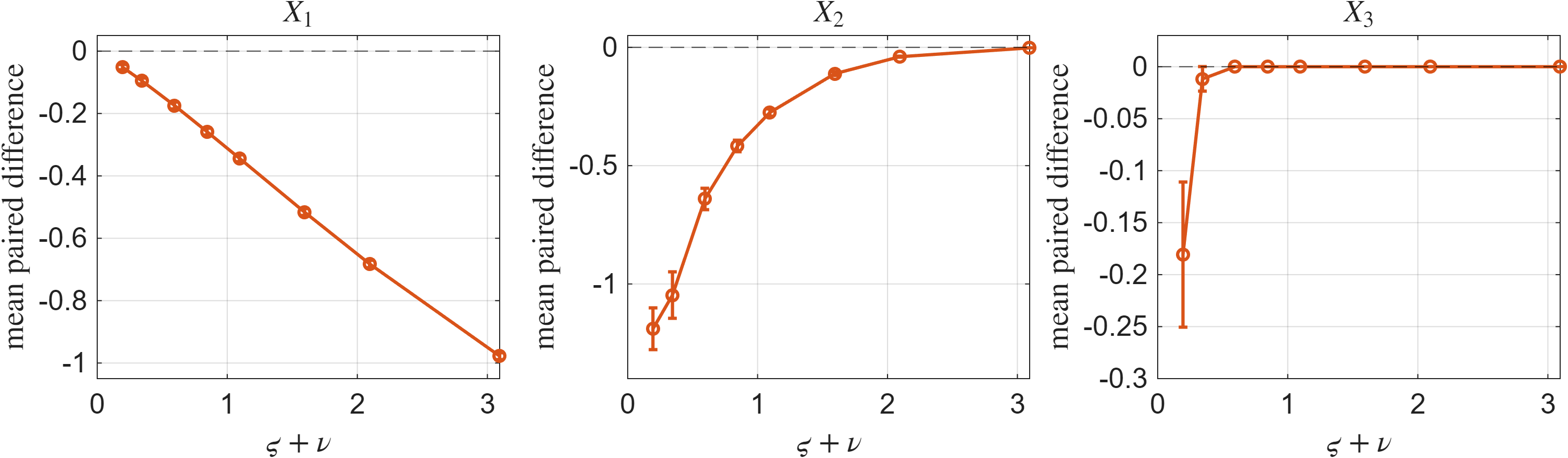}
        \caption{$\EE[|X_i^{\mathrm{JFSQ}}(\infty)|]-\EE[|X_i^{\mathrm{RR}}(\infty)|]$ for $i=1,2,3$.}
    \end{subfigure}

    \vspace{0.8em}

    \begin{subfigure}{\linewidth}
        \centering
        \includegraphics[width=0.72\linewidth]{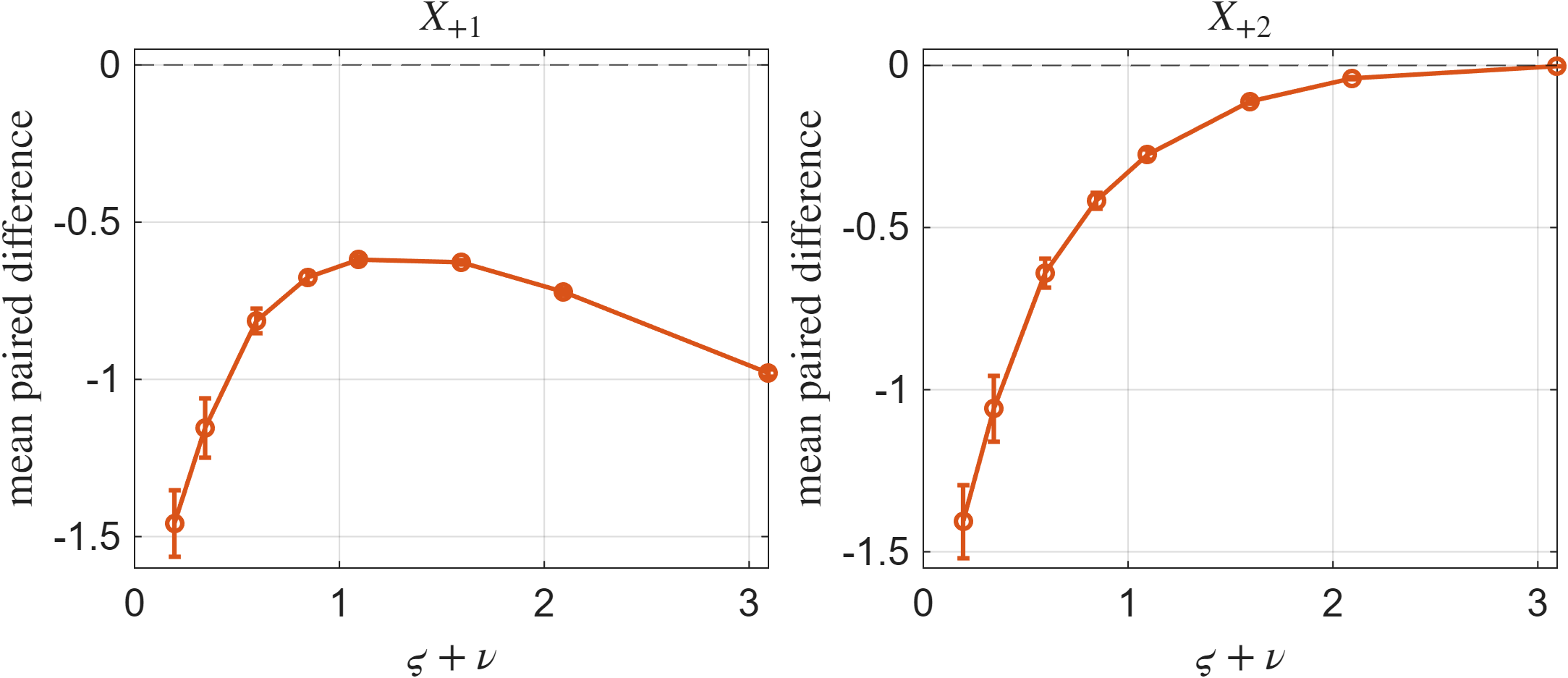}
        \caption{$\EE[|X_{+i}^{\mathrm{JFSQ}}(\infty)|]-\EE[|X_{+i}^{\mathrm{RR}}(\infty)|]$ for $i=1,2$.}
    \end{subfigure}

    \caption{Comparison of JFSQ and RR policies.}
    \label{fig:exact_JFSQ_minus_RR}
\end{figure}

Figure~\ref{fig:exact_JFSQ_minus_RR} shows the paired differences between the JFSQ and RR policies, estimated using simulation with common random variates to approximate steady-state behavior. Across all values of the excess capacity parameter $\varsigma+\nu$, the JFSQ policy yields a larger average number of idle servers, smaller average values of $X_2$ and $X_3$, a smaller average system size ($X_{+1}$), and a smaller average total queue length ($X_{+2}$) than the RR policy. The advantage of JFSQ in terms of the number of idle servers increases as the excess capacity increases, whereas its advantage in terms of $X_2$, $X_3$, and the average total queue length decreases. The behavior of $X_{+1}$ is slightly different. Since $X_{+1}=X_1+X_{+2}$, the paired difference in $X_{+1}$ reflects two competing effects: differences in idleness and differences in waiting. For small values of $\varsigma+\nu$, the waiting component dominates, whereas for larger values of $\varsigma+\nu$, $X_{+2}$ becomes negligible and the idleness component becomes more pronounced. This explains the non-monotone shape of the paired-difference curve for $X_{+1}$.

\section{Auxiliary Systems for Performance Bounds}
\label{sec:auxiliary_systems}
In this section, we discuss two auxiliary systems to provide upper and lower bounds for performance measures of the original load-balancing system. The first auxiliary system, which we refer to as the JSQ-fastest-serves-the-longest (JSQ-FSL), is introduced in \cite{bhambay2025asymptotic} to provide lower bounds on the system size and the total number of jobs in the queue in a JSQ system with pool-level heterogeneity under any tie-breaking policy. 

The JSQ-FSL system corresponding to the $n$th load-balancing system receives external arrivals according to the same Poisson process, $A^n(t)$, as the original system and consists of the same $n$ servers where the $k$th server serves with rate $\mu_k^n$. As in the original system, the arriving jobs are routed to one of the shortest buffers in the system upon arrival according to a pre-specified tie-breaking rule. The main difference between the JSQ-FSL system and the original system is that, in the JSQ-FSL system, servers and buffers are decoupled. After each arrival or departure, servers are allowed to move flexibly between buffers in a preemptive manner so that faster servers always serve longer queues. Specifically, if $\mu_k^n < \mu_{k'}^n$, then the length of the buffer served by server $k$ is \emph{less than or equal} to the length of the buffer served by server $k'$ for all $t \ge 0$. Hence, the behavior of the JSQ-FSL system does not depend on the tie-breaking rule for routing arrivals. Unlike \cite{bhambay2025asymptotic}, which restricts the buffer capacity to 2, we allow arbitrary buffer capacities. Throughout this section, each buffer in both the original JSQ system and the corresponding JSQ-FSL system has capacity $1 \leq b \leq \infty$. In a similar fashion, we also define the JSQ-slowest-serves-the-longest (JSQ-SSL) system, where after each arrival and departure, the servers arrange themselves so that slower servers always serve longer buffers, i.e., if $\mu_k^n < \mu_{k'}^n$, then the length of the buffer served by server $k$ is \emph{greater than or equal} to the length of the buffer served by server $k'$ for all $t \ge 0$.

Let $(\tilde{\mf{Q}}_b^{\mathrm{FSL},n}(t) : t \ge 0)$ and $(\tilde{\mf{Q}}_b^{\mathrm{SSL},n}(t) : t \ge 0)$ represent the state of the $n$th JSQ-FSL and JSQ-SSL systems with buffer capacity $b$, respectively, where
$\tilde{\mf{Q}}_b^{\mathrm{FSL},n}(t) = (\tilde{Q}_{b,i}^{\mathrm{FSL},n}(t) : i \in \mathbb{N})$,
and $\tilde{Q}_{b,i}^{\mathrm{FSL},n}(t)$ denotes the number of servers that serve buffers with $i$ or more jobs at time $t$. We assume that the JSQ-FSL system starts from the same initial condition as the original system, i.e., $\tilde{\mf{Q}}_b^{\FSL,n}(0) = \mf{Q}^{\pi, n}(0)$. The processes $\tilde{\mf{Q}}_b^{\mathrm{SSL},n}(t)$ and $\tilde{Q}_{b,i}^{\mathrm{SSL},n}(t)$ are defined analogously. As in \eqref{eq:cumulative_Q}, we also introduce 
\[
\tilde{Q}_{b,+i}^{\mathrm{FSL},n}(t) := \sum_{j=i}^\infty \tilde{Q}_{b,j}^{\mathrm{FSL},n}(t), \mbox{ and }
\tilde{Q}_{b,+i}^{\mathrm{SSL},n}(t) := \sum_{j=i}^\infty \tilde{Q}_{b,j}^{\mathrm{SSL},n}(t),
\]
to refer to the number of jobs that are in positions $i$ or higher in their buffers, i.e., for all $i\geq 1$. The next proposition establishes that the JSQ-FSL and JSQ-SSL systems act as stochastic lower and upper bounding systems, respectively, for the original JSQ system.

\begin{proposition}
\label{prop:upper_bound_system}
Under any tie-breaking rule $\pi$ for the JSQ system, where each buffer has capacity $1\leq b\leq \infty$, we have 
\[
\tilde{Q}_{b,+i}^{\FSL,n}(t)\leq_{st}Q_{+i}^{\pi,n}(t) \leq_{st} \tilde{Q}_{b,+i}^{\SSL,n}(t) \mbox{ for all $t \geq 0$ and $i \in \NN$}.
\]
\end{proposition}
\begin{proof}
The proof is based on a coupling argument in the spirit of Proposition~3 in \cite{bhambay2025asymptotic}. However, to account for the possibility of infinite buffer capacity, the argument differs substantially and requires an induction scheme based on additional stopping times rather than a direct comparison. We prove the inequality $Q_{+i}^{\pi,n}(t) \leq_{st} \tilde{Q}_{b,+i}^{\SSL,n}(t)$ and the proof of the other inequality follows the same lines. Without loss of generality, we assume that $\mu_1^n \leq \mu_2^n \leq \cdots \leq \mu_n^n$. 

As in the JSQ system, we denote by $\tilde{A}_i^n(t)$ the number of arrivals routed to buffers with $i$ or more jobs by time $t \geq 0$, and we define $\tilde{A}_0^n(t) = A^n(t)$, so that both the JSQ and JSQ-SSL systems observe the same arrival stream. When $b<\infty$, we interpret $\tilde{A}_b^n(t)$ as the number of rejected jobs and let $\tilde{A}_i^n(t)=0$ for all $i>b$. Similarly, we denote by $\tilde{D}_i^n(t)$ the number of departures from servers serving buffers with $i$ or more jobs by time $t$. The departure rates from buffers with $i$ or more jobs at time $t$ are given by  
\[
\sum_{k=1}^n \mu_k^n Q_{k,i}^{\pi,n}(t) \text{ and }  \sum_{k=1}^{\tilde{Q}_{b,i}^{\SSL,n}(t)} \mu_k^n,
\]  
for the JSQ and the corresponding JSQ--SSL systems, respectively.  We note that
\begin{equation}
    \sum_{k=1}^{\tilde{Q}_{b,i}^{\SSL,n}(t)} \mu_k^n \leq \sum_{k=1}^{Q_i^{\pi,n}(t)} \mu_k^n \leq \sum_{k=1}^n \mu_k^n Q_{k,i}^{\pi,n}(t) \quad \text{ if } \quad \tilde{Q}_{b,i}^{\SSL,n}(t) \leq Q_i^{\pi,n}(t).
    \label{eq:slower_depart_JSQ_SSL}
\end{equation}

To couple the departures, we define a potential departure process $S^n(t)$ as a Poisson process with rate $\sum_{k=1}^n \mu_k^n$, and let $(\tilde{U}_j : j \in \NN)$ be a sequence of i.i.d.\ Uniform(0,1) random variables. A potential departure occurring at the $j$th epoch $\tilde{\theta}_j^n$ is realized as a departure from queues of size $i$ or more in the JSQ and JSQ--SSL systems if  
\begin{equation}
\tilde{U}_j \leq \frac{\sum_{k=1}^n \mu_k^n Q_{k,i}^{\pi,n}(\tilde{\theta}_j^n-)}{\sum_{k=1}^n \mu_k^n}\text{ and } \tilde{U}_j \leq \frac{\sum_{k=1}^{\tilde{Q}_{b,i}^{\SSL,n}(\tilde{\theta}_j^n-)} \mu_k^n}{\sum_{k=1}^n \mu_k^n},\label{eq:u_coupling_JSQ_SSL}
\end{equation} 
respectively. The server from which the departure occurs is then determined using a splitting argument (see, e.g.,~\eqref{eq:splitting_departures}). 

For any $i\geq 1$ and $0\leq t \leq T$, we can write 
\begin{align}
    \tilde{Q}_{b,i}^{\SSL,n}(t) &= Q_i^{n}(0) + \tilde{A}_{i-1}^n (t) - \tilde{A}_{i}^n(t) -\tilde{D}_{i}^n(t) + \tilde{D}_{i+1}^n(t) \label{eq:modified_jssq_balance}\\
    \tilde{Q}_{b,+i}^{\SSL,n}(t) &= Q_{+i}^{n}(0) + \tilde{A}_{i-1}^n (t) - \tilde{D}_{i}^n(t).  \label{eq:aggregate_modified_jssq_balance}
\end{align}

Let $\tilde{\tau}_i^n = \inf\{t: \tilde{Q}_{b,+i}^{\SSL,n}(t)<Q_{+i}^{\pi,n}(t)\}$. Using \eqref{eq:aggregate_modified_jssq_balance}, $\tilde{Q}_{b,+i}^{\SSL,n}(t)<Q_{+i}^{\pi,n}(t)$ if and only if $\tilde{A}_{i-1}^n (t) - \tilde{D}_{i}^n(t) < A_{i-1}^n (t) - D_{i}^n(t)$. Hence, if $\tilde{\tau}_i^n<\infty$, then we have 
\begin{align}
    \tilde{Q}_{b,+i}^{\SSL,n}(\tilde{\tau}_{i}^n-) &= Q_{+i}^{\pi,n}(\tilde{\tau}_{i}^n-)\label{eq:ordering_Q}\\
    \tilde{A}_{i-1}^n (\tilde{\tau}_{i}^n-) - \tilde{D}_{i}^n(\tilde{\tau}_{i}^n-) &= A_{i-1}^n (\tilde{\tau}_{i}^n-) - D_{i}^n(\tilde{\tau}_{i}^n-)\label{eq:prior_equality}\\
    \tilde{A}_{i-1}^n (t) - \tilde{D}_{i}^n(t) &\geq  A_{i-1}^n (t) - D_{i}^n(t) \mbox{ for all }0\leq t<\tilde{\tau}_{i}^n. \label{eq:prior_inequality}
\end{align}

Assume that there exists a $T$ and $i$ such that $\tilde{\tau}_i^n < T$ and define $i^{*} =\max\{i: \tilde{\tau}_i^n<T\}$. As the system starts with a finite number of jobs in the system and $A^n(T)<\infty$, it is clear that $i^{*} <\infty$. We need to have 
\begin{equation*}
    \tilde{Q}_{b,i^{*}}^{\SSL,n}(\tilde{\tau}_{i^{*}}^n-) \leq Q_{i^{*}}^{\pi,n}(\tilde{\tau}_{i^{*}}^n-),\label{eq:Q_i_ordering}
\end{equation*} 
as otherwise \eqref{eq:ordering_Q} implies $\tilde{Q}_{b,+(i^{*}+1)}^{\SSL,n}(\tilde{\tau}_{i^{*}}^n-) < Q_{+(i^{*}+1)}^{\pi,n}(\tilde{\tau}_{i^{*}}^n-)$, contradicting the maximality of $i^{*}$. Hence, \eqref{eq:slower_depart_JSQ_SSL} and \eqref{eq:u_coupling_JSQ_SSL} imply that the event occurring at $\tilde{\tau}_{i^{*}}^n$ must be an arrival that is routed to servers with $i^{*}$ or more jobs in the JSQ system, but is routed to servers with fewer than $i^{*}$ jobs in the JSQ-SSL system. Together with \eqref{eq:modified_jssq_balance} and \eqref{eq:prior_equality}, this implies
\begin{align}
\tilde{Q}_{b,i^{*}-1}^{\SSL,n}(\tilde{\tau}_{i^{*}}^n-) &< Q_{i^{*}-1}^{\pi,n}(\tilde{\tau}_{i^{*}}^n-) = n,\\
\tilde{A}_{i^{*}-2}^n (\tilde{\tau}_{i^{*}}^n-) - \tilde{D}_{i^{*}-1}^n(\tilde{\tau}_{i^{*}}^n-) &< A_{i^{*}-2}^n (\tilde{\tau}_{i^{*}}^n-) - D_{i^{*}-1}^n(\tilde{\tau}_{i^{*}}^n-),
\end{align}
which implies that $\tilde{\tau}_{i^{*}-1}^n  < \tilde{\tau}_{i^{*}}^n$. Repeating the same argument with $T = \tilde{\tau}_{i^{*}}^n$, we obtain
\[
0 < \tilde{\tau}_{1}^n < \tilde{\tau}_{2}^n < \cdots < \tilde{\tau}_{i^{*}}^n,
\]
and the events causing the transitions at these epochs must be arrivals. At $\tilde{\tau}_{1}^n$, this implies that the arrival must be observed in the original system but not in the modified system. Since the arrival process is the same for both systems, this leads to a contradiction.
\end{proof}

We now derive the diffusion limits for the JSQ-FSL and JSQ-SSL systems with individual-server heterogeneity in Proposition~\ref{prop:diffusion_limit_FSL_SSL} whose proof is given in Appendix~\ref{sec:proof_JSQSSL}. Similar to \eqref{eq:scaled_system}, we define the scaled versions of these systems as $\tilde{\mf{X}}_b^{\pi,n}(t) = (\tilde{X}_{b,i}^{\pi,n}(t):i\in \NN)$ for $\pi\in \{\FSL, \SSL\}$, where
\[
\tilde{X}_{b,1}^{\pi,n}(t) = n^{-1/2}(\tilde{Q}_{b,1}^{\pi,n}(t)-n), \tilde{X}_{b,i}^{\pi,n}(t) = n^{-1/2}\tilde{Q}_{b,i}^{\pi,n}(t) \mbox{ for all }i\geq 2, n\geq 1 \mbox{ and } t\geq 0. 
\]
\begin{proposition}\label{prop:diffusion_limit_FSL_SSL}
    For $\pi \in \{ \FSL, \SSL\}$, $\tilde{\mf{X}}_b^{\pi,n} \Rightarrow \tilde{\mf{X}}^\pi$, where $\tilde{\mf{X}}^\pi$ is the unique solution of
        \begin{align}
            \label{eq:JSQ_SSL_dif1} X_1^\pi(t) & = X_1(0) - (\varsigma + \nu)t + \sqrt{2} W(t) - \mu_{1}^{\pi} \int_0^t X_1^\pi(s) \, ds + \mu_{2}^\pi \int_0^t X_2^\pi(s) \, ds - U_1^{\pi}(t), \\
            \label{eq:JSQ_SSL_dif2}X_2^\pi(t) & = X_2(0) + U_1^{\pi}(t) - \mu_{2}^\pi \int_0^t X_2^\pi(s) ds + \mu_{2}^\pi \int_0^t X_3^\pi(s) \, ds, \\
        \label{eq:JSQ_SSL_dif3}X_i^\pi(t) & = X_i(0) - \mu_{2}^\pi \int_0^t X_i^\pi(s) ds + \mu_{2}^\pi \int_0^t X_{i+1}^\pi(s) \, ds, \quad \text{for all } 3\leq i\leq b,
    \end{align}
    where $\mu_1^{\FSL}=\mu_2^{\SSL} = \mu_{\min}$, $\mu_2^{\FSL}=\mu_1^{\SSL} = \mu_{\max}$ and $U_1^{\pi}(t)\in \DD_{\RR_+}[0,\infty)$ is a non-decreasing function and satisfies
    \begin{align*}
        \int_0^\infty \II(X_1^{\pi}(s)<0)dU_1^{\pi}(s) = 0.
        \end{align*}
\end{proposition}

We conclude this section with a monotonicity property of the JSQ--SSL system with respect to the buffer capacity. The proof is deferred to Appendix~\ref{app:jsq_ssl_proofs}.

\begin{lemma}\label{lem:jsq_ssl_monotonicity}
For any $1\leq b_1<b_2\leq \infty$, all $i\in\NN$ and $\pi\in \{FSL, SSL\}$,
\[
\tilde{Q}_{b_1,i}^{\pi,n}(\infty)\leq_{st} \tilde{Q}_{b_2,i}^{\pi,n}(\infty).
\]
\end{lemma}

\section{Process-Level Limits for Heterogeneous JSQ Systems}
\label{sec:jsq-process-level-limits}

The presence of measure-valued integral equations introduces substantial technical challenges in establishing process-level convergence. At a high level, our proof follows the general structure of the approach for homogeneous JSQ systems in~\cite{eschenfeldt2018join}. However, the details differ significantly and require new ideas. We first summarize the approach in~\cite{eschenfeldt2018join} and then contrast it with ours.

The proof for homogeneous JSQ systems begins by considering a truncated system in which each server queue has a buffer limit of two jobs, with a reflection occurring due to rejected arrivals when the system reaches capacity. After centering the arrival and departure processes to obtain martingales, one shows that the truncated system equations possess a unique solution. Hence, the system equations can be viewed as a mapping that takes these martingales as inputs, and this mapping is continuous in the space of right-continuous functions. By the martingale central limit theorem, the martingales converge to Brownian motions, and the continuous mapping theorem then yields convergence of the system equations to the desired diffusion limit for the truncated system. Establishing this step requires proving convergence of the optional quadratic variations of martingales, which can be done using the fluid limit of the system. This, in turn, implies stochastic boundedness of the scaled processes. Finally, one shows that the limits of the original and truncated systems coincide on any finite time interval $[0,T]$ as a consequence of this boundedness.

For JSQ systems with individual-server level heterogeneity, considering an analog of the truncated system in~\cite{eschenfeldt2018join} requires including a reflection term in \eqref{eq:main_diffusion_xi2} which is particularly challenging as this equation concerns measure-valued processes. The stochastic boundedness in Lemma~\ref{lem:stochastic_boundedness}, proved using direct methods in Appendix~\ref{app:proof_stochastic_boundedness}, allows us to ensure that the probability of an incoming arrival being routed to a server with two or more jobs is negligible in any interval $[0,T]$ for large enough $n$. This allows us to eliminate the need to introduce a reflection in~\eqref{eq:main_diffusion_xi2}. The stochastic boundedness also allows us to prove the convergence of martingales without resorting to the fluid limit. 

To analyze the uniqueness of the solution to the system equations, we adopt a fixed-point approach as in~\cite{eschenfeldt2018join}. Similarly, we begin with a related system in which the reflection term is replaced by the one-dimensional reflection mapping. In the homogeneous case, it is not straightforward to work directly with the right-hand side of these equations. Therefore, rather than working with the related system itself, Eschenfeldt and Gamarnik~\cite{eschenfeldt2018join} introduce an additional system to define their contraction mapping. For heterogeneous systems, however, this second modification is difficult to handle due to the presence of measure-valued terms. Using a trick based on switching between equivalent norms, we can instead work directly with the mapping defined by the related system equations. 

This approach is also applicable to the case of homogeneous systems and allows a significantly shortened proof by eliminating the need to introduce an additional system of equations as in~\cite{eschenfeldt2018join}. As a result, the uniqueness of the fixed point implies the uniqueness of the desired solution without requiring further arguments. Finally, the non-standard definition of the limiting fairness measure prevents us from using the continuous mapping theorem directly. Instead, we take a step back and prove our weak convergence result using the Skorokhod representation theorem. 

\subsection{Uniqueness and Existence of Solutions to the Integral Equations}

We start our proof by showing the existence and uniqueness of a solution to \eqref{eq:main_diffusion_x1}-\eqref{eq:main_reflection}. We consider the equations for $(X_1, \xi_2)$ and $(\xi_i:i\geq 3)$ separately. Even though we work with right-continuous functions in the pre-limit, we use the supremum norm to prove our uniqueness and existence results. 

\begin{lemma}
\label{lem:reflection_equations}
    Consider the set of equations
    \begin{align}
    \label{eq:reflection_mappings_original1}
    x_1(t) &= b_1 + y_1(t) - c(t)\int_0^tx_1(s)ds + \int_0^t\langle \iota, \xi_2(s)\rangle ds -u_1(t)\\
    \label{eq:reflection_mappings_original2}
    \langle f, \xi_2(t)\rangle &= \langle f, \gamma\rangle + \langle f, \alpha(u_1(t))\rangle + \langle f, \beta(t)\rangle - \int_0^t \langle f\times \iota, \xi_2(s)\rangle ds
\end{align}
 for all $f\in \Lipbar{\RR, L}(\SSS)$ and $u_1\in\DD_{\RR_+}[0,\infty)$ is nondecreasing and satisfies
\begin{align}
    \label{eq:reflection_mappings_original3}
    \int_0^\infty \II(x_1(s)<0)du_1(s) = 0,
\end{align}
where $b_1\in \RR_-$, $\gamma\in \cM_f$, $y_1\in \DD_\RR[0,\infty)$, $c\in \GG_{\bar{\SSS}}[0,\infty)$,  $\alpha\in \Lip{\cM_f, L}[0,\infty)$ and $\beta\in \DD_{\cM_s}[0,\infty)$ are given such that $|\alpha|_{TV,t}^*<\infty$ and $|\beta|_{TV,t}^*<\infty$ for all $t\geq 0$. Then \eqref{eq:reflection_mappings_original1}-\eqref{eq:reflection_mappings_original3} has a unique solution $(x_1, u_1, \xi_2)\in \DD_{\RR_-\times \RR_+\times\cM_s}[0,\infty)$.
\end{lemma}

\begin{lemma}
\label{lem:reflection_equations2}
    Consider the set of equations
    \begin{align}
    \label{eq:reflection_mappings_original4}
    x_1(t) &= 0, \langle f, \xi_2(t)\rangle = 0\\
     \label{eq:reflection_mappings_original5}\langle f, \xi_i(t)\rangle & = \langle f, \xi_i(0)\rangle + \langle f, \beta_i(t)\rangle - \int_0^t\langle f\times \iota, \xi_i(s)\rangle ds +  \int_0^t \langle f\times\iota, \xi_{i+1}(s)\rangle ds,\mbox{ for all }i\geq 3,
\end{align}
for all $f\in \Lipbar{\RR, L}(\SSS)$ where $\boldsymbol{\beta}=(\beta_i:i\geq 3)\in \DD_{\cM_s^\infty}[0,\infty)$ is given such that d $|\beta_i|_{TV,t}^*<\infty$ for all $t\geq 0$ and $i\geq 3$. Then \eqref{eq:reflection_mappings_original4}-\eqref{eq:reflection_mappings_original5} has a unique solution $(x_1, \boldsymbol{\xi})\in \DD_{\RR_-\times\cM_s^\infty}[0,\infty)$.
\end{lemma}

As in \cite{eschenfeldt2018join}, we prove Lemma~\ref{lem:reflection_equations} by embedding \eqref{eq:reflection_mappings_original3} in \eqref{eq:reflection_mappings_original1}-\eqref{eq:reflection_mappings_original2} using the well-known one-dimensional reflection map with an upper barrier at 0 where $(\phi, \psi): \DD_{\RR}[0,\infty)\to \DD_{\RR}[0,\infty)\times \DD_{\RR}[0,\infty)$ defined as 
\[
\psi(x)(t) = |(x(\cdot))^+|_t^*, \mbox{ and } \phi(x)(t) = x(t) - \psi(x)(t).
\]
The functions $\psi$ and $\phi$ are known to be 1-Lipschitz and 2-Lipschitz with respect to the supremum norm, respectively, i.e., for all $T>0$, 
\begin{align*}
    |\psi(x)-\psi(x')|_T^*\leq |x-x'|_T^*\mbox{ and }|\phi(x)-\phi(x')|_T^*\leq 2|x-x'|_T^*.
\end{align*}

\begin{lemma}
\label{lem:modified_reflection_equations}
    Consider the set of equations
    \begin{align}
    w_1(t) &= b_1 + y_1(t) - c(t)\int_0^t\phi(w_1)(s)ds + \int_0^t\langle \iota, \xi_2(s)\rangle ds\label{eq:reflected_equation1}\\
    \langle f, \xi_2(t)\rangle &= \langle f, \gamma\rangle + \langle f, \alpha(\psi(w_1)(t))\rangle + \langle f, \beta(t)\rangle - \int_0^t \langle f\times \iota, \xi_2(s)\rangle ds \label{eq:reflected_equation2},
\end{align}
for all $f\in \Lipbar{\RR, L}(\SSS)$, where $b_1\in \RR_-$, $\gamma\in \cM_f$, $y_1\in \DD_\RR[0,\infty)$, $c\in \GG_{\bar{\SSS}}[0,\infty)$,  $\alpha\in \Lip{\cM_f, L}[0,\infty)$ and $\beta\in \DD_{\cM_s}[0,\infty)$ are given such that $|\alpha|_{TV,t}^*<\infty$ and $|\beta|_{TV,t}^*<\infty$ for all $t\geq 0$. Then \eqref{eq:reflected_equation1}-\eqref{eq:reflected_equation2} has a unique solution $(w_1, \xi_2)\in \DD_{\RR\times\cM_s}[0,\infty)$.
\end{lemma}
\begin{proof}
    The proof follows a contraction mapping argument. Contrary to~\cite{eschenfeldt2018join}, we define our operator $\cT: \DD_{\RR\times\cM_s}\to \DD_{\RR\times\cM_s}$ directly using the right-hand side of equations \eqref{eq:reflected_equation1} and \eqref{eq:reflected_equation2} so that
    \begin{align*}
        \cT(\tilde{w}_1, \tilde{\xi}_2)_1(t) &=  b_1 + y_1(t) - c(t)\int_0^t\phi(\tilde{w}_1)(s)ds + \int_0^t\langle \iota, \tilde{\xi}_2(s)\rangle ds
    \end{align*}
    and $\cT(\tilde{w}_1, \tilde{\xi}_2)_2$ is the unique measure that satisfies
    \begin{align*}
        \langle f, \cT(\tilde{w}_1, \tilde{\xi}_2)_2(t)\rangle &= \langle f, \gamma\rangle + \langle f, \alpha(\psi(\tilde{w}_1)(t))\rangle + \langle f, \beta(t)\rangle - \int_0^t \langle f\times \iota, \tilde{\xi}_2(s)\rangle ds
    \end{align*}
for all $f\in \Lipbar{\RR, \tilde{L}}(\SSS)$, where $\tilde{L}>0$ is specified below. Any $(w_1,\xi_2)\in \DD_{\RR\times\cM_s}[0,\infty)$ solves \eqref{eq:reflected_equation1}-\eqref{eq:reflected_equation2} if and only if it is a fixed point of $\cT(\tilde{w}_1, \tilde{\xi}_2)$. Hence, the result follows from the Banach fixed point theorem (cf.~\cite{rudin1976principles}, pg 220) if we can show that $\cT(\tilde{w}_1, \tilde{\xi}_2)$ is a contraction with respect to the supremum norm. To simplify the notation, we use the shorthand notation $|\xi_i-
\xi_i'|_T^*$ to denote $\left|d_{\cM,\tilde{L}}\left(\tilde{\xi}_i,\tilde{\xi}_i'\right)\right|_T^*$.

For any $(\tilde{w}_1, \tilde{\xi}_2),(\tilde{w}_1', \tilde{\xi}_2') \in \DD_{\RR\times\cM_s}[0,\infty) $ and $T>0$, we can write
    \begin{align}
  \nonumber \left|\cT(\tilde{w}_1, \tilde{\xi}_2)_1 - \cT(\tilde{w}_1', \tilde{\xi}_2')_1\right|_T^* &\leq  T\left(\mu_{\max} \left|\phi(\tilde{w}_1)-\phi(\tilde{w}_1')\right|_T^* +\left|\langle \iota, \tilde{\xi}_2\rangle-\langle \iota, \tilde{\xi}_2'\rangle\right|_T^*\right)\\
  \label{eq:reflection_eq_bound1}&\leq T\mu_{\max}\left(2 \left|\tilde{w}_1-\tilde{w}_1'\right|_T^*  +\frac{1}{\tilde{L}}\left|\tilde{\xi}_2-\tilde{\xi}_2'\right|_T^*\right),
    \end{align}
where the second inequality follows from the assumption that $\mu_{\max}\geq 1$ and hence, $\tilde{L}\iota(\mu)/\mu_{\max}\in \Lip{\SSS, \tilde{L}}$ and is bounded by $\tilde{L}$. Similarly, we know that for any $f\in\Lip{\SSS,\tilde{L}}[0,\infty)$, $f\times \iota\in \Lip{\SSS, \mu_{\max}\tilde{L}}[0,\infty)$ and $\sup_{\mu\in \cS}|\mu f(\mu)|\leq \mu_{\max} \tilde{L}$ and hence,
\begin{align}
     \nonumber \left|\langle f, \cT(\tilde{w}_1, \tilde{\xi}_2)_2\rangle-  \langle f, \cT(\tilde{w}_1', \tilde{\xi}_2')_2\rangle\right|_T^*&\leq  \nonumber \left|\langle f, \alpha(\psi(\tilde{w}_1))\rangle - \langle f, \alpha(\psi(\tilde{w}_1'))\rangle\right|_T^* +
    T\left|\langle f\times \iota, \tilde{\xi}_2\rangle - \langle f\times \iota, \tilde{\xi}_2'\rangle\right|_T^*\\
     \label{eq:reflection_eq_bound2} &\leq L\tilde{L}\left|\tilde{w}_1- \tilde{w}_1'\right|_T^* + T\mu_{\max}\left|\tilde{\xi}_2-\tilde{\xi}_2'\right|_T^*.
\end{align}
Choosing $\tilde{L}=(4L)^{-1}$ and $T_1 < (\max\{4\mu_{\max}/\tilde{L},8\mu_{\max}\})^{-1}$, taking the supremum over all $f\in \Lipbar{\RR, \tilde{L}}(\SSS)$ in \eqref{eq:reflected_equation2} and summing with \eqref{eq:reflection_eq_bound1},
\begin{align*}
    \left|\cT(\tilde{w}_1, \tilde{\xi}_2) - \cT(\tilde{w}_1', \tilde{\xi}_2')\right|_T^* \leq \frac{1}{2}\left|(\tilde{w}_1,\tilde{\xi}_2) -(\tilde{w}_1', \tilde{\xi}_2')\right|_T^*.
\end{align*}
Hence, $\cT$ is a contraction mapping on $[0,T_1]$. To be able to use Banach fixed point theorem, we need to ensure the completeness of the underlying space. As the space $\Lipbar{\RR, L}(\SSS)$ separates finite signed measures on the compact space $\SSS$, the equation holds for any bounded measurable $f$. Hence, equations \eqref{eq:reflected_equation1} and \eqref{eq:reflected_equation2} imply
\begin{align}
        |w_1|_t^* &\leq |b_1| + |y_1|_t^* + 2|c|_t^*\int_0^t|w_1|_s^*ds + \mu_{\max}\int_0^t|\xi_2|_{TV,s}^* ds\label{eq:reflected_equation1_ineq}\\
    |\xi_2|_{TV,t}^* &\leq |\gamma|_{TV} + |w_1|_t^* + | \beta|_{TV,t}^* + \mu_{\max}\int_0^t |\xi_2|_{TV,s}^* ds \label{eq:reflected_equation2_ineq}.
\end{align}
 Multiplying \eqref{eq:reflected_equation1_ineq} by two, adding two equations and using Gronwall's inequality, we can find $C_T>0$ such that
\[
|b_1| + |y_1|_T^* + |\gamma|_{TV} + | \beta|_{TV,T}^*<\frac{C_T}{2} \mbox{ and }  |w_1|_T^* +  |\xi_2|_{TV,T}^* \leq C_T,
\]
for any solution to the equations \eqref{eq:reflected_equation1} and \eqref{eq:reflected_equation2}. Let $T_0=(4(|c|_T^*+\mu_{\max}))^{-1}\wedge T_1$ and define the set
\[
\FF_{T_0} = \{(w_1, \xi_2)\in \DD_{\RR\times\cM_s}[0,T_0]: |w_1|_{T_0}^*\leq C_T, |\xi_2|_{TV,T_0}^*\leq 2C_T\}.
\]
Operator $\cT$ is a contraction mapping and using the definition of the operator it can easily be seen using \eqref{eq:reflected_equation1_ineq} and \eqref{eq:reflected_equation2_ineq} that $\cT$ maps $\FF_{T_0}$ into itself on $[0,T_0]$. Moreover, as the set of signed measures on a compact space with uniformly bounded total variation is complete (cf \cite{bogachev2007measure}, Theorem 8.6.2, pg 202), $\FF_{T_0}$ is complete. Hence, using the Banach fixed-point theorem, $\cT$ has a unique fixed point on this interval. Repeating the same argument for intervals $[T_0,2T_0], [2T_0,3T_0], \ldots$ as in \cite{eschenfeldt2018join}, we conclude that there exists a unique fixed point of $\cT(\tilde{w}_1, \tilde{\xi}_2)$ which is also the unique solution to \eqref{eq:reflected_equation1}-\eqref{eq:reflected_equation2}.
\end{proof}

\begin{proof}[Proof of Lemma~\ref{lem:reflection_equations}] 
 Let $(w_1, \xi_2)$ be the unique solution of the modified equations \eqref{eq:reflected_equation1}-\eqref{eq:reflected_equation2}. Then, $(x_1, u_1,\xi_2)=(\phi(w_1),\psi(w_1), \xi_2)$ solves \eqref{eq:reflection_mappings_original1} - \eqref{eq:reflection_mappings_original3}. To see uniqueness, let $(\tilde{x}_1, \tilde{u}_1,\tilde{\xi}_2)$ be another solution. Setting $\tilde{w}_1 = \tilde{x}_1 + \tilde{u}_1$, and using the uniqueness of reflection mapping (cf. \cite{chen2001fundamentals}, Theorem 6.1), we have $\tilde{x}_1=\phi(\tilde{w}_1)$ and $\tilde{u}_1 = \psi(\tilde{w}_1)$. Then, $(\tilde{w}_1, \tilde{\xi}_2)$ also solves \eqref{eq:reflected_equation1}-\eqref{eq:reflected_equation2}, which contradicts the uniqueness in Lemma~\ref{lem:modified_reflection_equations}. Hence, we conclude that \eqref{eq:reflection_mappings_original1} - \eqref{eq:reflection_mappings_original3} has a unique solution.
\end{proof}

\begin{proof}[Proof of Lemma~\ref{lem:reflection_equations2}]
The proof follows the same lines as in the proof of Lemma 5 in \cite{eschenfeldt2018join} after the notation is modified to handle the measure-valued form. We define the operator $\cT_2:\DD_{\cM_s^{\infty}}\to \DD_{\cM_s^{\infty}}$ such  that $\cT_2(\boldsymbol{\xi})_1(t)= \cT_2(\boldsymbol{\xi})_2(t) = 0$ and $\cT_2(\boldsymbol{\xi})_i$ is given by
\begin{align*}
    \left\langle f,\cT_2(\boldsymbol{\xi})_i(t)\right\rangle &= \langle f, \xi_i(0)\rangle + \langle f, \beta_i(t)\rangle - \int_0^t\langle f\times \iota, \xi_i(s)\rangle ds +  \int_0^t \langle f\times\iota, \xi_{i+1}(s)\rangle ds,\mbox{ for all }i\geq 3
\end{align*}
for all $f\in  \Lipbar{\RR, \tilde{L}}(\SSS)$. For any two vectors $\boldsymbol{\xi},\boldsymbol{\xi}'\in\cM_s^{\infty}$,
\begin{align*}
    \left|\cT_2(\boldsymbol{\xi})-\cT_2(\boldsymbol{\xi}')\right|_T^*&\leq T\mu_{\max} \left(\sum_{i=1}^\infty\rho^{i}\left|\xi_i- \xi_i'\right|_T^* + \rho^{-1} \rho^{(i+1)}\left|\xi_{i+1}- \xi_{i+1}'\right|_T^*\right)\leq T\mu_{\max} (1+\rho^{-1}) \left|\boldsymbol{\xi}-\boldsymbol{\xi}'\right|_T^*.
\end{align*}
Choosing $T_0<(\mu_{\max}(1+\rho^{-1}))^{-1}$, $\cT_2$ is a contraction mapping on the interval $[0,T_0]$ and hence, has a unique fixed point following a completeness argument as in the proof of Lemma~\ref{lem:modified_reflection_equations}. Repeating the argument over $[T_0,2T_0],[2T_0, 3T_0],\ldots$, we prove the existence and uniqueness of a solution to \eqref{eq:reflection_mappings_original4}-\eqref{eq:reflection_mappings_original5}.

Now, suppose $(\boldsymbol{\xi}^n(0), \boldsymbol{\beta}^n)\to (\boldsymbol{\xi}(0), \boldsymbol{\beta})$ in the supremum norm, i.e., for any $\epsilon>0$, there exists an $n_\epsilon$ such that $n>n_{\epsilon}$ implies
\[
\sum_{i=3}^\infty\rho^{i}(|\xi_i^n(0)-\xi_i(0)| + |\beta_i^n-\beta_i|_T^*)<\epsilon.
\]
Then, for any such $n$, we have
\begin{align*}
    \left|\boldsymbol{\xi}^n- \boldsymbol{\xi}\right|_t^* &\leq \epsilon +\mu_{\max}(1+\rho^{-1})\int_0^t  \left|\boldsymbol{\xi}^n- \boldsymbol{\xi}\right|_s^*ds.
\end{align*}
By Gronwall's inequality, 
\begin{align*}
   \left|\boldsymbol{\xi}^n- \boldsymbol{\xi}\right|_T^*\leq \epsilon e^{\mu_{\max}(1+\rho^{-1})T}.
\end{align*}
This implies the continuity of the solution mapping. 
\end{proof}


\subsection{Martingale Representations}
\label{subsec:martingale-representations}

Our process-level convergence proof relies on the martingale central limit theorem as presented in~\cite{ptw07}. Let $\cF_t^n$ be the filtration generated by $\mf{Q}^n(t)$. The arrival processes $A^n(t)$ are Poisson processes with rate $n\lambda^n$, and hence the process
\begin{align*}
    M_A^n(t) &=n^{-1/2}\left(A^n(t) - n\lambda^nt\right)
\end{align*}
is an $\cF_t^n$-martingale. 

Let $D_{k,i}^n(t)$ denote the number of departures from server $k$ when the server has $i$ or more jobs. We define the potential departure process from server $k$, $S_k^n(t)$, as a Poisson process with rate $\mu_k^n$. Then, using a thinning argument, a potential departure occurring at time $t$ is considered an actual departure for $D_{k,i}^n(t)$ if $Q_{k,i}^n(t-)=1$. Hence, Lemma 3.4 in~\cite{ptw07} implies that, for all $k,i\in \NN$, the processes
\begin{align*}
    M_{D,k,i}^n(t) &=n^{-1/2}\left( D_{k,i}^n(t) - \int_0^t \mu_k^n Q_{k,i}^n(s-)ds\right)
\end{align*}
are $\cF_t^n$-martingales. As a consequence, for any $f\in \Lipbar{\RR, L}(\SSS)$, the processes 
\begin{align*}
    M_{D,i}^{f,n}(t) &= \langle f, \hat{\Delta}_{i}^n(t) \rangle - \int_0^t \langle f\times \iota, \xi_i^n(s-)\rangle ds = \sum_{k=1}^nf(\mu_k^n)M_{D,k,i}^n(t)
\end{align*}
and 
\begin{align}
    M_{D,i}^n(t) = \sum_{k=1}^n M_{D,k,i}^n(t)
\end{align}
are also $\cF_t^n$-martingales.

\begin{lemma}\label{lem:martingale_convergence}
    Suppose $W_1(t)$ and $W_2(t)$ are independent $\cF_t$-Brownian motions. Then, as $n\to \infty$, we have
    \begin{enumerate}
        \item $(M_A^n(t), M_{D,1}^n(t))\Rightarrow (W_1(t), W_2(t))$,
        \item $M_{D,i}^{f,n}\toP 0$ for all $i\geq 2$ uniformly for all $f\in \Lipbar{\RR,L}(\SSS)$. 
    \end{enumerate}
\end{lemma}
\begin{proof}
    Defining $S(t)$ to be a unit rate Poisson process, we see that the martingale $M_{D,1}^n(t)$ is stochastically equivalent to 
    \begin{align}\label{eq:martingale_time_change}
    n^{-1/2}\left(S\left(\sum_{k=1}^n\mu_k^n\int_0^tQ_{k,1}^n(s)ds\right) - \sum_{k=1}^n\mu_k^n\int_0^tQ_{k,1}^n(s)ds\right).
    \end{align}
    We have 
    \begin{align*}
        \frac{\sum_{k=1}^n\mu_k^n\int_0^tQ_{k,1}^n(s)ds}{n} = \bar{\mu}^nt - \frac{\sum_{k=1}^n\mu_k^n\int_0^t(1-Q_{k,1}^n(s))ds}{n}\to t
    \end{align*}
    as Assumption~\ref{asm:arrival_service_rates} and the stochastic boundedness of $X_1^n(t)$ in Lemma~\ref{lem:stochastic_boundedness} imply that the first term converges to $t$  and the second term converges to 0. Using the functional central limit theorem for Poisson processes (cf. Theorem~4.2 in \cite{ptw07}), we have 
    \[\left(M_{A}^n(t),\frac{S(nt)-nt}{n^{1/2}},  \frac{\sum_{k=1}^n\mu_k^n\int_0^tQ_{k,1}^n(s)ds}{n}\right)\Rightarrow (W_1(t),W_2(t),t).\]
    Using the random-time change theorem (cf. Proposition 13.2.1 in \cite{whitt2002stochastic}) with \eqref{eq:martingale_time_change}, we obtain $(M_A^n(t), M_{D,1}^n(t))\Rightarrow (W_1(t), W_2(t))$.

    Concentrating on the predictable quadratic variation  $\llangle M_{D,i}^{f,n}(t)\rrangle$ of the martingale $M_{D,i}^{f,n}$, and again using Lemma~\ref{lem:stochastic_boundedness}, we have
    \begin{align*}
        \left|\llangle M_{D,i}^{f,n}(t)\rrangle\right|_T^* = \left|\frac{\int_0^t \langle f^2\times \iota^2, \xi_i^n(s-)\rangle ds}{n}\right|_T^* \leq \mu_{\max}^2L^2\frac{|X_i^n(t)|_T^*}{n^{1/2}}\to 0,
    \end{align*}
    for any $T>0$. Hence, Part 2 again follows from the martingale central limit theorem. 
\end{proof}

\subsection{Convergence of System Processes}
\label{subsec:convergence-system-processes}

We now have the necessary tools for the proof of Theorem~\ref{thm:process_level_convergence}. The assumptions of the theorem imply convergence in the Skorokhod-$J_1$ topology and hence, we use the relevant metric in this section.
\begin{proof}[Proof of Theorem~\ref{thm:process_level_convergence}] 
Let $\epsilon_k = 1/k$ and fix $f \in\Lipbar{\RR,L}(\SSS)$. Using Lemma~2 in \cite{buke2023many}, the assumptions of the theorem, and  Lemmas~\ref{lem:alpha_continuity} and \ref{lem:martingale_convergence}, we may assume without loss of generality that the vector
\begin{equation}\label{eq:converging_vector}
\left(\nu^n, X_1^n(0), \boldsymbol{\xi}^n(0),
    M_A^n(t),
    M_{D,1}^n(t),
    (M_{D,i}^{f,n}(t) : i \ge 2),
    (\alpha_i^n(t) : i \in \mathbb{N}),
    (\tau_{\epsilon_k}^n, \mathcal{S}_{\epsilon_k}\eta^n(t) : k \in \mathbb{N})
\right)
\end{equation}
converges almost surely in the product Skorokhod topology on $[0,T]$ for $T>0$. Hence, we may select a single sequence of Skorokhod time-change 
functions $\Lambda^n : [0,T] \to [0,T]$ such that $\Lambda^n$ are differentiable with derivative $\dot{\Lambda}^n(t)$ and $|\dot{\Lambda}^n(t) - 1|_T^*\toP 0$. 

Setting $\beta(t)$ in \eqref{eq:reflection_mappings_original5} to be the unique signed measure that satisfies $\langle f,\beta^n(t)\rangle=-M_{D,i}^{f,n}+M_{D,i+1}^{f,n}$ for all $f\in \Lipbar{\RR,L}(\SSS)$, the continuity of the mapping  in Lemma~\ref{lem:reflection_equations2} implies the convergence of $(\xi_{i}^n(t):i\geq 3)$ to the solution of \eqref{eq:main_diffusion_xi3}. To prove the convergence of $(X_1^n(t), \xi_2^n(t))$, we consider $(V_1^n(t)=X_1^n(t)+\hat{A}_1^n(t), \xi_2^n(t))$ which is also the unique solution of \eqref{eq:reflected_equation1}-\eqref{eq:reflected_equation2} with $y_1(t) = M_A^n(t) - M_{D,1}^n(t) + M_{D,2}^n(t)$, $c(t) = \langle \iota, \eta^n(t)\rangle$, $\gamma = \xi^n(0)$ and $\langle f,\beta(t)\rangle =-M_{D,2}^{f,n}+M_{D,3}^{f,n}$ i.e.,
\begin{align*}
    V_1^n(t) &= X_1^n(0)-(\varsigma^n+\nu ^n)t +M_{A}^n(t)-M_{D,1}^n(t)+M_{D,2}^n(t)-\langle \iota, \eta^n(t)\rangle\int_0^t \phi(V_1^n(s))ds\\&\qquad\qquad\qquad\qquad\qquad\qquad\qquad\qquad\qquad\qquad\qquad\qquad\qquad\qquad+\int_0^t \langle \iota, \xi_2^n(s)\rangle ds\\
     \langle f, \xi_2^n(t)\rangle &= \langle f, \xi_2^n(0)\rangle + \langle f, \hat{\alpha}^n(\psi(V_1^n)(t))\rangle-M_{D,2}^{f,n}(t)+M_{D,3}^{f,n}(t) - \int_0^t \langle f\times \iota, \xi_2^n(s)\rangle ds\\&\qquad\qquad\qquad\qquad\qquad\qquad\qquad\qquad\qquad\qquad\qquad\qquad\qquad\qquad+\int_0^t \langle f\times \iota, \xi_3^n(s)\rangle ds.
\end{align*}
Similarly, let $(X_1(t),U_1(t), \xi_2(t))$ be the solution of \eqref{eq:main_diffusion_x1} and \eqref{eq:main_diffusion_xi2} and $(V_1(t)=X_1(t)+U_1(t),\xi_2(t))$ to be the solution of \eqref{eq:reflected_equation1}-\eqref{eq:reflected_equation2} with $y_1(t) = W_1(t) - W_2(t)$, $c(t)=\langle \iota, \eta(t)\rangle$, $\gamma = \xi(0)$ and $\langle f,\beta(t)\rangle = 0$, i.e.,
\begin{align}
    \label{eq:reflection_replaced1}V_1(t) &= X_1(0) +W_1(t) - W_2(t) - \langle \iota, \eta(t)\rangle\int_0^t\phi(V_1(s))ds + \int_0^t\langle \iota, \xi_2(s)\rangle ds\\
     \label{eq:reflection_replaced2}\langle f, \xi_2(t)\rangle &= \langle f, \xi_2(0)\rangle + \langle f, \alpha(\psi(V_1)(t))\rangle - \int_0^t \langle f\times \iota, \xi_2(s)\rangle ds  +\int_0^t\langle f \times \iota, \xi_3(s)\rangle ds.
\end{align}
Our result follows if we can show that for any $\varepsilon, \rho>0$, there exists an $n_{\varepsilon, \rho}$ such that $n>n_{\varepsilon, \rho}$ implies
\[
     \PP\left(d_{D,\RR\times \cM,L}\left((V_1^n,\xi_2^n), (V_1, \xi_2\right))>\varepsilon\right) <\rho.
\]

Using the continuity of the solutions to \eqref{eq:main_diffusion_x1}--\eqref{eq:main_diffusion_xi2} and the stochastic boundedness in Lemma \ref{lem:stochastic_boundedness}, for any $\rho>0$ we can find $K_{\rho}>0$ such that
$|X_1(t)|_T^*\leq K_{\rho}$, $|X_2(t)|_T^*\leq K_{\rho}$ and, for all $n$,
\begin{align*}
\PP\Bigl(d_{D,\RR\times\mathcal{M},L}\bigl((V_1^n,\xi_2^n),(V_1,\xi_2)\bigr)>\varepsilon\Bigr)
&\le \PP\Bigl(d_{D,\RR\times\mathcal{M},L}\bigl((V_1^n,\xi_2^n),(V_1,\xi_2)\bigr)>\varepsilon,\ 
|X_i^n|_T^*<K_{\rho},i=1,2\Bigr)\\
&\qquad\qquad + \frac{\rho}{2}.
\end{align*}
Hence, without loss of generality, we assume that these terms are bounded by $K_{\rho}$ in the remainder of the proof. 

Applying the time change $\Lambda^n(t)$ and taking the difference, we obtain
\begin{align}
    \nonumber\left|V_1^n-V_1\circ\Lambda^n\right|_t^*&\leq |X_1^n(0)-X_1(0)| + \left|M_A^n-W_1\circ\Lambda^n\right|_t^*+ \left|M_{D,1}^n-W_2\circ\Lambda^n\right|_t^*+ \left|M_{D,2}^n(t)\right|_t^*\\
    \nonumber&\quad +\left|\langle \iota, \eta^n(t)\rangle\int_0^t\phi(V_1^n(s))ds - \langle \iota, \eta(\Lambda^n(t))\rangle\int_0^{\Lambda^n(t)}\phi(V_1(s))ds\right|_t^*\\
    &\quad +\left|\int_0^t \langle \iota, \xi_2^n(s)\rangle ds - \int_0^{\Lambda^n(t)}\langle \iota, \xi_2(s)\rangle ds \right|_t^*\label{eq:v_decomposition}
\end{align}
The uniqueness argument in the proof of Lemma~\ref{lem:reflection_equations} implies that $X_1^n=\phi\circ V_1^n$ and for any $\epsilon>0$, the definition of $\tau_\epsilon$ implies
\begin{align*}
    \left|\langle \iota, \eta^n\rangle\int_0^\cdot\phi(V_1^n(s))ds-\langle \iota, \cS_{\epsilon}\eta^n\rangle\int_{\tau_\epsilon^n}^\cdot\phi(V_1^n(s))ds\right|_T^*\leq \mu_{\max}\epsilon.
\end{align*}
Using a similar argument for $\eta$ and $V_1$, we can decompose the fifth term on the right-hand side of \eqref{eq:v_decomposition} for any $0\leq t\leq T$ as
\begin{align}
 \nonumber&\left|\langle \iota, \eta^n\rangle \int_{0}^\cdot\phi(V_1^n(s))ds - \langle \iota, \eta\circ\Lambda^n\rangle\int_{0}^{\Lambda^n(\cdot)}\phi(V_1(s))ds\right|_t^*\\
 \nonumber&\quad \leq 2\epsilon \mu_{\max} + \mu_{\max}K_\rho|\tau_\epsilon-\Lambda^n(\tau_\epsilon)|+\left|\langle \iota, \cS_{\epsilon}\eta^n\rangle\int_{\tau_\epsilon^n}^\cdot\phi(V_1^n(s))ds - \langle \iota, \cS_{\epsilon}\eta\circ\Lambda^n\rangle\int_{\Lambda^n(\tau_\epsilon)}^{\Lambda^n(\cdot)}\phi(V_1(s))ds\right|_t^*\\
 \nonumber&\quad \leq 2\epsilon \mu_{\max} + \mu_{\max}K_\rho|\tau_\epsilon-\Lambda^n(\tau_\epsilon)| + T K_\rho\left|\langle \iota, \cS_\epsilon\eta^n\rangle - \langle \iota, \cS_\epsilon\eta\circ\Lambda^n\rangle\right|_T^*\\
 \nonumber&\qquad + \mu_{\max}K_{\rho}\left|\int_{\tau_\epsilon^n}^\cdot\phi(V_1^n(s))ds -\int_{\tau_\epsilon}^{\cdot}\phi(V_1(\Lambda^n(s)))\dot{\Lambda}(s)ds\right|_T^*\\
  \nonumber&\quad \leq 2\epsilon \mu_{\max} + \mu_{\max}K_\rho|\tau_\epsilon-\Lambda^n(\tau_\epsilon)| + T K_\rho\left|\langle \iota, \cS_\epsilon\eta^n\rangle - \langle \iota, \cS_\epsilon\eta\circ\Lambda^n\rangle\right|_T^*  + \mu_{\max}K_\rho^2|\tau_\epsilon^n-\tau_\epsilon|\\
 \nonumber&\qquad + \mu_{\max}K_{\rho}\int_{\tau_\epsilon}^\cdot\left|\phi\circ V_1^n -\phi\circ (V_1\circ\Lambda^n)\times \dot{\Lambda}\right|_s^*ds \\
   \nonumber&\quad \leq 2\epsilon \mu_{\max} + \mu_{\max}K_\rho|\tau_\epsilon-\Lambda^n(\tau_\epsilon)| + T K_\rho\left|\langle \iota, \cS_\epsilon\eta^n\rangle - \langle \iota, \cS_\epsilon\eta\circ\Lambda^n\rangle\right|_T^*\\
 \nonumber&\qquad + \mu_{\max}K_{\rho}\left(\int_{\tau_\epsilon}^\cdot\phi(V_1(\Lambda^n(s)))\left|1-\dot{\Lambda}\right|_s^*ds + \int_{\tau_\epsilon}^\cdot\left|\phi\circ V_1^n -\phi\circ(V_1\circ\Lambda^n)\right|_s^*ds \right)\\
\nonumber&\quad \leq 2\epsilon \mu_{\max} + \mu_{\max}K_\rho|\tau_\epsilon-\Lambda^n(\tau_\epsilon)| + T K_\rho\left|\langle \iota, \cS_\epsilon\eta^n\rangle - \langle \iota, \cS_\epsilon\eta\circ\Lambda^n\rangle\right|_T^* + 2\mu_{\max}TK_{\rho}^2\left|1-\dot{\Lambda}\right|_s^*\\
&\qquad + 2\mu_{\max}K_{\rho}\int_{0}^\cdot\left|V_1^n - V_1\circ\Lambda^n\right|_s^*ds. \label{eq:fifth_term_decomposition}
\end{align}
Using a similar logic on the sixth term on the right-hand side of \eqref{eq:v_decomposition}, we obtain
\begin{align}
   \nonumber \left|\int_0^t \langle \iota, \xi_2^n(s)\rangle ds - \int_0^{\Lambda^n(t)}\langle \iota, \xi_2(s)\rangle ds \right|_t^*&\leq \int_0^t \left|\langle \iota, \xi_2^n\rangle - \langle \iota, \xi_2\circ\Lambda^n\rangle\times \dot{\Lambda}\right|_s^* ds\\
    \nonumber &\leq \mu_{\max}K_{\rho}|1-\dot{\Lambda}^n|_T^* + \int_0^t \left|\langle \iota, \xi_2^n\rangle - \langle \iota, \xi_2\circ\Lambda^n\rangle\right|_s^* ds\\
    &\leq \mu_{\max}K_{\rho}|1-\dot{\Lambda}^n|_T^* + 4\mu_{\max}\int_0^t \left|d_{\cM,1/4}(\xi_2^n, \xi_2\circ\Lambda^n)\right|_s^* ds \label{eq:sixth_term_decomposition}
\end{align}
We now concentrate on \eqref{eq:reflection_replaced2} and proceed in a similar fashion. For any $f\in \Lip{\RR,1/4}(\SSS)$ such that $\sup_{x\in S}|f(x)|\leq 1/4$
\begin{align}
   \nonumber \left|\langle f, \xi_2^n\rangle - \langle f, \xi_2\circ\Lambda^n\rangle\right|_t^* &\leq  \left|\langle f, \xi_2^n(0)\rangle - \langle f, \xi_2(0)\rangle\right| + \left|\langle f, \alpha^n\circ\psi(V_1^n)\rangle - \langle f, \alpha\circ\psi(V_1\circ\Lambda^n)\rangle\right|_t^*\\
   \nonumber&\quad +|M_{D,2}^{f,n}|_t^* + |M_{D,3}^{f,n}|_t^*  +\left\vert \int_0^\cdot \langle f \times \iota, \xi_3^n(s)\rangle ds-\int_0^{\Lambda^n(\cdot)} \langle f\times \iota, \xi_3(s)\rangle ds\right\vert_t^*\\
   \nonumber&\quad+ \int_0^t \left|\langle f\times \iota, \xi_2^n\rangle - \langle f\times \iota, \xi_2\circ \Lambda^n\rangle\right|_s^* ds\\
   \nonumber &\leq  d_{\cM, 1/4} (\xi_2^n(0), \xi_2(0)) + \frac{1}{2}\left|V_1^n- V_1\circ\Lambda^n\right|_t^*\\
   \nonumber &\quad+\cR^n +\mu_{\max}\int_0^t \left|d_{\cM,1/4}( \xi_2^n, \xi_2\circ \Lambda^n)\right|_s^* ds\\
   \left|d_{\cM,1/4}( \xi_2^n, \xi_2\circ \Lambda^n)\right|_t^*&\leq d_{\cM, 1/4} (\xi_2^n(0), \xi_2(0)) + \frac{1}{2}\left|V_1^n- V_1\circ\Lambda^n\right|_t^*+\cR^n \\&\quad+ \mu_{\max}\int_0^t \left|d_{\cM,1/4}( \xi_2^n, \xi_2\circ \Lambda^n)\right|_s^* ds.\label{eq:xi_decomposition}
\end{align}
where $\cR^n := \big|M_{D,2}^{f,n}\big|_T^* + \big|M_{D,3}^{f,n}\big|_T^*+ \mu_{\max}K_\rho\,|1-\dot\Lambda^n|_T^*+ 2\mu_{\max}\int_0^T \big|d_{\cM,1/4}(\xi_3^n,\xi_3\circ\Lambda^n)\big|_s^*\,ds$
collects the martingale increments and the level-three inflow, and satisfies $\sup_f\cR^n\toP 0$ by Lemma~\ref{lem:martingale_convergence}, \eqref{eq:converging_vector}, and the convergence of $(\xi_i^n:i\ge3)$.

Replacing \eqref{eq:fifth_term_decomposition} and \eqref{eq:sixth_term_decomposition} in \eqref{eq:v_decomposition}, multiplying with 3/2 and adding to the \eqref{eq:xi_decomposition} and finally using the convergence of the vector in \eqref{eq:converging_vector}, for every $\varepsilon>0$, we can choose $\epsilon>0$ and $n_\varepsilon\in \NN$ such that $n>n_{\varepsilon}$ implies
\begin{align}
    \nonumber \left|V_1^n-V_1\circ\Lambda^n\right|_t^* +  \left|d_{\cM,1/4}( \xi_2^n, \xi_2\circ \Lambda^n)\right|_t^* &\leq \varepsilon + 3\mu_{\max}K_{\rho}\int_{0}^t\left|V_1^n - V_1\circ\Lambda^n\right|_s^*ds\\&\quad + 7\mu_{\max}\int_0^t \left|d_{\cM,1/4}( \xi_2^n, \xi_2\circ \Lambda^n)\right|_s^* ds
\end{align}
Using the Gronwall's inequality (cf. Lemma 4.1 in \cite{ptw07}), we have
\[
\left|V_1^n-V_1\circ\Lambda^n\right|_T^* +  \left|d_{\cM,1/4}( \xi_2^n, \xi_2\circ \Lambda^n)\right|_T^*\leq \varepsilon e^{\mu_{\max}\max\{3K_\rho, 7\}T},
\]
which implies that $(V_1^n,\xi_2^n)\toP (V_1,\xi_2)$ in the Skorokhod topology and concludes the proof.
\end{proof}

\begin{proof}[Proof of Corollary~\ref{cor:tau_0_continuity}]
The continuous mapping theorem implies that if $\tau_0^\pi>t$, then for all $u\in [0,t]$, $X_1(s) = 0$ and this implies
\[
\sqrt{2}\,W(u)
=
- X_1(0)
+ (\varsigma+\nu)u
- \int_0^u \langle \iota,\xi_2^{\pi}(s)\rangle\,ds
+ U_1^{\pi}(u).
\]
The right-hand side is the sum of finite variation processes and hence, this is only possible if $W(t)$ has finite variation. Hence, we conclude that $\PP(\tau_0^\pi>t) = 0$ for all $t>0$ and the result follows. 
\end{proof}

By Theorem~\ref{thm:process_level_convergence}, it suffices to prove the convergence of $\eta^n$ and $\alpha_1^n$ to obtain the corollaries in Section~\ref{sec:main_results}. Since the convergence of $\eta^n$ is already established in Propositions~\ref{prop:priority_idleness} and \ref{prop:totally_blind_idleness}, we now focus on $\alpha_1^n$. The following lemma plays a key role in understanding the convergence of $\alpha_1^n$. 

\begin{lemma}\label{lem:A_1_bound}
Let $\tau_s^n:=\inf\{t:A_1^{\pi,n}(t)>n^{1/2}s\}$ for any $s\geq 0$. Then, for any $\epsilon>0$, there exists a $T_{s,\epsilon}$ such that $\PP(\tau_s^n>T_{s,\epsilon})<\epsilon$ for all $n\in \NN$.
\end{lemma}

\begin{proof}
Since the result holds for any tie-breaking policy, we suppress the superscript $\pi$ throughout the proof to simplify the notation. Assumptions~\ref{asm:arrival_service_rates} and \ref{asm:xi_initial} imply that for all $\epsilon>0$, there exists $K_{\epsilon}>s$ such that 
\[
\PP\left(\{X_1^n(0)<-K_{\epsilon}\}\cup \{X_2^n(0)>K_{\epsilon}\}\cup \left\{\sum_{k=1}^n\mu_k^n> n\lambda^n + K_{\mu,\epsilon}n^{1/2}\right\}\right)<\frac{\epsilon}{2}
\]
and we can write
\begin{align}
    \PP(\tau_s^n>T_{s,\epsilon}^\alpha) &\leq \PP\left(\tau_s^n>T_{s,\epsilon}^\alpha|X_1^n(0)\geq -K_{\epsilon}, X_2^n(0)\leq K_{\epsilon}, \sum_{k=1}^n\mu_k^n\leq n\lambda^n + K_{\mu,\epsilon}n^{1/2}\right) + \frac{\epsilon}{2}
    \label{eq:probability_decomp_alpha}
\end{align}
Once we show that we can choose $T_{s,\epsilon}^\alpha$ to be large enough to make the conditional probability on the right-hand side of \eqref{eq:probability_decomp_alpha} below $\epsilon/2$, our result follows. In the remainder of the proof, we assume that $X_1^n(0)\geq -K_{\epsilon}$, $X_2^n(0)\leq K_\epsilon$, $\sum_{k=1}^n\mu_k^n\leq n\lambda^n + K_{\mu,\epsilon}n^{1/2}$ and all expectations and probabilities below should be understood as conditional on this event. 

Using \eqref{eq:aggregate_balance}, for any $0\leq t_0\leq t$, we can bound $A_1^n(t)$ from below as
\begin{align}
    A_1^n(t) & \geq Q_1^n(t_0)-n + A^n(t) - A^n(t_0) - D_1^n(t) + D_1^n(t_0),
    \label{eq:route_to_queue_bound}
\end{align}
which implies that, if the arrivals occurring on any interval $[t_0,t]$ exceed the departures over the same period by $n^{1/2}s$ plus the initially idle servers at time $t_0$, then $A_1(t)\geq n^{1/2}s$. Hence, we prove the proposition using the following steps:
\begin{enumerate}
    \item We couple the JSQ system with a continuous-time random walk $(Q_{RW}^n(t):t\geq 0)$ in such a way that, by the time the random walk reaches a specified endpoint from a specified initial point, the right-hand side of \eqref{eq:route_to_queue_bound}, and hence $A_1^n(t)$, exceeds $n^{1/2}s$.
    
    \item We derive a bound on the expected hitting time of the endpoint from the initial point for the continuous-time random walk. This bound then enables us to obtain the desired bounds on the tail probabilities. 
\end{enumerate}

To establish the coupling in Step~1, consider a continuous-time random walk
$(Q_{RW}^n(t):t\geq 0)$ that, when in state $q$, moves to the right with rate
$n\lambda^n$ and to the left with state-dependent rate
\[
\mu_{RW}^{n}(q) =
\begin{cases}
    n\lambda^n+K_\epsilon n^{1/2}, & q>q_{RW}^n,\\[2mm]
    n\lambda^n-K_\epsilon n^{1/2}, & q\le q_{RW}^n,
\end{cases}
\]
where
$q_{RW}^n=-2(\mu_{\min})^{-1}K_\epsilon n^{1/2}$. We assume that the random walk starts at $Q_{RW}^n(0)=q_{RW}^n+1$. To obtain the desired coupling, we use the construction introduced in the proof of Proposition~\ref{lem:stochastic_boundedness}. The random walk moves right at the arrival epochs of $A^n(t)$ and moves left at the $j$th event epoch of Poisson process $S^n(t)$ with rate $n\lambda^n+K_{\mu, \epsilon}n^{1/2}$ if  
\begin{equation}
\tilde{U}_{D,j}^n \leq \frac{\mu_{RW}^{n}(Q_{RW}^n(\theta_{S,j}^n-))}{n\lambda^n+K_{\mu, \epsilon}n^{1/2}}\label{eq:departure_construction1}
\end{equation}
as in \eqref{eq:departure_construction} and denote the number of steps that the random walk takes towards left by time $t$  as $D_{RW}^n(t)$. Hence, 
\begin{align}
    Q_{RW}^n(t) & = q_{RW}^n + 1 + A^n(t) - D_{RW}^n(t),
    \label{eq:random_walk_balance}
\end{align}

Next, we define $\tau_s^{RW,n} = \inf\{t:Q_{RW}^n(t)\geq 4K_\epsilon n^{1/2}\}$. Then, for any $K>0$,
\begin{align}
   \nonumber \PP(\tau_s^n>T_{s,\epsilon}^\alpha) &= \PP(\tau_s^n>T_{s,\epsilon}^\alpha, \tau_s^{RW,n}<\tau_s^n) + \PP(\tau_s^{RW,n}\geq\tau_s^n> T_{s,\epsilon}^\alpha)\\
    &\leq \PP(\tau_s^n>T_{s,\epsilon}^\alpha, \tau_s^{RW,n}<\tau_s^n) + \PP(\tau_s^{RW,n}> T_{s,\epsilon}^\alpha). \label{eq:prob_bound_alpha}
\end{align}
We now prove that the first term on the right-hand side of \eqref{eq:prob_bound_alpha} is zero, which constitutes the first step we outlined above. Let $t_0^n=\sup\{t\leq \tau_s^{RW,n}:Q_{RW}^n(t) < 0\}$. All the potential departures in the original system over the interval $[t_0^n,\tau_s^{RW,n}]$ are realized as actual left steps in the random walk, and hence, $D_{RW}^n(\tau_s^{RW,n})-D_{RW}^n(t_0^n)\geq D_1^n(\tau_s^{RW,n}) -  D_1^n(t_0^n)$. Using \eqref{eq:random_walk_balance} and the fact that $Q_{RW}^n(t_0^n)=0$ by definition, we also have
\begin{align*}
    Q_{RW}^n(\tau_s^{RW,n}) &= A(\tau_s^{RW,n}) - A(t_0^n) - D_{RW}^n(\tau_s^{RW,n}) + D_{RW}^n(t_0^n) = 4K_{\epsilon}n^{1/2}.
\end{align*}
Then, using \eqref{eq:route_to_queue_bound}  along with $K_\epsilon>s$ imply that for the scenarios corresponding to the first event on the right-hand side of \eqref{eq:prob_bound_alpha}, we have
\begin{align}
Q_1^n(t_0^n)-n  \leq -3K_\epsilon n^{1/2} < Q_{RW}^n(t_0^n)=0.
\end{align}
The initial position of the random walk is chosen so that $Q_1^n(0)>Q_{RW}^n(0)$, and hence, there exists $t_1^n = \sup\{t\leq t_0^n: Q_1^n(t)-n>Q_{RW}^n(t)\}$. We note that the function $\mu_{RW}^n(q)$ is chosen so that the departure rate at the JSQ system when there are $(q)^-$ idle servers is less than $\mu_{RW}^n(q)$. We know that
\begin{align*}
   Q_{RW}^n(t_0^n) - (Q_1^n(t_0^n) - n) &= Q_{RW}^n(t_0^n)-Q_{RW}^n(t_1^n) - (Q_1^n(t_0^n) - n) - (Q_1^n(t_1^n)-n)\\
   3K_\epsilon n^{1/2}& \leq - D_{RW}^n(t_0^n) + D_{RW}^n(t_1^n) + A_1^n(t_0^n)-A_1^n(t_1^n) +D_1^n(t_0^n)-D_1^n(t_1^n)\\
   2K_\epsilon n^{1/2}& \leq - D_{RW}^n(t_0^n) + D_{RW}^n(t_1^n) +D_1^n(t_0^n)-D_1^n(t_1^n),\\
\end{align*}
where the first inequality follows by \eqref{eq:aggregate_balance} and \eqref{eq:random_walk_balance}. This is only possible if the number of departures in the JSQ system between $[t_1^n,t_0^n]$ is greater than the number of left steps in the random walk. However, we see that for any $t\in [t_1^n,t_0^n]$ 
\[
\sum_{k=1}^n \mu_k^n Q_{k,1}^n(\tilde{t}^n(j)-)\leq \sum_{k=1}^n \mu_k^n + \mu_{\min}(Q_1^n(t-)-n)\leq \mu_{RW}^n(Q_{RW}^n(t-)).
\]
Equations \eqref{eq:departure_construction} and \eqref{eq:departure_construction1} imply the number of departures in $[t_1,t_0]$ cannot exceed the number of left steps in the same interval, which is a contradiction. Hence, we conclude that the first probability on the right-hand side of \eqref{eq:prob_bound_alpha} is zero. 

As the second step, we bound the second term on the right-hand side of \eqref{eq:prob_bound_alpha} by obtaining a bound on $\EE[\tau_s^{RW,n}]$. We first apply uniformization to convert the continuous-time dynamics into an equivalent discrete-time formulation; see \cite{chen2001fundamentals} for further details. Let $P^n(t)$ be a Poisson process with total rate $\lambda_P^n = 2n\lambda^n + K_\epsilon n^{1/2}$, and interpret its event epochs $\tilde{t}^n(j)$ as the \emph{potential} event times for the random walk. At each epoch $\tilde{t}^n(j)$, the event corresponds to a right step with probability $p_A^n = n\lambda^n/\lambda_P^n$ and the random walk moves right. On the other hand, the epoch corresponds to a potential left step with probability $p_S^n = (n\lambda^n + K_\epsilon n^{1/2})/\lambda_P^n$ and it becomes an actual left step of the random walk with probability
\[
p_{RW,D}^n(j) = \frac{\mu_{RW}^n\left(Q_{RW}^n(\tilde{t}^n(j)-)\right)}{n\lambda^n + K_\epsilon n^{1/2}}.
\]
We denote the discretized version of the random walk defined as $Q_{RW}^{d,n}(j) = Q_{RW}^n(\tilde{t}^n(j))$ and define $j^{n,*} = \min\{j: Q_{RW}^{d,n}(j)\geq 4K_\epsilon n^{1/2}\}$. Using Wald's identity, we have
\begin{align}
    \EE[\tau_s^{RW,n}] = \frac{\EE[j^{n,*}]}{\lambda^n_P}.
    \label{eq:wald_tau}
\end{align}
To bound $\EE[j^{n,*}]$, we consider excursions above and below $q_{RW}^n$, referred to upper and lower excursions, respectively. To understand the methodology, consider the sample path representation provided in Figure~\ref{fg:sample_path_RW}. The black arrows represent upper excursions where the random walk moves over the interval $[q_{RW}^n + 1,\infty)$, i.e., the jumps out of states in this interval. Similarly, blue arrows represent lower excursions over $(-\infty, q_{RW}^n]$ indicating jumps out of states in this interval.

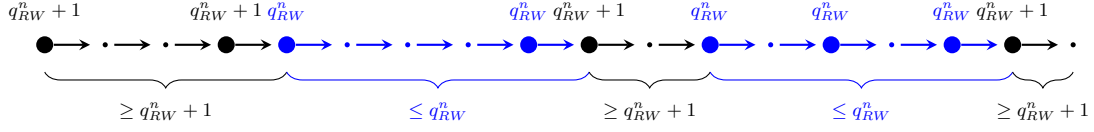
\begin{figure}[t]
\centering
\begin{tikzpicture}[
    scale=0.8,
    transform shape,
    >=stealth,
    every node/.style={font=\small}
]


\fill (0,0) circle (4pt);
\node[above=8pt] at (0,0) {$q_{RW}^{n}+1$};
\draw[->,thick] (0.15,0) -- (0.75,0);

\fill (1,0) circle (1.2pt);
\draw[->,thick] (1.15,0) -- (1.75,0);

\fill (2,0) circle (1.2pt);
\draw[->,thick] (2.15,0) -- (2.75,0);

\fill (3,0) circle (4pt);
\node[above=8pt] at (3,0) {$q_{RW}^{n}+1$};
\draw[->,thick] (3.15,0) -- (3.75,0);

\fill[blue] (4,0) circle (4pt);
\node[blue,above=8pt] at (4,0) {$q_{RW}^{n}$};

\draw[->,thick,blue] (4.15,0) -- (4.75,0);

\fill[blue] (5,0) circle (1.2pt);
\draw[->,thick,blue] (5.15,0) -- (5.75,0);

\fill[blue] (6,0) circle (1.2pt);
\draw[->,thick,blue] (6.15,0) -- (6.75,0);

\fill[blue] (7,0) circle (1.2pt);
\draw[->,thick,blue] (7.15,0) -- (7.75,0);

\fill[blue] (8,0) circle (4pt);
\node[blue,above=8pt] at (8,0) {$q_{RW}^{n}$};
\draw[->,thick,blue] (8.15,0) -- (8.75,0);

\fill (9,0) circle (4pt);
\node[above=8pt] at (9,0) {$q_{RW}^{n}+1$};

\draw[->,thick] (9.15,0) -- (9.75,0);

\fill (10,0) circle (1.2pt);
\draw[->,thick] (10.15,0) -- (10.75,0);

\fill[blue] (11,0) circle (4pt);
\node[blue,above=8pt] at (11,0) {$q_{RW}^{n}$};

\draw[->,thick,blue] (11.15,0) -- (11.75,0);

\fill[blue] (12,0) circle (1.2pt);
\draw[->,thick,blue] (12.15,0) -- (12.75,0);

\fill[blue] (13,0) circle (4pt);
\node[blue,above=8pt] at (13,0) {$q_{RW}^{n}$};

\draw[->,thick,blue] (13.15,0) -- (13.75,0);

\fill[blue] (14,0) circle (1.2pt);
\draw[->,thick,blue] (14.15,0) -- (14.75,0);

\fill[blue] (15,0) circle (4pt);
\node[blue,above=8pt] at (15,0) {$q_{RW}^{n}$};
\draw[->,thick,blue] (15.15,0) -- (15.75,0);

\fill (16,0) circle (4pt);
\node[above=8pt] at (16,0) {$q_{RW}^{n}+1$};

\draw[->,thick] (16.15,0) -- (16.75,0);
\fill (17,0) circle (1.2pt);


\draw[
decorate,
decoration={brace,mirror,amplitude=6pt}
]
(0,-0.45) -- (4,-0.45)
node[midway,yshift=-0.65cm]
{$\ge q_{RW}^{n}+1$};

\draw[
decorate,
decoration={brace,mirror,amplitude=6pt},
blue
]
(4,-0.45) -- (9,-0.45)
node[midway,yshift=-0.65cm]
{$\le q_{RW}^{n}$};

\draw[
decorate,
decoration={brace,mirror,amplitude=6pt}
]
(9,-0.45) -- (11,-0.45)
node[midway,yshift=-0.65cm]
{$\ge q_{RW}^{n}+1$};

\draw[
decorate,
decoration={brace,mirror,amplitude=6pt},
blue
]
(11,-0.45) -- (16,-0.45)
node[midway,yshift=-0.65cm]
{$\le q_{RW}^{n}$};

\draw[
decorate,
decoration={brace,mirror,amplitude=6pt}
]
(16,-0.45) -- (17,-0.45)
node[midway,yshift=-0.65cm]
{$\ge q_{RW}^{n}+1$};

\end{tikzpicture}

\caption{Sample Path Representation of the Random Walk}
\label{fg:sample_path_RW}
\end{figure}

The random walk starts at $Q_{RW}^{d,n}(0)=q_{RW}^n+1$ in an upper excursion and we denote $j_u^n(j)$ to be the cumulative time the random walk spends above $q_{RW}^n$ up to epoch $j$, i.e.,
\[
j_u^n(j) = \sum_{l=0}^j\II(Q_{RW}^n(l) > q_{RW}^n).
\]
At each visit of the random walk to $q_{RW}^n +1$, the random walk switches to a lower excursion with probability $(\lambda^n-K_\epsilon n^{1/2})/(2\lambda^n+K_\epsilon n^{1/2})$ or continues the upper excursion with probability $(\lambda^n+2K_\epsilon n^{1/2})/(2\lambda^n+K_\epsilon n^{1/2})$. Defining the number of completed lower excursions up to time $j$ as $m_l^n(j)$ and the length of the $m$th excursion as $L_m^n$, we can decompose $\EE[j^{n,*}]$ as
\begin{align}
    \EE[j^{n,*}] = \EE[j_u^n(j^{n,*})] + \EE\left[\sum_{m=1}^{m_l^n(j^{n,*})}L_m^n\right] = \EE[j_u^n(j^{n,*})] + \EE[m_l^n(j^{n,*})]\EE[L_m^n].
    \label{eq:excursion_decomposition}
\end{align}
In Figure~\ref{fg:sample_path_RW}, $\EE[j_u^n(j^{n,*})]$ and $\EE[m_l^n(j^{n,*})]\EE[L_m^n]$ correspond to the expected number of black dots and blue dots until the random walk hits $4K_\epsilon n^{1/2}$, respectively.

We first focus on bounding $\EE[j_u^n(j^{n,*})]$. After suppressing the lower excursions, i.e. deleting all the blue dots and arrows in Figure~\ref{fg:sample_path_RW}, the resulting time-changed path has the dynamics of a reflecting random walk on
$\{q:q\ge q_{RW}^n+1\}$, with reflection at $q_{RW}^n+1$. Equivalently, after shifting the state space by $q_{RW}^n+1$, this is a reflecting random walk on $\mathbb Z_+$ starting from $0$. Define $C_{RW}^n$ by
\[
C_{RW}^n n^{1/2}
=
\left\lceil 4K_\epsilon n^{1/2}-(q_{RW}^n+1)\right\rceil.
\]
Using the standard formula for the expected hitting time of a reflecting random walk (cf.~\cite{LevinPeres2017}, p.~26), we can find $K_u>0$ such that
\begin{align}
\nonumber \EE[j_u^n(j^{n,*})]
&=
\frac{1}{p_S^n-p_A^n}
\left(
p_S^n
\frac{
\left(p_S^n/p_A^n\right)^{C_{RW}^n n^{1/2}}-1
}
{p_S^n-p_A^n}
-
C_{RW}^n n^{1/2}
\right)\\
\nonumber &=
\frac{2n\lambda^n+K_\epsilon n^{1/2}}{K_\epsilon n^{1/2}}
\left(
\frac{n\lambda^n+K_\epsilon n^{1/2}}{K_\epsilon n^{1/2}}
\left[
\left(
1+\frac{K_\epsilon}{\lambda^n}n^{-1/2}
\right)^{C_{RW}^n n^{1/2}}
-1
\right]
-
C_{RW}^n n^{1/2}
\right)\\
&\le K_u n.
\label{eq:upper_excursion_bound}
\end{align}

Now we focus on the second term on the right-hand side of
\eqref{eq:excursion_decomposition}, beginning with a bound on
$\EE[m_l^n(j^{n,*})]$. The probability that an upper excursion,
starting from $q_{RW}^n+1$, reaches $4K_\epsilon n^{1/2}$ before
visiting $q_{RW}^n$ is given by the classical gambler's ruin formula (cf.~\cite{feller1968introduction}, page 345):
\[
p_{GR}
=
\frac{\frac{K_\epsilon}{\lambda^n n^{1/2}}}
{\left(
\frac{\lambda^n+K_\epsilon n^{-1/2}}
{\lambda^n}
\right)^{C_{RW}^n n^{1/2}}-1}.
\]
Hence, the number of lower excursions until the random walk successfully reaches at $4K_\epsilon n^{1/2}$ follow a geometric distribution. 
The denominator converges to a positive real number and hence, there exists a $\kappa_0>0$ such that $p_{GR}\geq \kappa_0 n^{-1/2}$. Hence, using the geometric trials argument, the expected number of completed lower excursions can be bounded as 
\begin{align}
\EE[m_l^n(j^{n,*})]\leq \frac{n^{1/2}}{\kappa_0}.
\label{eq:lower_excursion_number_bound}
\end{align}

Now we turn to the calculation of $\EE[L_m^n]$. As above, each time the process returns to $q_{RW}^n$ during a lower excursion, the excursion terminates with probability $p_A^n$. Hence, the number of visits to $q_{RW}^n$ during a lower excursion is geometric with parameter $p_A^n$ and  mean $1/p_A^n$.

After suppressing the upper excursions, i.e., deleting all black dots and arrows in Figure~\ref{fg:sample_path_RW}, the resulting time-changed process evolves as a random walk on $\{q:q\le q_{RW}^n\}$ with reflection at $q_{RW}^n$. The probability of an upward and downward jump from state $q$ is given by $p_A^n$ and $p_{RW,D}^n(q)$, respectively. This Markov chain is positive recurrent and has stationary probability
$K_\epsilon n^{-1/2}/\lambda^n$
at state $q_{RW}^n$. Using the fact that the mean return time to a state is the reciprocal of its stationary probability, the expected time between successive visits to $q_{RW}^n$ is $
\lambda^n n^{1/2}/K_\epsilon$. 
Multiplying by the expected number of visits and accounting for the initial visit to $q_{RW}^n$, we obtain
\begin{align}
\EE[L_m^n]=
\frac{\lambda^n n^{1/2}}{p_A^nK_\epsilon}+1 \mbox{ and }\EE[m_l^n(j^{n,*})]\EE[L_m^n]\leq \frac{n^{-1/2}}{\kappa_0}\left(\frac{\lambda^n n^{1/2}}{p_A^nK_\epsilon}+1\right).
\label{eq:lower_excusion_bound}
\end{align}
Combining \eqref{eq:upper_excursion_bound} and \eqref{eq:lower_excusion_bound}, we can find a constant $K_{exc}>0$ such that
\[
\EE[j^{n,*}]\leq K_{exc}n .
\]
Using \eqref{eq:wald_tau} and Assumption~\ref{asm:arrival_service_rates}, we can then find a constant $K_{\tau}>0$, independent of $n$, such that
\[
\EE[\tau_s^{RW,n}] \le K_{\tau}.
\]
By Markov's inequality,
\[
\PP(\tau_s^{RW,n}>T)\leq \frac{K_{\tau}}{T}.
\]
Choosing $T_\epsilon$ large enough so that $K_{\tau}/T_\epsilon<\epsilon/2$ proves the result.
\end{proof}

Having developed the necessary analytical tools, we are now ready to establish the corollaries stated in Section~\ref{sec:main_results}. The remainder of this section is devoted to their proofs.

\begin{proof}[Proof of Corollary~\ref{cor:SA-JSQ}]
    Using the tightness of $(n^{-1/2}\alpha_1^n(n^{1/2}\cdot):n\in \NN)$ in Lemma~\ref{lem:alpha_continuity}, the process level convergence $\alpha^n\toP\iota\times\delta_{\mu_{\max}}$ in the Skorokhod topology follows if we can show that the finite dimensional distributions converge or equivalently, $n^{-1/2}\alpha^n(n^{1/2}t)\toP\delta_{\mu_{\max}}t$ for all $t\geq 0$. For any $\epsilon>0$, under the speed-aware-JSQ, we only route incoming jobs to the servers with rates in set  $\AAA_\epsilon:=[\mu_{\min},\mu_{\max}-\epsilon]$ only when all servers in the set  $(\mu_{\max}-\epsilon, \mu_{\max}]$ have more jobs than servers in $\AAA_\epsilon$. Hence, for any $a>0$
    \begin{align*}
    \PP\left(n^{-1/2}\alpha^n(n^{1/2}t)(\AAA_\epsilon)>0\right)&\leq \PP\left(A_1^n(t)\geq \sum_{k=1}^n \delta_{\mu_k^n}(\mu_{\max}-\epsilon, \mu_{\max}] - X_2^n(0)\right)\\
    &\leq \PP\left(A_1^n(t)\geq an - X_2^n(0)\right) + \PP\left(\sum_{k=1}^n \delta_{\mu_k^n}(\mu_{\max}-\epsilon, \mu_{\max}]<an\right)
    \end{align*}
    Choosing $0<a<1-F(A_\epsilon)$ and using Assumptions~\ref{asm:arrival_service_rates} and ~\ref{asm:xi_initial} along with Lemma~\ref{lem:A_1_bound}, the right-hand side converges to 0. This implies that for any closed set $\AAA\not\ni\mu_{\max}$, $n^{-1/2}\alpha_1^n(n^{1/2}\cdot)(\AAA)\toP 0$. Then, the corollary follows using the portmanteau theorem (cf.~Theorem 2.1 in \cite{billingsley1999convergence}). 
\end{proof}

\begin{proof}[Proof of Corollary~\ref{cor:JSQ_random}]
We have to show that for any $f\in \Lipbar{\RR,L}(\SSS)$, $\langle f, n^{-1/2}\alpha^n(n^{1/2}\cdot)\rangle\Rightarrow \iota\times\langle f, F\rangle$ and as above, process level convergence follows if $\langle f, n^{-1/2}\alpha_1^n(n^{1/2}t)\rangle\toP t\langle f, F\rangle$ for all $t\geq 0$. 

Without loss of generality, we assume that the uniform selection of one of the shortest servers is implemented using an acceptance-rejection scheme as follows: We select a server uniformly at random independent of its current queue length. Then, we check if the server has the current minimum queue length and if so, we route the job to the server. If there are servers with shorter queues, we reject the selected server and repeat the process. Let $\tilde{\mu}_{1,j}^{n,0}$ be the first selected server in the scheme when routing the $j$th job to a server with one or more jobs and define $\alpha_1^{n,0}$ by replacing $\tilde{\mu}_{1,j}^n$ with $\tilde{\mu}_{1,j}^{n,0}$ in \eqref{eq:alpha_definition}. Then, for any $\epsilon>0$ and $t\geq 0$,
\begin{align*}
    \PP\left(\left|\langle f, n^{-1/2}\alpha_1^n(n^{1/2}t)\rangle- t\langle f, F\rangle\right|\geq \epsilon\right)&\leq \PP\left(n^{-1/2}\left|\langle f, \alpha^n(n^{1/2}t)\rangle- \langle f, \alpha_1^{n,0}(n^{1/2}t)\rangle\right|\geq \epsilon/3\right)\\
    &\quad + \PP\left(\left|\langle f, n^{-1/2}\alpha_1^{n,0}(n^{1/2}t)\rangle- t\langle f, F^n\rangle\right|\geq \epsilon/3\right)\\
    &\quad + \PP\left(t\left|\langle f, F^n\rangle- \langle f, F\rangle\right|\geq \epsilon/3\right).
\end{align*}
Since the first selected servers, $\tilde{\mu}_{1,j}^{n,0}$, are i.i.d. random variables with distribution $F^n$, the second term on the right-hand side converges to 0 as a direct consequence of the functional law of large numbers, and the third term converges to 0 by Assumption~\ref{asm:arrival_service_rates}. We can bound the first term for any $f\in \Lipbar{\RR,L}(\SSS)$ as :
\begin{align*}
&\PP\left(n^{-1/2}\left|\langle f, \alpha^n(n^{1/2}t)\rangle- \langle f, \alpha_1^{n,0}(n^{1/}t)\rangle\right|\geq \epsilon/3\right)\\
&\qquad\qquad\qquad \leq \PP\left(  n^{-1/2}L\sum_{j=1}^{\lceil n^{1/}t\rceil}\II(\tilde{\mu}_{i,j}^n\neq \tilde{\mu}_{i,j}^{n,0})\geq \epsilon/3, \tau_{t+1}^n\leq K_1, |X_2^n|_{K_1}^*\leq K_2\right) + \PP\left(\tau_{t+1}^n>K_1\right)\\
    &\qquad\qquad\qquad \leq \frac{3n^{-1/2}L\sum_{j=1}^{\lceil n^{1/}t\rceil}\PP\left(\tilde{\mu}_{i,j}^n\neq \tilde{\mu}_{i,j}^{n,0}|\tau_{t+1}^n\leq K_1, |X_2^n|_{K_1}^*\leq K_2\right)}{\epsilon}  + \PP\left(\tau_{t+1}^n>K_1\right)\\
    &\qquad\qquad\qquad \leq \frac{3n^{-1/2}LK_2(t+1)}{\epsilon}  + \PP\left(\tau_{t+1}^n>K_1\right),
\end{align*}
where the second inequality uses Markov's inequality and the third inequality follows from the observation that the conditional probabilities in the summation are bounded by $K_2n^{-1/2}$. Hence, using Lemma~\ref{lem:A_1_bound} and letting $K_1$ and then $n$ go to infinity, we can see that the right-hand side converges to 0. Hence, the result follows. 
\end{proof}

\begin{proof}[Proof of Corollary~\ref{cor:JSQ_random_pool}] Using Proposition 2 in~\cite{bhambay2025asymptotic} with $B=\infty$ therein, we know that equations~\eqref{eq:JSQ_random_pool1}-\eqref{eq:JSQ_random_pool3} possess a unique solution. Corollary~\ref{cor:JSQ_random} implies $n^{-1/2}\alpha^n(n^{1/}\cdot)\Rightarrow \iota(\cdot)\times \sum_{j=1}^m\delta_{\mu_j}$ and by inspection, we can see that $\xi_2 = \sum_{j=1}^m\delta_{\mu_j}X_{2,j}$ is the unique solution of \eqref{eq:JSQ_random1}-\eqref{eq:JSQ_random3}, which proves the corollary.  
\end{proof}
\section{Analysis of Stationary Probabilities}
\label{sec:stationary-probabilities}

In Section~\ref{sec:auxiliary_systems}, we introduced the JSQ-SSL system as an auxiliary system to derive upper bounds on the total system size and the total queue length. As our next step, we prove an analogue of Theorem~\ref{thm:stationary_bounds} for the JSQ-SSL system without the finite-buffer assumption, as stated in Proposition~\ref{prop:stationary_bound_JSQ_SSL}. The proof of the proposition uses the generator expansion approach of Stein's method as in \cite{Braverman2020} and \cite{bhambay2025asymptotic} and is deferred to Appendix~\ref{app:stationary_bound_JSQ_SSL}.

\begin{proposition}
    \label{prop:stationary_bound_JSQ_SSL}
    Suppose the assumptions of Theorem~\ref{thm:stationary_bounds} hold. Then, there exist constants $\tilde{C}_{1,\varepsilon}, \tilde{C}_{2,\varepsilon}$, $\tilde{C}_{3,\varepsilon}$ and $n_0$ such that for all $n\geq n_0$ 
    \begin{align}
        \label{eq:stationary_bound_JSQ_SSL1}&\EE\left[|\tilde{X}_{\infty,1}^n(\infty)|\mid \varsigma^n+\nu^n>\varepsilon\right]<\tilde{C}_{1,\varepsilon},\\
        \label{eq:stationary_bound_JSQ_SSL2}&\EE\left[|\tilde{X}_{\infty,2}^n(\infty)|\mid \varsigma^n+\nu^n>\varepsilon\right]<\tilde{C}_{2,\varepsilon}, \mbox{ and }\\
        \label{eq:stationary_bound_JSQ_SSL3}&\EE\left[|\tilde{X}_{\infty,i}^n(\infty)|\mid \varsigma^n+\nu^n>\varepsilon\right]<\tilde{C}_{3,\varepsilon}n^{-1/2} \mbox{ for all }i\geq 3.
    \end{align}
\end{proposition}

Although the JSQ-SSL system yields upper bounds for $X_{b,i}^{\pi,n}(\infty)$ for all $i\geq2$, this is insufficient for $X_{b,1}^{\pi,n}(\infty)$, since $X_{b,1}^{\pi,n}(\infty)\leq0$. To bound this quantity, we instead require a lower bound as provided in Lemma~\ref{lem:x_1_lower_bound}. The proof of the lemma uses coupling as in the proof of \ref{lem:A_1_bound} and is provided in Appendix~\ref{app:x1_stationary_bound}

\begin{lemma}
    \label{lem:x_1_lower_bound}
    Suppose the assumptions of Theorem~\ref{thm:stationary_bounds} hold. Then, there exists a $C_{1,\varepsilon}$ such that 
    \begin{equation}
        \EE\left[|X_{\infty,1}^{\pi,n} (\infty)||\mid \varsigma^n+\nu^n>\varepsilon\right]<C_{1,\varepsilon}.
    \end{equation}
\end{lemma}

\begin{proof}[Proof of Theorem~\ref{thm:stationary_bounds}]
Lemma\ref{lem:jsq_ssl_monotonicity} and Proposition~\ref{prop:stationary_bound_JSQ_SSL} along with the finite buffer $b$ assumption imply that
\begin{align*}
\EE[\tilde{X}_{b,1+}^n(\infty)]&\leq \tilde{C}_1 + \tilde{C}_2 + b\tilde{C}_3n^{-1/2}\\
\EE[\tilde{X}_{b,2+}^n(\infty)]&\leq  \tilde{C}_2 + b\tilde{C}_3n^{-1/2}, \mbox{ and }\\
\EE[\tilde{X}_{b,i+}^n(\infty)]&\leq  b\tilde{C}_3n^{-1/2}\mbox{ for all }i\geq 3.
\end{align*}
Hence, the result follows using Proposition~\ref{prop:upper_bound_system}.
\end{proof}
\begin{remark}
We use the finite-buffer assumption only in the proof of Theorem~\ref{thm:stationary_bounds}. In that proof, Proposition~\ref{prop:stationary_bound_JSQ_SSL} gives bounds on $\tilde{X}^n_{b,i}(\infty)$ for the JSQ--SSL system. To apply
the sample-path comparison in Proposition~\ref{prop:upper_bound_system}, we need bounds on the cumulative quantities
$\tilde{X}^n_{b,i+}(\infty)$. The finite-buffer assumption makes this conversion immediate. We do not use this
assumption in any other part of the paper.
\end{remark}

\begin{proof}[Proof of Corollary~\ref{cor:interchange_limits}]
Consider a tie-breaking rule $\pi$ and let $E_n:=\{\varsigma^n+\nu^n>\varepsilon\}$. Throughout the proof, we work under the conditional law $\mathbb{P}(\cdot\mid E_n)$. More precisely, we first condition on the service-rate vector $\boldsymbol{\mu}^n$. For every realization of $\boldsymbol{\mu}^n$ satisfying $E_n$, the corresponding queueing system is positive recurrent and therefore possesses a stationary distribution. We then aggregate these stationary distributions by integrating with respect to the conditional distribution of $\boldsymbol{\mu}^n$ given $E_n$. For all $i\geq 2$, we have
$\langle 1,\xi_i^{\pi,n}(\infty)\rangle=X_i^{\pi,n}(\infty)$. Since $\SSS$ is compact, tightness of the total masses $X_i^{\pi,n}(\infty)$ implies tightness of the corresponding $\xi_i^{\pi,n}(\infty)$ (cf. Theorem 4.10 in~\cite{kallenberg2017random}). Hence, \eqref{eq:stationary_bound1} and \eqref{eq:stationary_bound2} imply the
tightness of $(X_1^{\pi,n}(\infty),\xi_2^{\pi,n}(\infty))$.

Using the monotonicity of $X_i^{\pi,n}(\infty)$ for $i\geq 2$, we have
\[
d_{\infty,\rho}\bigl((\xi_i^{\pi,n}(\infty):i\ge3),0\bigr)
\le
\sum_{i=3}^\infty
\rho^{i-2}
\langle 1,\xi_i^{\pi,n}(\infty)\rangle
\le
\frac{\rho}{1-\rho}X_3^{\pi,n}(\infty).
\]
Hence, \eqref{eq:stationary_bound3} and Markov's inequality imply that $(\xi_i^{\pi,n}(\infty):i\geq 3)\toP 0$. Combining this with the tightness of $(X_1^{\pi,n}(\infty),\xi_2^{\pi,n}(\infty))$ yields the tightness of $(X_1^{\pi,n}(\infty),\mf{\xi}^{\pi,n}(\infty))$.

Now, suppose that along a subsequence $n_k\to\infty$, $(X_1^{\pi,n_k}(\infty),\mf{\xi}^{\pi,n_k}(\infty))\Rightarrow (X_1^\pi(0),\mf{\xi}^\pi(0))$. Consider the stationary version of the $n_k$th system. Then, for every $t\geq0$, $(X_1^{\pi,n_k}(t),\mf{\xi}^{\pi,n_k}(t))$ has the same distribution as $(X_1^{\pi,n_k}(\infty),\mf{\xi}^{\pi,n_k}(\infty))$. Applying the diffusion limit theorem along this subsequence gives $(X_1^{\pi,n_k},\mf{\xi}^{\pi,n_k})\Rightarrow (X_1^\pi,\mf{\xi}^\pi)$, where $(X_1^\pi,\mf{\xi}^\pi)$ solves \eqref{eq:main_diffusion_x1}--\eqref{eq:main_reflection} with initial condition $(X_1^\pi(0),\mf{\xi}^\pi(0))$. Therefore, for every $t\geq0$, $(X_1^\pi(t),\mf{\xi}^\pi(t))$ has the same distribution as $(X_1^\pi(0),\mf{\xi}^\pi(0))$. Hence, the law of $(X_1^\pi(0),\mf{\xi}^\pi(0))$ is a stationary distribution of the limiting
diffusion.

If the limiting diffusion has a unique stationary distribution, then every subsequential stationary limit must coincide with this distribution. Therefore, the stationary distributions of the pre-limit systems converge to the stationary distribution of the limiting diffusion, and the steady-state and many-server limits can be interchanged.
\end{proof}

We are now ready to prove our final result. The asymptotic optimality of the JFSQ policy is a direct consequence of Theorem~\ref{thm:stationary_bounds}.

\begin{proof}[Proof of Theorem~\ref{thm:asymptotical_optimality}]
We first prove the theorem for $b\geq 2$. Using the dominated convergence theorem, taking the limit as $t\to\infty$ in Proposition~\ref{prop:upper_bound_system} and Lemma~\ref{lem:jsq_ssl_monotonicity}, we get for any policy $\pi$ that yields a stationary system
\[
\tilde{Q}_{2,i+}^{\FSL,n}(\infty)
\leq_{st}
\tilde{Q}_{b,i+}^{\FSL,n}(\infty)
\leq_{st}
Q_{b,i+}^{\pi,n}(\infty)
\]
for any $b\geq 2$.

We next identify the limiting stationary distribution of the lower bound. Consider the stationary version of the JFSQ system. Corollary~\ref{cor:interchange_limits} implies that $((X_1^{\JFSQ,n}(\infty),\mf{\xi}^{\JFSQ,n}(\infty)),n\in\NN)$ is tight and $\xi_i^{\JFSQ,n} (\infty)\toP 0$ for all $i\geq 3$. Hence, along any convergent subsequence, the subsequential limit is a stationary distribution of the limiting diffusion in Corollary~\ref{cor:SA-JSQ}. Since the coordinates $i\geq 3$ vanish in the stationary limit, this limiting diffusion agrees with the diffusion limit of JSQ-FSL with $b=2$ provided in Proposition~\ref{prop:diffusion_limit_FSL_SSL}. Therefore, by the uniqueness of the stationary distribution of this limiting diffusion, $X_{+i}^{\JFSQ,n}(\infty)$ and $\tilde{X}_{2,+i}^{\FSL,n}(\infty)$ converge to the same limiting stationary distribution for $i=1,2$.

For positive levels $y$, the limiting stationary distributions above have no atoms. This follows from the Brownian noise in the $X_1$ equation and, for
$X_{+2}=X_2$ in the $b=2$ limiting system, from the representation
\[
X_2(t)=e^{-\mu_{\max}t}X_2(0)+\int_0^t e^{-\mu_{\max}(t-s)}\,dU_1(s).
\]
Hence, for all $y>0$,
\[
\lim_{n\to \infty} \PP(X_{+1}^{\JFSQ,n}(\infty)>y)
=
\lim_{n\to \infty} \PP(\tilde{X}_{2,+1}^{FSL,n}(\infty)>y).
\]
Combining this with the inequality obtained from Proposition~\ref{prop:upper_bound_system} and Lemma~\ref{lem:jsq_ssl_monotonicity} gives
\[
\lim_{n\to \infty} \PP(X_{+1}^{\JFSQ,n}(\infty)>y)
=
\lim_{n\to \infty} \PP(\tilde{X}_{2,+1}^{FSL,n}(\infty)>y)
\leq
\liminf_{n\to \infty} \PP(X_{+1}^{\pi,n}(\infty)>y)
\]
for all $\pi$. The proof for $X_{+2}^{\pi,n}(\infty)$ is the same.

The case $b=1$ is proven using the same steps after observing that both $X_1^{\JFSQ,n}$ and $\tilde{X}_{1,1}^{\FSL,n}$ converge to the unique solution of
the stochastic differential equation
\[
X_1(t) = X_1(0) - (\varsigma+\nu)t + \sqrt{2}W(t)
- \mu_{\min}\int_0^t X_1(s)ds - U_1(t),
\]
with $X_1(t)\leq 0$, $U_1(0)=0$, $U_1(t)$ nondecreasing, and $\int_0^t \II\bigl(X_1(s)<0\bigr)\,dU_1(s)=0$ for all $t\geq 0$. Therefore, the stationary distributions of $X_1^{\JFSQ,n}$ and $\tilde{X}_{1,1}^{\FSL,n}$ converge to the same stationary distribution of this limiting diffusion. The inequality obtained from Proposition~\ref{prop:upper_bound_system} and Lemma~\ref{lem:jsq_ssl_monotonicity} then gives the desired result for $b=1$.
\end{proof}

\section{Conclusion}
\label{sec:conclusion}

We studied load balancing in many-server systems with dedicated queues and individually heterogeneous service rates in the Halfin-Whitt regime. In contrast to pool-based heterogeneous models, the service rates in our setting may vary at the level of individual servers, and the empirical service-rate distribution converges weakly to a limiting law. This individual-level heterogeneity changes the structure of the state descriptor: aggregate occupancy counts alone do not determine the service capacity at each queue-length level. The dynamics must instead be described using measure-valued processes that record how queue lengths and idleness are distributed across service rates.

Our main result, Theorem~\ref{thm:process_level_convergence}, establishes a process-level diffusion framework for heterogeneous JSQ systems in which the effect of a tie-breaking rule is captured through its limiting fairness process and limiting routing measure. The limit is characterized as the unique solution of a reflected stochastic equation for the diffusion-scaled idle-server process, coupled with a family of measure-valued integral equations in weak form for the higher queue levels. A key structural feature of the limit is that the routing policy enters only through these two policy-dependent objects: the limiting fairness process, which describes the allocation of cumulative idleness across service rates, and the limiting routing measure, which describes the service-rate distribution of jobs routed to busy servers.

This separation gives a modular description of the diffusion limit. Once the limiting fairness process and routing measure of a candidate policy are characterized, the corresponding diffusion limit follows directly from Theorem~\ref{thm:process_level_convergence}, without requiring a separate convergence proof for each policy. This decoupling permits a mix-and-match treatment of policies that resolve ties differently among idle and busy servers. As applications, we derived explicit limits for JFSQ, for JSQ with totally blind tie-breaking rules, and for mixed policies such as fastest-idle routing with random tie-breaking when all servers are busy. When the limiting service-rate distribution has finite support, the measure-valued equations recover the finite-pool formulation, thereby connecting the present individual-level model with the pool-based setting of~\cite{bhambay2025asymptotic}.

For the stationary analysis, we developed policy-uniform stochastic bounds using coupled fastest-serves-longest (JSQ-FSL) and slowest-serves-longest (JSQ-SSL) auxiliary systems. These bounds show that, uniformly over admissible tie-breaking rules, the diffusion-scaled number of idle servers and the diffusion-scaled number of queues with at least two jobs are \(O(1)\) in expectation in stationarity, while the contribution of queues with length at least three is asymptotically negligible. The same comparison construction also yields the asymptotic optimality of JFSQ in steady state, as stated in Theorem~\ref{thm:asymptotical_optimality}: the stationary diffusion-scaled total number of jobs and the stationary diffusion-scaled number of waiting jobs under JFSQ are stochastically dominated, in the many-server limit, by those under any admissible tie-breaking policy.

Two methodological tools underlie these results. The first is the coupling-based sample-path comparison with the modified JSQ-FSL and JSQ-SSL systems, which provides stochastic lower and upper bounds at the level of each fixed \(n\) and does not depend on the Halfin-Whitt scaling. The second is a contraction argument for the measure-valued equations based on a weighted Kantorovich-Rubinstein norm. This norm switch allows us to work directly with the measure-valued system and avoids the auxiliary finite-dimensional construction used in the homogeneous JSQ analysis of~\cite{eschenfeldt2018join}.

Several questions remain open. Theorem~\ref{thm:asymptotical_optimality} establishes asymptotic optimality for the stationary distributions of the diffusion-scaled system size and total queue length.. A related question is whether similar asymptotic optimality results can be established for the mean response time via Little's law. This would first require establishing a uniform integrability result, which in turn would require tighter control of the convergence rates of the stationary queue lengths beyond the bounds established in this work. A natural next problem is to obtain explicit convergence rates for the heterogeneous systems, analogous to the \(O(n^{-1/2})\) estimates available for the homogeneous JSQ system in~\cite{BraSSY2023}.

A second direction concerns richer forms of heterogeneity. The present framework treats the service rate as a property of the server alone. In settings where service rates depend jointly on the server and the job, such as multiplicative job-server interactions, a scalar service-rate descriptor may no longer be sufficient. Such models appear to require measure-valued state descriptors on a product space, and extending the present analysis to that setting remains open.

Finally, the decoupling of the fairness process and the routing measure suggests possible extensions beyond exact JSQ information. One important class consists of sampling-based policies, such as power-of-\(d\) choices, where routing decisions are based only on a subset of servers rather than the full queue-length and service-rate configuration. Characterizing the corresponding limiting fairness process and routing measure would broaden the reach of the framework. Incorporating dynamically varying service rates or strategic behavior of servers, in the spirit of~\cite{BdRP2025}, provides another direction for future work.
\newpage
\appendix
\section*{Appendices}
The following appendices contain proofs for results presented in the main body of the paper. To simplify the notation, we suppress the superscript $\pi$ when the tie-breaking policy need not be emphasized.

\section{Proofs of Results Presented in Section~\ref{sec:model}}
\label{app:model_basics}

\subsection{Proof of Lemma~\ref{lem:stochastic_boundedness}}\label{app:proof_stochastic_boundedness}
\begin{proof}[Proof of Lemma~\ref{lem:stochastic_boundedness}]

Stochastic boundedness of $\mf{X}^{n}$ follows if, for any given $T>0$ and $\varepsilon>0$, there exists a constant $K_{T,\varepsilon}>0$ such that
\begin{align*}
    \PP\left(|\mf{X}^{n}|_{\rho, T}^*>K_{T,\varepsilon}\right)\leq \varepsilon.
\end{align*}
By definition, we have $X_2^{n}(t)\geq X_3^{n}(t) \geq \cdots$  for all $t\geq 0$. Setting $c=\max\{\rho, \rho^2/(1-\rho)\}$, we have
\begin{align*}
    \PP\left(|\mf{X}^{n}|_{\rho, T}^*>K_{T,\varepsilon}\right)&\leq \PP\left(\rho|X_1^{n}|_{T}^* + |X_2^{n}|_{T}^*\sum_{i=2}^\infty \rho^i >K_{T,\varepsilon}\right)\leq \sum_{i=1}^2\PP\left(|X_i^{n}|_{T}^*  >\frac{K_{T,\varepsilon}}{2c}\right).
\end{align*}
Hence, it is enough to prove that $(X_1^{n}:n\in \NN)$ and $(X_2^{n}:n\in \NN)$ are stochastically bounded. Using Assumptions~\ref{asm:arrival_service_rates} and \ref{asm:xi_initial}, we can find a $K_{\mu,\epsilon}$ and a $K_0$ such that for $i=1,2$
\begin{align*}
    \PP\left(|X_i^{n}|_T^*  >K_{T,\varepsilon}\right) &\leq \PP\left(|X_i^{n}|_T^*  >K_{T,\varepsilon}, \bar{\mu}^n\leq \lambda^n+K_{\mu, \epsilon}n^{-1/2}, |X_i^n(0)|<K_0\right)\\&\quad + \PP\left(\bar{\mu}^n>\lambda^n+K_{\mu, \epsilon}n^{-1/2}\right) + \PP(|X_i^n(0)|\geq K_0)\\
    &\leq \PP\left(|X_i^{n}|_T^*  >K_{T,\varepsilon}, \bar{\mu}^n\leq \lambda^n+K_{\mu, \epsilon}n^{-1/2}, |X_i^n(0)|<K_0\right) + \frac{\varepsilon}{4}.
\end{align*}
Without loss of generality, we assume that $\bar{\mu}^n \leq \lambda^n + K_{\mu,\epsilon} n^{-1/2}$ and $|X_i^n(0)| < K_0$ for $i=1,2$ hold on all events considered below.

To prove the stochastic boundedness of $(X_1^n, n\in \NN)$, we first couple the number of idle servers in the original JSQ system with the queue length of a Markovian queue as follows. Take $A^n(t)$ as the Poisson arrival process with rate $n\lambda^n$ and the routing policy as described in Section~\ref{sec:model}. To rigorously construct the departure process of the JSQ system, let $S^n(t)$ be a Poisson process with rate $n\lambda^n+K_{\mu, \epsilon}n^{1/2}$ to act as the potential departure process and $(\tilde{U}_{D,j}^n, j\in \NN)$  be a sequence of i.i.d. uniform(0,1) random variables. The potential departure occurring at the $j$th event epoch $\theta_{S,j}^n$ of $S^n(t)$ is observed as an actual departure in the JSQ system, if 
\begin{equation}
\tilde{U}_{D,j}^n \leq \frac{\sum_{k=1}^n\mu_k^nQ_{k,1}^n(\theta_{S,j}^n-)}{n\lambda^n+K_{\mu, \epsilon}n^{1/2}}.\label{eq:departure_construction}
\end{equation}
The server $s_j^n$ from which this departure is observed is determined by sampling from the distribution
\begin{equation}
\PP(s_j^n=k) = \frac{\mu_k^nQ_{k,1}^n(\theta_{S,j}^n-)}{\sum_{l=1}^n\mu_l^nQ_{l,1}^n(\theta_{S,j}^n-)}.
\label{eq:splitting_departures}
\end{equation}
Using the splitting property of Poisson processes, it can be seen that the departure from each server $k$ occurs with rate function $\mu_k^nQ_{k,1}^n(t-)$ in this setting, matching the dynamics of the JSQ system.

Now consider a Markovian queueing process $\tilde{Q}^{U,n}$, where arrivals occur at the event epochs of $S^n$ and departures occur at the event epochs of $A^n$ only if there are jobs in this new queueing system. Suppose $\tilde{Q}^{U,n}(0) = n - Q_1^n(0)$, and there exists a first time $\tilde{t}$ such that $\tilde{Q}^{U,n}(\tilde{t}) < n - Q_1^n(\tilde{t})$, which would imply $\tilde{Q}^{U,n}(\tilde{t}-) = n - Q_1^n(\tilde{t}-)$. Since all potential departures in the original JSQ system are observed as arrivals in the Markovian queue $\tilde{Q}^{U,n}$, the event occurring at time $\tilde{t}$ can only be an arrival in the original system that is not observed as a departure in the Markovian queue. However, this would mean $\tilde{Q}^{U,n}(\tilde{t}-) = n - Q_1^n(\tilde{t}-) = 0$ and $Q_1^n(\tilde{t}) = n + 1$, which is not possible. Hence, such a time $\tilde{t}$ cannot exist. We conclude that $n^{-1/2}\tilde{Q}^{U,n}(t) \geq_{st} |X_1^n(t)|$ for all $t\geq 0$ and hence $X_1^n$ is stochastically bounded if $n^{-1/2}\tilde{Q}^{U,n}$ is stochastically bounded.

For any $K_{T,\varepsilon}>2K_0$, define the random times
\begin{align*}
    \theta_{Q,1}^1 &= \inf\{t\geq 0: n^{-1/2}\tilde{Q}^{U,n}(t)>K_{T,\varepsilon}\}\wedge T\\
    \theta_{Q,1}^2 &= \sup\{0\leq t\leq \theta_{Q,1}^1:n^{-1/2}\tilde{Q}^{U,n}(t)<K_{T,\varepsilon}/2\}.
\end{align*}
The definition implies that all potential departures (arrivals for the original process) in the interval $[\theta_{Q,1}^2, \theta_{Q,1}^1]$ are observed as actual departures. Hence, 
\begin{align*}
    \PP\left(|X_1^n(t)|_T^*  >K_{T,\varepsilon}\right) &\leq \PP(|n^{-1/2}\tilde{Q}^{U,n}(t)|_T^*  >K_{T,\varepsilon})\\
    &=\PP\left(\frac{S^n(\theta_{Q,1}^1) - S^n(\theta_{Q,1}^2) - A^n(\theta_{Q,1}^1) + A^n(\theta_{Q,1}^2) }{n^{1/2}} >\frac{K_{T,\varepsilon}}{2}\right)\\
    &\leq \PP\left(2\frac{|S^n(t)-A^n(t)|_T^*}{n^{1/2}} >\frac{K_{T,\varepsilon}}{2}\right)\\
    &\leq \PP\left(\frac{|S^n(t)-(n\lambda^n+K_{\mu, \epsilon}n^{1/2})t|_T^*}{n^{1/2}} + \frac{|A^n(t)-n\lambda^n t|_T^*}{n^{1/2}}>\frac{K_{T,\varepsilon}}{2} - K_{\mu,\epsilon}\right)
\end{align*}
Focusing on the terms within the probability, we observe that the martingale central limit theorem implies that the terms on the right-hand side are stochastically bounded, and Assumption~\ref{asm:arrival_service_rates} ensures that the term on the right-hand side converges to a finite value. Hence, it is possible to choose $K_{T,\varepsilon}$ sufficiently large so that the probability is less than $\varepsilon$, proving the stochastic boundedness of $(X_1^n: n \in \NN)$.

To prove the stochastic boundedness of $(X_2^n: n \in \NN)$, define the set $\KK_n = \{k : Q_{k,2}^n(0) = 1\}$ as the set of servers that have two or more jobs at time $0$, and let $\KK_n^c$ be its complement. Define $A_{\KK_n^c}^n(t)$ and $D_{\KK_n^c}^n(t)$ to be the number of arrivals and departures by time $t\geq 0$, respectively, for the servers not in $\KK_n$. 
For any $T>0$, we have
\begin{align}
   \nonumber \PP\left(|X_2^n(t)|_T^*  >K_{T,\varepsilon}\right) &= \PP\left(\left|n^{-1/2}\sum_{k\in \KK_n}Q_{k,2}^n(t) + n^{-1/2}\sum_{k\in \KK_n^c}Q_{k,2}^n(t)\right|_T^*>K_{T,\varepsilon}\right)\\
   \nonumber &\leq \PP\left(X_2^n(0) + \left|\frac{(A_{\KK_n^c}^n(t) - D_{\KK_n^c}^n(t))^+}{n^{1/2}}\right|_T^*>K_{T,\varepsilon}\right)\\
   \nonumber &\leq \PP\left(X_2^n(0) + \left|\frac{(A^n(t) - D_{\KK_n^c}^n(t))^+}{n^{1/2}}\right|_T^*>K_{T,\varepsilon}\right)\\
   \nonumber&\leq \PP\left((1+\mu_{\max}T)X_2^n(0) + n^{1/2}|\lambda^n-\bar{\mu}^n|T+ \mu_{\max}T|X_1^n(t)|_T^* \right.\\
    \label{eq:x2_boundedness} &\left.\qquad\qquad+  \left|\frac{A^n(t)-n\lambda^nt}{n^{1/2}}\right|_T^* + \left|\frac{D_{\KK_n^c}^n(t) - \sum_{k\in \KK_n^c}\mu_k^n\int_0^tQ_{k,1}^n(s)ds}{n^{1/2}}\right|_T^*>K_{T,\varepsilon}\right).
\end{align}
Assumptions~\ref{asm:arrival_service_rates} and \ref{asm:xi_initial}, stochastic boundedness of $X_1^n(t)$ and the martingale central limit theorem ensure that the random quantity inside the probability on the right-hand side of \eqref{eq:x2_boundedness} is stochastically bounded, which in turn implies the stochastic boundedness of $(X_2^n:n\in \NN)$.
\end{proof}

\subsection{Proofs of Tightness of Fairness Processes}\label{app:fairness_measure_results}

The proof of Lemma~\ref{lem:fairness_tightness} relies on Jakubowski's criteria~\cite{Jakubowski1986} to establish the tightness of the shifted fairness processes, following the approach of B\"uke and Qin~\cite{buke2023many}.  Let $\cP(\bar{\SSS})$ denote the space of probability measures on the compact space $\bar{\SSS}$ endowed with the Prokhorov metric (cf. Section 1.6 in \cite{billingsley1999convergence}).

\begin{theorem}[Jakubowski's criteria, Theorem 4.6 in \cite{Jakubowski1986}]
\label{thm:jakubowski}
The sequence of probability measure-valued processes $(\eta^n:n\in\NN)$ with sample paths in $\DD_{\cP(E)}[0,T]$, where $E$ is a metric space, is tight provided that
\begin{itemize}
\item[(J1)] for every $\bar{\rho}>0$, there exists a compact set $\cK_{\bar{\rho}}\subset\cP(E)$ such that
$$\liminf_{n\to \infty}\PP(\eta^n(t)\in\cK_{\bar{\rho}}\ \text{for all }t\in[0,T])\ge 1-\bar{\rho}.$$
\item[(J2)] there exists a family $\cF$ of real-valued maps on $\cP(E)$ that separates points, is closed under addition, and such that $(H(\eta^n):n\in\NN)$ is tight in $\DD_\RR[0,T]$ for every $H\in\cF$.
\end{itemize}
\end{theorem}

As the first step, we prove the tightness of shifted fairness processes for any $\epsilon>0$.

\begin{proposition}\label{prop:shifted_fairness_tight}
For any fixed $\epsilon>0$ and $T>0$, the sequence $( \cS_\epsilon\eta^n:n\in\NN)$ is tight in $\DD_{\cP(\bar{\SSS})}[0,T]$.
\end{proposition}

\begin{proof}
Since $\bar{\SSS}$ is compact, so is $\cP(\bar{\SSS})$, and (J1) holds trivially with $\cK_{\bar{\rho}}=\cP(\bar{\SSS})$. To show that condition (J2) holds, we take $$\cF:=\{H_f:H_f(\eta)=\langle f,\eta\rangle, \mbox{ for all }\eta \in \cP(\bar{\SSS}), f\in C_b(\bar{\SSS})\}$$ 
as in \cite{buke2023many}, which separates points of $\cP(\bar{\SSS})$ and is closed under addition. Fixing $f\in \CC_{\RR}^b(\bar{\SSS})$, the tightness of $(H_f(\cS_\epsilon\eta^n):n\in\NN)$ follows if we can show 
\begin{enumerate}
    \item (compact containment) for any $\bar{\rho}>0$, there exists a $b>0$ such that $$\PP\big(|H_f(\cS_\epsilon\eta^n)|_T^*>b\big)<\bar{\rho} \mbox{ for all }n\in \NN.$$
    \item (modulus of continuity) for any $\bar{\rho},\varepsilon>0$, there exists a $\varrho^*$ such that $\varrho<\varrho^*$ implies
    \[
    \limsup_{n\to\infty}
\PP\left(
w'_T\bigl(H_f(\cS_\epsilon\eta^n),\varrho\bigr)>\varepsilon
\right)
<\bar{\rho},
    \]
where 
$w'\big(H_f(\cS_\epsilon \eta_t^n), \varrho\big)
:=
\inf
\max_{i}
w\big(H_f(\cS_\epsilon \eta_t^n), (t_i,t_{i+1}]\big)$ and 
\[
w\big(H_f(\cS_\epsilon \eta_t^n), (t_i,t_{i+1}]\big)
:=
\sup_{t_i < s,t \le t_{i+1}}
\left|
H_f(\cS_\epsilon \eta_t^n)
-
H_f(\cS_\epsilon \eta_s^n)
\right|.
\]
The infimum is taken over all finite partitions $\{ t_i\}_{i\in\NN}$ of $[0,T]$ such that $|t_{i+1} - t_i| \ge \varrho$ for all $i$. 
\end{enumerate}

We have $|H_f(\eta^n)|_T^*= |\langle f,\cS_\epsilon\eta^n\rangle|_T^*\leq \sup_{\mu\in \bar{\SSS}}|f(\mu)|$ and hence, the compact containment condition holds. For the modulus of continuity condition, choose $K_{\bar{\rho}}>0$ such that $\PP\left( |X_1^n|_T^*> K_{\bar{\rho}}\right)<\bar{\rho}/2$ and 
\begin{align*}
 \PP\left(w'\big(H_f(\cS_\epsilon \eta_t^n), \varrho\big)>\varepsilon\right) & \leq  \PP\left(w'\big(H_f(\cS_\epsilon \eta_t^n), \varrho\big)>\varepsilon, |X_1^n|_T^*\leq K_{\bar{\rho}}\right) + \PP\left( |X_1^n|_T^*> K_{\bar{\rho}}\right)\\
    &\leq \PP\left(w'\big(H_f(\cS_\epsilon \eta_t^n), \varrho\big)>\varepsilon, |X_1^n|_T^*\leq K_{\bar{\rho}}\right) +\bar{\rho}/2.
\end{align*}
Hence, without loss of generality we can assume that $|X_1^n|_T^*\leq K_{\bar{\rho}}$, which implies $\tau_\epsilon^n\geq \epsilon/K_{\bar{\rho}}$. For any $t>s>\tau_\epsilon^n$, we have 
\begin{align*}
|\langle f,\eta^n(t)\rangle-\langle f,\eta^n(s)\rangle| &=\left|
\frac{\sum_{k=1}^n f(\mu_k^n)\int_s^t X_{k,1}^n(r) dr}{\int_0^t X_1^n(r) dr} -
\frac{\left(\sum_{k=1}^n f(\mu_k^n)\int_0^s X_{k,1}^n(r) dr\right)\left(\int_s^t X_1^n(r) dr\right)}{\left(\int_0^t X_1^n(r) dr\right)\left(\int_0^s X_1^n(r) dr\right)}\right|\\
&\leq 2\sup_{\mu\in \bar{\SSS}}|f(\mu)|K_{\bar{\rho}}\epsilon^{-1}(t-s).
\end{align*}
Choosing $\varrho^*$ such that
\[
\varrho^*<\frac{\epsilon}{K_{\bar{\rho}}}\min\left\{1, \frac{\varepsilon}{2\sup_{\mu\in \bar{\SSS}}|f(\mu)|}\right\}
\]
and letting $t_1=0, t_2 = \tau_\epsilon^n$ and $t_i = t_2 + (i-2)\varrho^*$, for any $\varrho<\varrho^*$, we have $w'_T\bigl(H_f(\cS_\epsilon\eta^n),\varrho\bigr)<\varepsilon$. This establishes the modulus-of-continuity condition. Together with the compact containment condition, it follows that $(H_f(\cS_\epsilon\eta^n):n\in\NN)$ is tight.
\end{proof}

\begin{proof}[Proof of Lemma~\ref{lem:fairness_tightness}]
 By Proposition~\ref{prop:shifted_fairness_tight}, for each $m\in\NN$, the sequence $(\cS_{1/m}\eta^{n}: n\in \NN)$ is tight in $\DD_{\cP(\bar{\SSS})}[0,T]$. Using a diagonal argument, for any subsequence, we can find a further subsequence, $\{n_j,j\in \NN\}$, along which $\cS_{1/m}\eta^{n_j} \Rightarrow \cS_{1/m}\eta$ for all $m\in \NN$. Using the Skorokhod representation theorem (cf. Section~1.6 in
\cite{billingsley1999convergence}), we can assume that the
convergence is with probability 1.

Since $\eta^{n_j}(t)\neq \delta_0$ for all
$t\ge \tau_\epsilon^{n_j}$, the shift times are determined by the
shifted paths via
\[
\tau_\epsilon^{n_j}
=
\inf\{t:\cS_\epsilon\eta^{n_j}(t)\neq \delta_0\}.
\]
Therefore, the limiting shift times are identified by the limiting
shifted processes. 
 For any $n_j$ and $l>m$, we have $\tau_{1/l}^{n_j}\le \tau_{1/m}^{n_j}$. Passing to the limit yields $\tau_{1/l}\le \tau_{1/m}$ almost surely. Moreover, for $l>m$, we have
$\cS_{1/l}\eta^{n_j}(t)=\cS_{1/m}\eta^{n_j}(t)$ for all
$t\ge \tau_{1/m}^{n_j}$. Passing to the limit yields
$\cS_{1/l}\eta(t)=\cS_{1/m}\eta(t)$ for all
$t\ge \tau_{1/m}$. Therefore, the family
$(\cS_{1/m}\eta:m\in\NN)$ is consistent and defines a unique
limiting fairness process $\eta$. Hence, $(\tau_{1/m})_{m\in\NN}$ is nonincreasing and converges almost surely to a random variable $\tau_0$. Hence, $\eta^{n_j}\tof \eta$ as $j\to \infty$ and the result follows.
\end{proof}

\subsection{Proof of Lemma~\ref{lem:positive_recurrent_finite}}
\label{app:positive_recurrent_finite}
\begin{proof}[Proof of Lemma~\ref{lem:positive_recurrent_finite}]

To prove positive recurrence, we use the Foster--Lyapunov criterion by showing that there exist a Lyapunov function $V:\ZZ_+^\infty\to \RR_+$ satisfying
\[
V(\mf q)\to\infty
\qquad\mbox{as }\mf q\to\infty,
\]
a finite set $\KK\subset\ZZ_+^\infty$, and constants $\epsilon,K>0$ such that
\[
\cG^nV(\mf q)
<
-\epsilon
+
K\II(\mf q\in\KK).
\]
    For any $\mf{q}$, we define 
    \begin{align*}
        i^*(\mf{q}) := \max\{i:q_i>0\},
        i_*(\mf{q}) := \min\{i: q_i<n\}
    \end{align*} 
   with the convention that $i^*(\mf{0})=0$, and 
    \[
    V(\mf{q}) = \sum_{i=1}^\infty i q_i.
    \]
    When a JSQ system is in state $\mf{q}$, any arrival will be routed to a buffer with $i_*(\mf{q})-1$ jobs and hence yield an increase of $i_*(\mf{q})$ in the Lyapunov function $V(\mf{q})$. Similarly, any departure from the system yields a decrease of at least $i_*(\mf{q})-1$. Hence, we have 
\begin{align}
\nonumber
\cG^n V(\mf{q})
&\leq n\lambda^n i_*(\mf{q})
-
\sum_{k=1}^n \mu_k^n (i_*(\mf{q}) - 1)\\
\label{eq:lyapunov_eq1}
&=
-n^{1/2}(\nu^n+\varsigma^n)i_*(\mf{q})
+
\sum_{k=1}^n\mu_k^n .
\end{align}
    Moreover, the Lyapunov function decreases by $i^*(\mf{q})$ at rate at least $\mu_{\min}$. Hence, we also have 
    \begin{align}
        \label{eq:lyapunov_eq2}\cG^n V(\mf{q}) &\leq n\lambda^n i_*(\mf{q}) - \mu_{\min} i^*(\mf{q}).
    \end{align}
    Now, fix $\epsilon>0$ and consider the finite set 
\[
\KK = \left\{\mf q:i^*(\mf q)\leq\frac{
n\lambda^n\left(\sum_{k=1}^n\mu_k^n+\epsilon\right)}
{\mu_{\min}n^{1/2}(\nu^n+\varsigma^n)}
+\frac{\epsilon}{\mu_{\min}}
\right\}.
\]
We know that for any $\mf{q}\in \KK$, 
\[
\cG^n V(\mf{q}) \leq K = n\lambda^n
\left(
\frac{
n\lambda^n\left(\sum_{k=1}^n\mu_k^n+\epsilon\right)}
{\mu_{\min}n^{1/2}(\nu^n+\varsigma^n)}
+
\frac{\epsilon}{\mu_{\min}}
+1
\right)<\infty.
\]
For any $\mf{q}\notin \KK$, if
\[
i_*(\mf{q})>
\frac{\sum_{k=1}^n\mu_k^n + \epsilon}{n^{1/2}(\nu^n +\varsigma^n)},
\]
then \eqref{eq:lyapunov_eq1} implies that $\cG^nV(\mf{q})<-\epsilon$. On the other hand, if
\[
i_*(\mf{q})\leq
\frac{\sum_{k=1}^n\mu_k^n + \epsilon}{n^{1/2}(\nu^n+\varsigma^n)},
\]
then, since $\mf{q}\notin\KK$, \eqref{eq:lyapunov_eq2} implies that $\cG^n V(\mf{q})<-\epsilon$. Hence, the result follows.  
\end{proof}

\section{Proofs of Results Presented in Section~\ref{sec:auxiliary_systems}}
\label{app:jsq_ssl_proofs}

\subsection{Proof of Proposition~\ref{prop:diffusion_limit_FSL_SSL}} \label{sec:proof_JSQSSL}
\begin{proof}[Proof of Proposition~\ref{prop:diffusion_limit_FSL_SSL}]The proof of Proposition~\ref{prop:diffusion_limit_FSL_SSL} follows the same steps as in the proof of Theorem~\ref{thm:process_level_convergence}. We provide the proof for JSQ-SSL and the proof for JSQ-FSL follows the same lines. To simplify the notation, we suppress the superscript $\SSL$ below.

As in Theorem~\ref{thm:process_level_convergence}, we again start by proving the stochastic boundedness of the processes under consideration. The proof of Lemma~\ref{lem:stochastic_boundedness_jsq_ssl} is analogous to the proof of Lemma~\ref{lem:stochastic_boundedness} and hence is omitted.

\begin{lemma}
    \label{lem:stochastic_boundedness_jsq_ssl}
    The sequence of processes $(\mf{\tX}_b^{\SSL,n}(t):n\in \NN)$ is stochastically bounded under any tie-breaking rule $\pi$.
\end{lemma}

Adding and subtracting the necessary terms to \eqref{eq:modified_jssq_balance}, we get
\begin{align}
\label{eq:x_1_jsqssl}
   \tX_1^n(t)  &= \tX_1^n(0) - n^{-1/2}\tA_1^n(t)  + \tilde{Y}_1^n(t) - \mu_{\max}\int_0^t \tX_1^n(s)ds  + \mu_{\min}\int_0^t\tX_2^n(s)ds\\
   \label{eq:x_2_jsqssl}
    \tX_2^n(t) &= \tX_2^n(0) + n^{-1/2}\tA_1^n(t)  + \tilde{Y}_2^n(t) - \mu_{\min}\int_0^t\tX_2^n(s)ds  + \mu_{\min}\int_0^t\tX_3^n(s)ds,\\
    \label{eq:x_i_jsqssl}
    \tX_i^n(t) &= \tX_i^n(0) + \tilde{Y}_i^n(t) - \mu_{\min}\int_0^t\tX_i^n(s)ds   + \mu_{\min}\int_0^t\tX_{i+1}^n(s)ds, \mbox{ for all }i\geq 3,
\end{align}
where 
\begin{align*}
    \tilde{Y}_1^n(t) &= \tM_A^n(t) - \tM_{D,1}^n(t) + \tM_{D,2}^n(t) +  n^{-1/2}\left(n\lambda^nt  - \sum_{j=1}^n\mu_j^nt\right)\\
    &\quad + \mu_{\max}\int_0^t \tX_1^n(s)ds + n^{-1/2}\sum_{j=1}^n\mu_j^n\left(t - \int_0^t \II(\tQ_1^n(s)\geq j)ds\right)\\
    &\quad + n^{-1/2}\sum_{j=1}^n\mu_j^n\int_0^t \II(\tQ_2^n(s)\geq j)ds - \mu_{\min}\int_0^t\tX_2^n(s)ds,\\
    \tilde{Y}_2^n(t) & = - n^{-1/2}\tA_2^n(t) - \tM_{D,2}^n(t) + \tM_{D,3}^n(t)\\
    &\quad - n^{-1/2}\sum_{j=1}^n\mu_j^n\int_0^t \II(\tQ_2^n(s)\geq j)ds  + \mu_{\min}\int_0^t\tX_2^n(s)ds\\
    &\quad + n^{-1/2}\sum_{j=1}^n\mu_j^n\int_0^t \II(\tQ_3^n(s)\geq j)ds - \mu_{\min}\int_0^t\tX_3^n(s)ds,\\
    \tilde{Y}_i^n(t) & = + n^{-1/2}\tA_{i-1}^n(t) - n^{-1/2}\tA_{i}^n(t)- n^{-1/2}\sum_{j=1}^n\mu_j^n\int_0^t \II(\tQ_i^n(s)\geq j)ds + \mu_{\min}\int_0^t\tX_i^n(s)ds\\
    &\quad + n^{-1/2}\sum_{j=1}^n\mu_j^n\int_0^t \II(\tQ_{i+1}^n(s)\geq j)ds - \mu_{\min}\int_0^t\tX_{i+1}^n(s)ds, \mbox{ for all }i\geq 3,\\
    \tM_A^n(t) &=n^{-1/2}\left(\tA_0^n(t) - n\lambda^nt\right),\\
    \tM_{D,i}^n(t) &= n^{-1/2}\left(\tD_i^n(t) - \sum_{j=1}^n\mu_j^n\int_0^t \II(\tQ_i^n(s)\geq j)ds\right) \mbox{ for } i \in \NN.
\end{align*}

The following two lemmas are key in proving our result.
 \begin{lemma}[Bhambay et al. \cite{bhambay2025asymptotic}, Lemma 10]
 \label{lem:two-dim-sys}
Consider the following integral equations
\begin{align}
x_1(t) &=h_1 +y_1(t) - \int_0^t(\mu_{\max} x_1(s) - \mu_{\min} x_2(s))ds -u_1(t), \label{eqn:ref_1}\\
x_2(t) &=h_2 +y_2(t) - \int_0^t\mu_{\min} x_2(s)ds +u_1(t)-u_2(t),\\
x_1(t)&\leq0, \ 0\leq x_2(t) \leq B, \ t\geq0, \label{eqn:ref_3}
\end{align}
where $u_1$ and $u_2$ are non-decreasing and non-negative functions satisfying
\begin{align}
\int_0^t\II{x_1(s)<0}du_1(s)&=0,\label{eqn:ref4}\\
\int_0^t\II{x_2(s)<B}du_2(s)&=0.\label{eqn:ref5}
\end{align}
Then for given $B\in \bar{\RR}_+$, $\mathbf{h}\in \RR^2$, and $\mathbf{y}\in \DD_{\RR^2}[0,\infty)$~\eqref{eqn:ref_1}-\eqref{eqn:ref_3} has a unique solution $(\mf x,\mf u) \in \DD_{\RR^2}[0,\infty) \times \DD_{\RR^2}[0,\infty)$. Moreover, there exists a well-defined  function $(f,  g): \bar{\RR}_+ \times \RR^2 \times \DD_{\RR^2}[0,\infty) \to \DD_{\RR^2}[0,\infty) \times \DD_{\RR^2}[0,\infty)$ which maps $(B,\mf h,\mf y)$ to $\mf x =f(B,\mf h,\mf y)$ and $\mf u=g(B,\mf h,\mf y)$. Furthermore, the function $(f,g)$ is continuous on $\bar{\RR}_+ \times \RR^2 \times \DD_{\RR^2}[0,\infty)$. Finally, $\mf y$ continuous implies that $\mf x$ and $\mf u$ are also continuous.
 \end{lemma}

 \begin{lemma}[cf. \cite{eschenfeldt2018join}, Lemma 5]
     \label{lem:higher_queue}
     For any $h\in \RR^{\infty}$ and $y\in \DD_{\RR}^{\infty}[0,\infty)$, the following system of equations 
     \begin{align}
         x_1(t) &= x_2(t) = 0,\\
         x_i(t) & = h_i + y_i(t) - \mu_{\min}\int_0^t x_i(s)ds + \mu_{\min}\int_0^t x_{i+1}(s)ds, \mbox{ for all }i\geq 3
     \end{align}
     has a unique solution $x\in \DD_{\RR}^{\infty}[0,\infty)$ and there exists a well-defined mapping $f:\RR^\infty\times \DD_{\RR}^{\infty}[0,\infty)\to \DD_{\RR}^{\infty}[0,\infty)$. Furthermore, $f$ is a continuous mapping and $x=f(h,y)$ is continuous if $y$ is continuous. 
 \end{lemma}

We now show that $\mf{\tilde{Y}}^n\Rightarrow \mf{y}=(y_1, 0, 0,\ldots)$, where 
\[
y_1(t) = -(\varsigma + \nu)t + \sqrt{2}W(t) \mbox{ for all }t\geq 0,
\]
and then use the continuity of the mappings in Lemmas~\ref{lem:two-dim-sys} and \ref{lem:higher_queue} to prove our result. First, following the same steps as in the proof of Lemma~\ref{lem:martingale_convergence}, the martingale central limit theorem and Lemma~\ref{lem:stochastic_boundedness_jsq_ssl} imply that $\tM_A^n-\tM_{D,1}^n + \tM_{D,2}^n\Rightarrow \sqrt{2}W$. Using Skorokhod representation theorem, we can assume that the processes converge almost surely and in probability. For any $\rho,T>0$, we also have a $K_{\rho}>0$ such that
\[
\PP(|\mf{\tilde{Y}}^n - \mf{y}|_{\rho,T}^*>\epsilon)\leq \PP(|\mf{\tilde{Y}}^n - \mf{y}|_{\rho,T}^*>\epsilon, |\tX_i^n|_T^*\leq K_\rho \mbox{ for }i=1,2) + \frac{\rho}{2}. 
\]
Hence, without loss of generality, we can assume that $|\tX_i^n|_T^*\leq K_\rho$ for $i=1,2$ and this implies
\begin{align*}
    \left|\mu_{\max}\int_0^t \tX_1^n(s)ds + n^{-1/2}\sum_{j=1}^n\mu_j^n\left(t - \int_0^t \II(\tQ_1^n(s) \geq j)ds\right)\right|&\leq n^{-1/2}T\sum_{j=n-K_\rho n^{1/2}}^n(\mu_{\max}-\mu_j^n)\\
    \left|n^{-1/2}\sum_{j=1}^n\mu_j^n\int_0^t \II(\tQ_i^n(s)\geq j)ds - \mu_{\min}\int_0^t\tX_i^n(s)ds\right| &\leq n^{-1/2}T\sum_{j=1}^{K_\rho n^{1/2}}(\mu_j^n-\mu_{\min}),  i\geq 2.
\end{align*}
As the support of the limiting service rate distribution $F$ is $[\mu_{\min},\mu_{\max}]$, for any $\varepsilon>0$, there exists $p_{\varepsilon}>0$ and $n_\varepsilon$ such that $n_\varepsilon>K_\rho n_\varepsilon^{1/2}$ and $n>n_\varepsilon$ implies
\[
\sum_{j=1}^n \delta_{\mu_{j}^n}[\mu_{\min}, \mu_{\min}+\varepsilon]>p_\varepsilon n\mbox{ and }\sum_{j=1}^n \delta_{\mu_{j}^n}[\mu_{\max}-\varepsilon, \mu_{\max}]>p_\varepsilon n.
\]
Then, for all $n>n_\varepsilon$,  
\begin{align}
   \label{eq:bounds_mu1} \left|\mu_{\max}\int_0^t \tX_1^n(s)ds + n^{-1/2}\sum_{j=1}^n\mu_j^n\left(t - \int_0^t \II(\tQ_1^n(s) \geq j)ds\right)\right|&\leq \varepsilon K_{\rho}T\\
    \label{eq:bounds_mu2} \left|n^{-1/2}\sum_{j=1}^n\mu_j^n\int_0^t \II(\tQ_i^n(s)\geq j)ds - \mu_{\min}\int_0^t\tX_i^n(s)ds\right| &\leq \varepsilon K_{\rho}T.
\end{align}
As $\varepsilon$ is arbitrary, we conclude that the left-hand sides of \eqref{eq:bounds_mu1} and \eqref{eq:bounds_mu2} converge to 0. As the arrivals are routed to servers with $i$ or more jobs only when $\tX_i^n =n^{1/2}$, for any $n>K_\rho^2$, we also have $\tA_i^n = 0$ for all $i\geq 2$. Hence, we conclude that $\mf{\tilde{Y}}^n\Rightarrow \mf{y}=(y_1, 0, 0,\ldots)$. Using Lemma~\ref{lem:higher_queue}, this implies that $\tX_i^n\Rightarrow X_i^{\SSL}$ as given in \eqref{eq:JSQ_SSL_dif3}. Finally, taking $B=\infty$ in Lemma~\ref{lem:two-dim-sys}, the proposition follows. 
\end{proof}

\subsection{Proof of Lemma~\ref{lem:jsq_ssl_monotonicity}}
\begin{proof}[Proof of Lemma~\ref{lem:jsq_ssl_monotonicity}]
    We start by proving that the coupled CTMC $(\tilde{\mf{Q}}_{b_1}^{\SSL, n}, \tilde{\mf{Q}}_{b_2}^{\SSL,n})$ is positive recurrent using a similar method as in Lemma~\ref{lem:positive_recurrent_finite}. When both $b_1$ and $b_2$ are finite, we have a finite state CTMC and hence, it is positive recurrent. Hence, we focus on the case where $b_2=\infty$. 
    Denote the generator for the coupled system as $\tilde{\cG}_{b_1,b_2}^n$. For the coupled system, we define
    \[V(\mf{q}_{b_1},\mf{q}_{b_2}) = \sum_{i=1}^\infty i(q_{b_1,i} + q_{b_2,i})\]
    Using similar logic, we have 
    \begin{align}
       \nonumber \tilde{\cG}_{b_1,b_2}^n V(\mf{q}_{b_1},\mf{q}_{b_2}) &\leq  n\lambda^n (b_1 + i_*(\mf{q}_{b_2})) - \sum_{k=1}^n \mu_k^n (i_*(\mf{q}_{b_2}) - 1)\\
       \label{eq:lyapunov_combined_eq1} &= -n^{1/2}(\nu^n+\varsigma^n) i_*(\mf{q}_{b_2}) + n\lambda^nb_1 + \sum_{k=1}^n \mu_k^n
    \end{align}
    We also have
    \begin{align}
        \label{eq:lyapunov_combined_eq2}
        \tilde{\cG}_{b_1,b_2}^n V(\mf{q}_{b_1},\mf{q}_{b_2})\leq n\lambda^n (b_1 + i_*(\mf{q}_{b_2})) - \mu_{\min} i^*(\mf{q}_{b_2}) 
    \end{align}
    We next define
    \[
    \KK := \left\{(\mf{q}_{b_1}, \mf{q}_{b_2}): i^*(\mf{q}_{b_2})\leq\frac{
n\lambda^n\left(n\lambda^nb_1 + \sum_{k=1}^n\mu_k^n+\epsilon\right)}{\mu_{\min}n^{1/2}(\nu^n+\varsigma^n)}+\frac{\epsilon}{\mu_{\min}} \right\}.
    \]
    We have
    \[
    \cG^n V(\mf{q}_{b_1},\mf{q}_{b_2}) \leq K = n\lambda^n \left(\frac{
n\lambda^n\left(n\lambda^nb_1 + \sum_{k=1}^n\mu_k^n+\epsilon\right)}{\mu_{\min}n^{1/2}(\nu^n+\varsigma^n)}+\frac{\epsilon}{\mu_{\min}}+1\right)<\infty.
    \]
    Again, for any $\mf{q}\notin\KK$, if
    \[
i_*(\mf{q}_{b_2})>
\frac{n\lambda^nb_1+ \sum_{k=1}^n\mu_k^n + \epsilon}{n^{1/2}(\nu^n +\varsigma^n)},
\]
then \eqref{eq:lyapunov_combined_eq1} implies that $\cG^nV(\mf{q})<-\epsilon$, and if
\[
i_*(\mf{q}_{b_2})\leq
\frac{n\lambda^nb_1+ \sum_{k=1}^n\mu_k^n + \epsilon}{n^{1/2}(\nu^n +\varsigma^n)},
\]
then, since $\mf{q}\notin\KK$, \eqref{eq:lyapunov_combined_eq2} implies that $\cG^n V(\mf{q})<-\epsilon$. We conclude that $(\tilde{\mf{Q}}_{b_1}^{\SSL, n}, \tilde{\mf{Q}}_{b_2}^{\SSL,n})$ is positive recurrent for any $1\leq b_1<b_2\leq \infty$. Now, define 
\begin{equation}
\label{eq:monotonicity_stopping}
\tilde{\tau} = \inf\{t: \tilde{\mf{Q}}_{b_1, i}^{\SSL, n}(t)\leq \tilde{\mf{Q}}_{b_2,i}^{\SSL, n}(t)\mbox{ for all }i\in \NN\}.
\end{equation}
Suppose there exists a $t>\tilde{\tau}$ such that the inequality in \eqref{eq:monotonicity_stopping} does not hold for some $i$ and let 
\[
\tilde{t} = \inf\{t>\tilde{\tau}:\tilde{\mf{Q}}_{b_1, i}^{\SSL, n}(t)>\tilde{\mf{Q}}_{b_2,i}^{\SSL, n}(t) \mbox{ for some }i\in \NN\},
\]
and let $i^{**}$ be the first such $i$. Since queues with more than $b_1$ jobs cannot occur in the first system, we have $i^{**}\leq b_1$. Then, $\tilde{\mf{Q}}_{b_1, i^{**}}^{\SSL, n}(\tilde{t}-)=\tilde{\mf{Q}}_{b_2,i^{**}}^{\SSL, n}(\tilde{t}-)$. The event occurring at time $\tilde{t}$ cannot be an arrival as this will require $\tilde{\mf{Q}}_{b_1, i^{**}}^{\SSL, n}(\tilde{t}-)=\tilde{\mf{Q}}_{b_2,i^{**}}^{\SSL, n}(\tilde{t}-)<n$, in which case the arrival is routed to a buffer with exactly $i^{**}-1$ jobs in both systems. If the event at time $\tilde{t}$ is a departure, it should be a departure from a buffer with $\tilde{i}>i^{**}$ jobs in the first system and a departure from a buffer with exactly $i^{**}$ jobs in the second system. However, as the definition of $\tilde{t}$ implies that 
$\tilde{\mf{Q}}_{b_1, i}^{\SSL, n}(\tilde{t}-)\leq\tilde{\mf{Q}}_{b_2,i}^{\SSL, n}(\tilde{t}-)$, and in turn, \eqref{eq:u_coupling_JSQ_SSL} implies that such a departure is not possible. Using the positive recurrence of the coupled process,
\[
\lim_{t\to \infty}\PP(\tilde{\mf{Q}}_{b_1, i}^{\SSL, n}(t)> \tilde{\mf{Q}}_{b_2,i}^{\SSL, n}(t))
\leq
\lim_{t\to \infty}\PP(\tilde{\tau}>t)
=0.
\]
Now, for any $k\geq 0$, we have 
\begin{align*}
   \lim_{t\to \infty}\PP(\tilde{\mf{Q}}_{b_1, i}^{\SSL, n}(t)\geq k)&\leq \lim_{t\to \infty} \left( \PP(\tilde{\mf{Q}}_{b_2,i}^{\SSL, n}(t)\geq k) + \PP(\tilde{\mf{Q}}_{b_1, i}^{\SSL, n}(t)> \tilde{\mf{Q}}_{b_2,i}^{\SSL, n}(t))\right)\\
   &=\lim_{t\to \infty} \PP(\tilde{\mf{Q}}_{b_2,i}^{\SSL, n}(t)\geq k), 
\end{align*}
which proves the desired monotonicity result. 
\end{proof}
\section{Proofs of Results Presented in Section 7}

\subsection{Proof of Proposition~\ref{prop:stationary_bound_JSQ_SSL}}
\label{app:stationary_bound_JSQ_SSL}
The proof of Proposition~\ref{prop:stationary_bound_JSQ_SSL} is based on Stein's method, specifically the generator expansion approach as in \cite{Braverman2020}. The Poisson equation obtained through generator expansion has the same structure as the one studied in \cite{bhambay2025asymptotic}, and we therefore employ the Lyapunov function developed therein. The main challenge in our current setting stems from the analysis of the error terms. We start with writing the infinitesimal generator $\tilde{\cG}_{\boldsymbol{\mu}}^n$ of the underlying Markov chain $(\mf{\tilde Q}^n(t), t\geq 0)$ with known $\boldsymbol{\mu}^n$.
When applied to a function $f:\RR_{+}^\infty\to \RR$ at point $\mf{\tilde Q}(t)=\mf{q}$, it takes the form
\begin{align}
   \nonumber \tilde{\cG}_{\boldsymbol{\mu}}^n f(\mf{q}) 
   &= n\lambda^n\II(q_1<n)\bigl(f(\mf{q}+\mf{e}_1)-f(\mf{q})\bigr) \\
   &\quad + n\lambda^n\sum_{i=1}^n\II(q_i=n, q_{i+1}<n)\bigl(f(\mf{q}+\mf{e}_{i+1})-f(\mf{q})\bigr) \nonumber\\
   &\quad + \sum_{i=1}^\infty 
   \left(\sum_{k=1}^{q_i}\mu_k^n-\sum_{k=1}^{q_{i+1}}\mu_k^n\right)
   \bigl(f(\mf{q}-\mf{e}_{i})-f(\mf{q})\bigr),
   \label{eq:generator_jsq_ssl}
\end{align}
where $\mf{e}_i\in\RR^\infty$ is the unit vector whose $i$th entry is $1$ and all other entries are $0$. For any service rate vector $\boldsymbol{\mu}^n$ with $\varsigma^n+ \nu^n>\varepsilon$, Lemma~\ref{lem:jsq_ssl_monotonicity} implies that the JSQ-SSL system is ergodic and therefore by Lemma~1 in \cite{Braverman2020} we have 
\begin{align}\label{eq:generator_rate_balance}
    \EE[\tilde{\cG}_{\boldsymbol{\mu}}^n f(\tilde{\mf{Q}}^n(\infty))|\varsigma^n+ \nu^n>\varepsilon]=0
\end{align}
for all $f\in \SSS_Q$ with $\EE[f(\tilde{\mf{Q}}^n(\infty))|\varsigma^n+ \nu^n>\varepsilon]<\infty$. Using the result above we establish Lemma~\ref{lem:rate_balance_q} as a modification of Lemma 2 in \cite{Braverman2020} and Lemma 7 in \cite{bhambay2025asymptotic},  and is key in proving \eqref{eq:stationary_bound_JSQ_SSL1}. 
We provide its proof in Appendix~\ref{sec:rate_balance_proof}.

\begin{lemma}\label{lem:rate_balance_q}
    For all $n\in \NN$, we have 
    \begin{align} \label{eq:rate_balance_q1} \EE\left[\sum_{k=1}^{\tilde{Q}_{1}^n(\infty)}\mu_{k}^n \mid \varsigma^n + \nu^n>\varepsilon\right]&=n\lambda^n,\\
    \label{eq:rate_balance_q2}
        \EE\left[ \sum_{k=1}^{\tilde{Q}_{i}^n(\infty)}\mu_k^n\mid \varsigma^n + \nu^n>\varepsilon \right] & = n\lambda^n\PP\left(\tilde{Q}_1^n(\infty) = \cdots = \tilde{Q}_{i-1}^n(\infty) = {n}\right)\mbox{ for all }i\geq 2.
    \end{align}
\end{lemma}

 To prove \eqref{eq:stationary_bound_JSQ_SSL2} and \eqref{eq:stationary_bound_JSQ_SSL3}, we use the generator expansion approach on lifted versions $Af:\SSS_Q\to \RR$ of $f:\RR_+^2\to \RR$ as in \cite{Braverman2020} and \cite{bhambay2025asymptotic}. 
 We define the lifting operator $A$ as 
$Af(\mf{q}) = f(\mf{y}), \mbox{ for all }\mf{q}\in \SSS_Q$,
where $\mf{y}=(y_1,y_2)$ with $y_1=(q_1-n)/n$ and $y_2=q_2/n$. The generator applied to the lifted function then can be written as
\begin{align}
    \nonumber \tilde{\cG}_{\boldsymbol{\mu}}^n A f(\mf{q}) &=n\lambda^n(f(\mf{y}+\mf{e}_1/n)-f(\mf{y})) - n\lambda^n\II(y_1=0)(f(\mf{y}+\mf{e}_1/n)-f(\mf{y}))\\
    \nonumber&\quad + n\lambda^n\II(y_1=0)(f(\mf{y}+\mf{e}_2/n)-f(\mf{y}) ) - n\lambda^n\II(y_1=0, y_2=1) (f(\mf{y}+\mf{e}_2/n)-f(y))\\
    &\quad + \left(\sum_{k=1}^{q_1}\mu_k^n - \sum_{k=1}^{q_2}\mu_k^n \right)(f(\mf{y}-\mf{e}_1/n)-f(\mf{y})) + \left(\sum_{k=1}^{q_2}\mu_k^n - \sum_{k=1}^{q_3}\mu_k^n \right)(f(\mf{y}-\mf{e}_2/n)-f(\mf{y}))\label{eq:lifted_generator}.
\end{align}
Now, defining the operator $\cL$ as
\begin{equation}
\cL^n f(y) = \left(-\frac{\varsigma^n + \nu^n}{n^{1/2}} -\mu_{\max}y_1 + \mu_{\min}y_2\right)f_1(y) -\mu_{\min}y_2f_2(y),    
\label{eq:pde_operator}
\end{equation}
and replacing the Taylor expansions 
 \begin{align*}
    f(\mf y+\mf e_{1}/n)-f(\mf y) &=\frac{1}{n}f_1(\mf y)+\int_{y_1}^{y_1+1/n}(y_1+1/n-u)f_{11}(u,y_2)du,\\
    f(\mf y-\mf e_{1}/n)-f(\mf y) &=-\frac{1}{n}f_1(\mf y)+\int_{y_1-1/n}^{y_1}(u-y_1+1/n)f_{11}(u,y_2)du,
\end{align*}
in \eqref{eq:lifted_generator} we have
\begin{align}\label{eq:generator_decomposition}
    \tilde{\cG}_{\boldsymbol{\mu}}^n Af(\mf{q})=\cL^n f(\mf{y}) -{\lambda^n}\II(y_1=0)(f_1(\mf{y}) - f_2(\mf{y})) + \varepsilon^n(\mf{q}),
\end{align}
where
\begin{align}
   \label{eq:error_decomp1} \epsilon^n(\mf{q})
   &= \sum_{k=q_1+1}^n \frac{\mu_k^n-\mu_{\max}}{n}f_1(\mf{y})
    + \sum_{k=1}^{q_2}\frac{\mu_k^n-\mu_{\min}}{n} f_1(\mf{y})
    - \sum_{k=1}^{q_2}\frac{\mu_k^n-\mu_{\min}}{n} f_2(\mf{y})\\
    \label{eq:error_decomp3}&\quad - n\lambda^n\II(y_1=0, y_2=1)\int_{y_2}^{y_2+1/n}f_2(y_1,u)du
    \\
    \label{eq:error_decomp4}&\quad + \sum_{k=1}^{q_3}\mu_k^n \int_{y_2-1/n}^{y_2}f_2(y_1,u)du\\
    \label{eq:error_decomp5}&\quad + n\lambda^n \II(y_1<0)\int_{y_1}^{y_1+1/n}(y_1+1/n-u)f_{11}(u,y_2)du\\
    \label{eq:error_decomp6}&\quad + n\lambda^n \II(y_1=0)\int_{y_2}^{y_2+1/n}(y_2+1/n-u)f_{22}(y_1, u)du\\
    \label{eq:error_decomp7}& \quad  + \left(\sum_{k=1}^{q_1}\mu_k^n-\sum_{k=1}^{q_2}\mu_k^n\right)
    \int_{y_1-1/n}^{y_1}(y_1+1/n-u)f_{11}(u,y_2)du\\
    \label{eq:error_decomp8}&\quad + \sum_{k=1}^{q_2}\mu_k^n
    \int_{y_2-1/n}^{y_2}(y_2+1/n-u)f_{22}(y_1, u)du
\end{align}

The following lemma is a restatement of Lemma 8 of \cite{bhambay2025asymptotic} is central to our proof. As in this paper we require the coefficients $m_1$ and $m_2$ in \eqref{eq:pde} to be generic, we reproduce its proof in Appendix~\ref{app:lyapunov_derivation}

\begin{lemma}[Lemma 8 in \cite{bhambay2025asymptotic}]
\label{lem:solution_PDE}
For $m_1, m_2, \beta > 0$ with $m_1\neq m_2$ define the linear operator $\mc{L}$ acting on functions $f:\RR^2 \to \RR$ as 
\begin{equation}
    \mc{L }f(\mf y) = \left(-\frac{\beta}{n^{1/2}} -m_{1}y_1 + m_{2}y_2\right)f_1(\mf y) -m_{2}y_2f_2(\mf y),
    \label{eq:pde}
\end{equation}
where $f_i\equiv\partial f /\partial y_i$ for each $i\in [2]$.
Furthermore, define $\Omega=(-\infty,0]\times [0,\infty)$. Then, for each fixed $\kappa>\beta/m_{2}$ there exists a function $f^*(\mf y)$ that solves the PDE 
\begin{align}\label{eq:PDE_Lyapunov1}
    \cL f^*(\mf y) &= -(y_2-\kappa/n^{1/2})_+, \mbox{ for all } \mf y\in \Omega,\\
    \label{eq:PDE_Lyapunov2}f_1^*(0,y_2) &= f_2^*(0,y_2)=\frac{n^{1/2}}{\beta}\left(y_2-\frac{\kappa}{n^{1/2}}\right)_+, \mbox{ for all }y_2\geq 0,
\end{align}
with absolutely continuous first order derivatives $f_1^*(\cdot, y_2)$ and $f_2^*(y_1, \cdot)$. Moreover, we have
\begin{align}
f_1^*(\mf y)\geq 0, f_{12}^*(\mf y)\geq 0, f_{11}^*(\mf y)\geq 0, f_{22}^*(\mf y)\geq 0, \ \ \ \ \ \mf y\in \Omega, \label{eqn:second_derive_bound_nn}\\
f_{11}(\mf y)=f_{22}(\mf y)=0, \ \  y_2 \in [0,\kappa/n^{1/2}],\label{eq:zero_second_derivative}\\
f_{11}^*(\mf y)\leq C_5(\kappa)n^{1/2}, f_{22}^*(\mf y)\leq C_6(\kappa)n^{1/2}, \ \ \ \ \ \mf y\in \Omega,\label{eq:bound_secoond_derivative}
\end{align}
where $f_{ij}^*\equiv \partial^2 f^*/ \partial y_i \partial y_j $, for $i,j\in [2]$, $C_5(\kappa)= \frac{1}{\beta}\frac{m_{2}+m_{1}}{m_{2}} \frac{\kappa}{\kappa-\beta/m_{2}}$ and $$C_6(\kappa)=\frac{1}{m_{2} \kappa}+ \frac{1}{\beta}\brac{\frac{m_1\vee m_{2}}{m_{1}-m_{2}}}^2\frac{\kappa}{\kappa-\beta/m_{2}}+ \frac{1}{\beta}\frac{(m_{1}\vee m_{2})^2}{(m_{1}-m_{2})^2}.$$
\end{lemma}

\begin{proof}[Proof of Proposition~\ref{prop:stationary_bound_JSQ_SSL}] Now, we are in a position to prove Proposition~\ref{prop:stationary_bound_JSQ_SSL}.
To prove \eqref{eq:stationary_bound_JSQ_SSL1}, we condition on the value of $\boldsymbol{\mu}^n= \mf{m}$ such that $\varsigma^n + \nu^n>\varepsilon$ and apply Lemma~\ref{lem:rate_balance_q} which implies that for all sufficiently large $n$ it holds that
\begin{align}
    \nonumber \mu_{\min}\EE\left[|\tilde{X}_1^n(\infty)|\mid \varsigma^n + \nu^n>\varepsilon\right] &=  n^{-1/2} \EE\left[ \sum_{k=1}^{n}\mu_{\min}(1-\II(k\leq\tilde Q_{1}^n(\infty)))\mid \varsigma^n + \nu^n>\varepsilon\right]\\
    \nonumber&\leq n^{-1/2} \EE\left[ \sum_{k=1}^{n}\mu_{k}^n-\sum_{k=1}^{\tilde{Q}_{1}^n(\infty)}\mu_{k}^n\mid \varsigma^n + \nu^n>\epsilon\right]\\ &= \EE\left[  \varsigma^n + \nu^n\mid \varsigma^n + \nu^n>\epsilon\right]\leq \frac{\EE[|\varsigma^n + \nu^n|]}{\PP(\varsigma^n + \nu^n>\epsilon)}.
\end{align}
Hence,~\eqref{eq:stationary_bound_JSQ_SSL1} follows by Assumption~\ref{asm:arrival_service_rates}.

We define $\tilde{Y}_i^n(\infty) = n^{-1/2}\tilde{X}_i^n(\infty)$ for all $i\in \NN$ to match the definition of $y_i$.Now, choosing $m_1=\mu_{\max}, m_2=\mu_{\min}$, and $\beta=\varsigma^n+\nu^n$ in Lemma~\ref{lem:solution_PDE} to obtain $f^*(\mf{q})$, plugging in $\tilde{\mf{Q}}^n(\infty)$ and using \eqref{eq:generator_rate_balance}, we have 
\[
\EE\left[\left(Y_2^n(\infty)-\frac{\kappa}{n^{1/2}}\right)_+\mid \varsigma^n + \nu^n>\varepsilon\right]=\EE[\epsilon^n(\mf{Q}^n(\infty))|\varsigma^n + \nu^n>\varepsilon].
\]
The positivity of $f_1^*(\mf y)$ implies that the first term in \eqref{eq:error_decomp1} is negative. For the next two terms  in \eqref{eq:error_decomp1} we have
\begin{align*}
   \sum_{k=1}^{q_2}\frac{\mu_k^n-\mu_{\min}}{n} f_1^*(\mf{y})
    - \sum_{k=1}^{q_2}\frac{\mu_k^n-\mu_{\min}}{n} f_2^*(\mf{y}) &= \sum_{k=1}^{q_2} \frac{\mu_k^n-\mu_{\min}}{n}\left( f_1^*(\mf y) -  f_2^*(\mf y)\right)
\end{align*}
and
\begin{align*}
     (f_1^*(\mf y) -f_2^*(\mf y))&=\frac{1}{\mu_{\min}y_2}\left(-\left(y_2-\frac{\kappa}{n^{1/2}}\right)_+ + \left(\frac{\varsigma^n + \nu^n}{n^{1/2}} +\mu_{\max}y_1\right)f_1^*(\mf y)\right)\\
     &\leq\frac{1}{\mu_{\min}y_2}\left(-\left(y_2-\frac{\kappa}{n^{1/2}}\right)_+ + \frac{\varsigma^n + \nu^n}{n^{1/2}}f_1^*(\mf y)\right)\\
     & \leq\frac{1}{\mu_{\min}y_2}\left(-\left(y_2-\frac{\kappa}{n^{1/2}}\right)_+ + \frac{\varsigma^n + \nu^n}{n^{1/2}}f_1^*(0,y_2)\right) \leq 0,
\end{align*}
where the first line follows since $\cL f^*(y)=-(y_2-\kappa/\sqrt n)_+$, the second line follows since $y_1\leq 0, y_2\geq 0,$ and $f_1^*(\mf y)\geq 0$, the first inequality in the third line follows since $f_{11}^*(\mf y)\geq 0$ and finally the last inequality follows from~\eqref{eq:PDE_Lyapunov2}. Since $\mu_k^n\geq \mu_{\min}$ for all $k$,
we conclude that the error term on the right-hand side of \eqref{eq:error_decomp1} is non-positive and and can be safely omitted to obtain an upper bound.

Next, we focus on the terms appearing on lines~\eqref{eq:error_decomp3}-\eqref{eq:error_decomp4}. Using the monotonicity of the first derivatives and the boundary condition, we have
\begin{align*}
    &-n\lambda^n\II(y_1=0, y_2=1)\int_{y_2}^{y_2+1/n}f_2^*(y_1,u)du + \sum_{k=1}^{q_3}\mu_k^n \int_{y_2-1/n}^{y_2}f_2^*(y_1,u)du\\ &\qquad\qquad\qquad\qquad\qquad\qquad\qquad\leq - \lambda^n\II(y_1=0, y_2=1)f_2^*(0,1)
+\frac{\sum_{k=1}^{q_3}\mu_k^n}{n} f_2^*(0,1).
\end{align*}
Evaluating the right-hand side with $y_i=\tilde{Y}_i^n(\infty)$ for $i=1,2$ and  $q_3=\tilde{Q}_3^n(\infty)$,  taking the conditional expectation and using Lemma~\ref{lem:rate_balance_q}, we see that the terms corresponding to \eqref{eq:error_decomp3} and \eqref{eq:error_decomp4} is also non-positive and can be omitted in the upper bound. 

Equation \eqref{eq:zero_second_derivative} imply that the second derivatives in \eqref{eq:error_decomp5}--\eqref{eq:error_decomp8} are zero when $y_2\leq \kappa/n^{1/2}-1/n$. Combining this with \eqref{eq:bound_secoond_derivative}, we get 
\[
    n\lambda^n \II(y_1<0)\int_{y_1}^{y_1+1/n}(y_1+1/n-u)f_{11}(u,y_2)du \leq \lambda^nn^{-1/2}C_5(\kappa)\II(y_2> \kappa/n^{1/2}-1/n), 
\]
and the error terms corresponding to \eqref{eq:error_decomp6}-\eqref{eq:error_decomp8} can be bounded similarly. Aggregating all the terms,
for any $\kappa > a(\nu^n+\varsigma^n)/\mu_{\min}$ with $a>1$
\begin{align}
    \nonumber&\EE\left[\left(Y_2^n(\infty)-\frac{\kappa}{n^{1/2}}\right)_+\mid \varsigma^n + \nu^n>\varepsilon\right]\\
    &\qquad\qquad\qquad\leq \frac{2(\lambda^n+\mu_{\max})(C_5(\kappa)+C_6(\kappa))}{n^{1/2}}\PP\left(Y_2^n(\infty)\geq \frac{\kappa}{n^{1/2}}-\frac{1}{n}\mid \varsigma^n + \nu^n>\varepsilon\right).\label{eq:y2_bound}
\end{align}
Equation~\eqref{eq:stationary_bound_JSQ_SSL2} follows realizing that the constants $C_5(\kappa)$ and $C_6(\kappa)$ are both decreasing in $\kappa$ and $\varsigma^n + \nu^n>\varepsilon$ imply
\[
C_5(\kappa) + C_6(\kappa)
\leq 
\frac{1}{\varepsilon}
\frac{\mu_{\min}+\mu_{\max}}{\mu_{\min}}
\frac{a}{a-1}
+ 
\frac{1}{\varepsilon}
\left[
\frac{1}{a}
+
\left(\frac{\mu_{\max}}{\mu_{\max}-\mu_{\min}}\right)^2
\frac{2a-1}{a-1}
\right]:= C_{7,\epsilon}.
\]
To establish~\eqref{eq:stationary_bound_JSQ_SSL3} we use a bootstrapping argument as in~\cite{Braverman2020}.
For large enough $n$, we can choose $\tilde{\kappa}$ such that $\max\{a (\varsigma^n + 
\nu^n)/(\mu_{\min}n^{1/2}), 1/n\}<\tilde{\kappa}<1$ with $a>1$ and setting $\kappa=n^{1/2}\tilde \kappa$ in \eqref{eq:y2_bound}, we obtain
\begin{align}
    \nonumber\EE\left[(Y_2^n(\infty)-\tilde\kappa)_+\mid \varsigma^n + \nu^n>\varepsilon\right]& \leq \frac{2(\lambda^n+\mu_{\max})C_{7,\varepsilon}}{n^{1/2}}\PP(Y_2^n(\infty)\geq \tilde\kappa-1/n\mid \varsigma^n + \nu^n>\varepsilon)\\
    \nonumber&\leq \frac{2(\lambda^n+\mu_{\max})C_{7,\varepsilon}}{n^{1/2}}\frac{\EE[Y_2^n(\infty)]}{\tilde\kappa-1/n}\\
    &\leq \frac{2(\lambda^n+\mu_{\max})C_{7,\varepsilon}}{n}\frac{\tilde{C}_{2,\varepsilon}}{\tilde\kappa-1/n}\label{eq:y_order_n}
\end{align}
where the second line follows due to Markov inequality and the third line follows from \eqref{eq:stationary_bound_JSQ_SSL2}.
Hence, we have 
\begin{align*}
    \frac{2(\lambda^n+\mu_{\max})C_{7,\varepsilon}}{n}\frac{\tilde{C}_{2,\varepsilon}}{\tilde\kappa-1/n}&\geq \EE\left[(Y_2^n(\infty)-\tilde\kappa)_+\mid \varsigma^n + \nu^n>\varepsilon\right]\\ &\geq (1-\tilde \kappa)\PP(Y_2^n(\infty)=1\mid \varsigma^n + \nu^n>\varepsilon),\\
    &= (1-\tilde \kappa)\PP(\tilde{Q}_2^n(\infty)=n\mid \varsigma^n + \nu^n>\varepsilon),\\
    &\geq (1-\tilde \kappa)\mu_{\min}\EE[\tilde{Q}_3^n(\infty)\mid \varsigma^n + \nu^n>\varepsilon]\\
    & = n^{1/2}(1-\tilde \kappa)\mu_{\min}\EE[\tilde{X}_3^n(\infty)\mid \varsigma^n + \nu^n>\varepsilon],
\end{align*}
where we use~\eqref{eq:rate_balance_q2} to obtain the second inequality. Hence, the result follows.\end{proof}

\subsection{Proof of Lemma~\ref{lem:x_1_lower_bound}}
\label{app:x1_stationary_bound}

\begin{proof}[Proof of Lemma~\ref{lem:x_1_lower_bound}]
As we assume  a general buffer capacity which can be finite, the Lyapunov approach in bounding the similar quantity in \cite{Braverman2020} and \cite{bhambay2025asymptotic} requires additional control on the probability that all buffers are full. Hence, we prove this result with a coupling argument as in the proof of Lemma~\ref{lem:A_1_bound} instead. 

Consider the construction for the load-balancing system therein now letting the potential departure process $S^n(t)$ be a Poisson process with rate $\sum_{k=1}^n\mu_k^n$ and modify the probability of departure accordingly. Let $B_1^n(t)$ be a birth-death process with death rate $n\lambda^n$ and a state-dependent birth rate 
\[
\mu_B^n(y) = \begin{cases}
    \sum_{k=1}^n \mu_k^n = n\lambda^n + n^{1/2}(\varsigma^n + \nu^n) &\mbox{if } y\leq y^{n,*},\\
    \sum_{k=1}^n \mu_k^n - 2n^{1/2}(\varsigma^n + \nu^n) = n\lambda^n - n^{1/2}(\varsigma^n + \nu^n) &\mbox{if } y> y^{n,*},
\end{cases}
\]
where $y^{n,*} =\mu_{\min}^{-1}2n^{1/2}(\varsigma^n + \nu^n)$.

To construct the coupling, we assume that the $j$th arrival epoch $\theta_{A,j}^n$ of $A^n(t)$ converges to a death if $B_1^n(\theta_{A,j}^n-)>0$ and the $j$th potential departure epoch $\theta_{S,j}^n$ of $S^n(t)$ corresponds to a birth if 
\[
\tilde{U}_j^n\leq \frac{\mu_B^n(B_1^n(\theta_{S,j}^n-))}{\sum_{k=1}^n\mu_k^n}
\]
Now assume that $B_1^n(0) = n-Q_1^{\pi,n}(0)$ and let $\tilde{t}=\inf\{t: B_1^n(t)<n-Q_1^{\pi,n}(t)\}<\infty$. This implies $B_1^n(\tilde{t}-)=n-Q_1^{\pi,n}(\tilde{t}-)$. The event at $\tilde{t}$ cannot be an arrival as any arrival that decreases $B_1$ should also increase $X_1$. We also have 
\[
\mu_B^n(B_1^n(\tilde{t}-)) - \sum_{k=1}^n\mu_k^n(1-Q_{1,k}^{\pi,n}(\tilde{t}-)) \geq \begin{cases}
    \sum_{k=1}^n \mu_k^nQ_{1,k}^{\pi,n}(\tilde{t}-)    &\mbox{if } B_1^n(\tilde{t}-)\leq y^{n,*},\\
   \mu_{\min} &\mbox{if } B_1^n(\tilde{t}-)> y^{n,*},
\end{cases}
\]
which is positive in both cases and hence, the event at $\tilde{t}$ cannot be a potential departure as well. We can conclude that such a $\tilde{t}$ cannot exist and $B_1^n(t)\geq_{st}-Q_1^{\pi,n}(t)$ for all $t\geq 0$. 

Similarly, define $B_2^n(t)$ to be a birth-death process with $B_1^n(0) = B_2^n(0)$, birth rate $n\lambda^n - 2n^{1/2}(\varsigma^n + \nu^n)$ and death rate $n\lambda^n$. Repeating the same argument as above, we have $n-Q_1^{n\pi}(t)\leq_{st} B_1^n(t)\leq_{st} y^{n,*} + B_2^n(t)$ for all $t\geq 0$. Hence, taking the expectations of the stationary versions
\begin{align*}
    \EE[|X_1^{n,\pi}(\infty)|\mid (\varsigma^n + \nu^n)>\epsilon] & \leq  \mu_{\min}^{-1}\EE[(\varsigma^n + \nu^n)|(\varsigma^n + \nu^n)>\epsilon] + \mathbb{E}\left[n^{-1/2}B_2^n(\infty)|\epsilon\right]\\
&\leq  \mu_{\min}^{-1}\EE[(\varsigma^n + \nu^n)|(\varsigma^n + \nu^n)>\epsilon] + 
\frac{\lambda^n}{2\epsilon}.
\end{align*}
The convergence of expectations in \ref{asm:arrival_service_rates} implies the convergence of the first term on the right-hand side. Hence, the result follows. 
\end{proof}

\subsection{Proof of Lemma~\ref{lem:rate_balance_q}}
\label{sec:rate_balance_proof}
\begin{proof}[Proof of Lemma~\ref{lem:rate_balance_q}]
    
    Again, applying the generator as defined in~\ref{eq:generator_jsq_ssl} on function     \[
    g(\mf{q}) = \min\left\{T, \sum_{i=1}^\infty q_i\right\},
    \] we have
    \[
    \tilde{\cG}_{\boldsymbol{\mu}}^n g(\mf{q}) = n\lambda^n \II\left(\sum_{i=1}^\infty q_i<T\right) - \sum_{k=1}^{q_1}\mu_k^n \II\left(\sum_{i=1}^\infty q_i\leq T\right).
    \]
    Taking the expectation conditional on $\varsigma^n + \nu^n>\varepsilon$ on both sides and applying monotone convergence theorem, we obtain \eqref{eq:rate_balance_q1}. The proof of \eqref{eq:rate_balance_q2} follows similar lines using the Lyapunov function
    \[
        g(\mf{q}) = \min\left\{ T, \sum_{j = i}^\infty q_j\right\}.
    \]
    for all $i>1$.
\end{proof}

\subsection{Proof of Lemma~\ref{lem:solution_PDE}}\label{app:lyapunov_derivation}

\begin{proof}
The argument in \cite{bhambay2025asymptotic} is provided with the assumption that $m_1>m_2$. In order to ensure that the function is applicable to general $m_1\neq m_2$, we present the proof of Lemma 8 in~\cite{bhambay2025asymptotic} in our notation. 
\[
\beta:=\nu+\varsigma,
\qquad
d:=m_1-m_2\neq 0.
\]

Recall the first-order operator
\begin{equation}\label{eq:appF:L}
\cL f(\mathbf{y})
=
\left(-\frac{\beta}{\sqrt n}-m_1y_1+m_2y_2\right)f_1(\mathbf{y})-m_2y_2f_2(\mathbf{y}),
\qquad
\mathbf{y}=(y_1,y_2)\in\Omega:=(-\infty,0]\times[0,\infty).
\end{equation}
Associated with \eqref{eq:appF:L} is the piecewise vector field
\[
F(\mathbf{y})=
\begin{cases}
\left(0,-\frac{\beta}{\sqrt n}\right), & \text{if } y_1=0 \text{ and } y_2>\beta/(m_2\sqrt n),\\[4pt]
\left(-\frac{\beta}{\sqrt n}-m_1y_1+m_2y_2,-m_2y_2\right), & \text{otherwise}.
\end{cases}
\]
Thus, away from the boundary segment $\{(0,y_2):y_2>\beta/(m_2\sqrt n)\}$, the dynamics follow the interior drift, whereas on this segment, the trajectory is constrained to remain in $\Omega$ and therefore slides along the boundary.

For each $\mathbf{y}\in\Omega$, let $\mathbf{v}^{\mathbf{y}}(t)=(v_1^{\mathbf{y}}(t),v_2^{\mathbf{y}}(t))$ denote the unique reflected trajectory solving
\begin{align}
v_1^{\mathbf{y}}(t)
&=
y_1-\frac{\beta}{\sqrt n}t-\int_0^t\bigl(m_1v_1^{\mathbf{y}}(s)-m_2v_2^{\mathbf{y}}(s)\bigr)\,ds-U_1(t), \label{eq:appF:v1}\\
v_2^{\mathbf{y}}(t)
&=
y_2-\int_0^t m_2v_2^{\mathbf{y}}(s)\,ds+U_1(t), \label{eq:appF:v2}
\end{align}
where $U_1$ is nondecreasing, $U_1(0)=0$, $U_1(t)\ge 0$, and $\int_0^\infty v_1^{\mathbf{y}}(s)\,dU_1(s)=0$. Let $\tau(\mathbf{y}):=\inf\{t\ge 0:v_1^{\mathbf{y}}(t)=0\}$. For $t<\tau(\mathbf{y})$, the complementarity condition implies $U_1(t)=0$, and hence $v_2^{\mathbf{y}}(t)=y_2e^{-m_2t}$. Substituting this into \eqref{eq:appF:v1} and solving the resulting linear equation yields
\begin{equation}\label{eq:appF:v1sol}
v_1^{\mathbf{y}}(t)
=
-\frac{\beta}{m_1\sqrt n}
+\left(y_1+\frac{\beta}{m_1\sqrt n}\right)e^{-m_1t}
+\frac{m_2y_2}{d}\bigl(e^{-m_2t}-e^{-m_1t}\bigr),
\qquad 0\le t<\tau(\mathbf{y}).
\end{equation}
If $y_2e^{-m_2\tau(\mathbf{y})}>\beta/(m_2\sqrt n)$, then after time $\tau(\mathbf{y})$ the trajectory remains on the boundary and evolves according to
\begin{equation}\label{eq:appF:boundary}
v_1^{\mathbf{y}}(t)=0,
\qquad
v_2^{\mathbf{y}}(t)=y_2e^{-m_2\tau(\mathbf{y})}-\frac{\beta}{\sqrt n}\bigl(t-\tau(\mathbf{y})\bigr),
\qquad t\ge \tau(\mathbf{y}).
\end{equation}

Fix $\kappa>\beta/m_2$. We define $\Gamma^\kappa$ as the set of initial states $\mathbf{y}\in\Omega$ for which the trajectory reaches $y_1=0$ exactly when the second coordinate reaches $\kappa/\sqrt n$. For fixed $y_1\le 0$, imposing $y_2e^{-m_2\tau}=\kappa/\sqrt n$ in \eqref{eq:appF:v1sol} leads to
\[
g_{y_1}(\tau)
:=
-\frac{\beta}{m_1\sqrt n}
+\left(y_1+\frac{\beta}{m_1\sqrt n}\right)e^{-m_1\tau}
+\frac{m_2\kappa}{d\sqrt n}\bigl(1-e^{-d\tau}\bigr)
=0.
\]
Since $g_{y_1}(0)=y_1\le 0$, $g_{y_1}'(\tau)>0$ for all $\tau\ge 0$, and $g_{y_1}(\tau)$ becomes positive for large $\tau$, there exists a unique solution $\tau^*(y_1)$. Writing $\gamma^\kappa(y_1):=\frac{\kappa}{\sqrt n}e^{m_2\tau^*(y_1)}$, we obtain $\Gamma^\kappa=\{(y_1,\gamma^\kappa(y_1)):y_1\le 0\}$, which partitions $\Omega$ into
\begin{align*}
\Omega_1&:=\left\{(y_1,y_2)\in\Omega:y_2\le \frac{\kappa}{\sqrt n}\right\},\\
\Omega_2:&=\left\{(y_1,y_2)\in\Omega:\frac{\kappa}{\sqrt n}<y_2\le \gamma^\kappa(y_1)\right\},\\
\Omega_3&:=\left\{(y_1,y_2)\in\Omega:y_2>\gamma^\kappa(y_1)\right\}.
\end{align*}

Now let
\[
h(\mathbf{y}):=\left(y_2-\frac{\kappa}{\sqrt n}\right)_+,
\qquad
f^*(\mathbf{y}):=\int_0^\infty h(\mathbf{v}^{\mathbf{y}}(s))\,ds
=
\int_0^\infty \left(v_2^{\mathbf{y}}(s)-\frac{\kappa}{\sqrt n}\right)_+\,ds.
\]
If $\mathbf{y}\in\Omega_1$, then $v_2^{\mathbf{y}}(0)=y_2\le \kappa/\sqrt n$ and $v_2^{\mathbf{y}}$ is nonincreasing, so $f^*(\mathbf{y})=0$. If $\mathbf{y}\in\Omega_2$, then the trajectory reaches the level $\kappa/\sqrt n$ before hitting $y_1=0$, at time $T_\kappa=\frac{1}{m_2}\log(\sqrt n\,y_2/\kappa)$, and therefore
\begin{equation}\label{eq:appF:fOmega2}
f^*(\mathbf{y})
=
\int_0^{T_\kappa}\left(y_2e^{-m_2t}-\frac{\kappa}{\sqrt n}\right)\,dt
=
\frac{y_2}{m_2}
-\frac{\kappa}{m_2\sqrt n}
-\frac{\kappa}{m_2\sqrt n}\log\frac{\sqrt n\,y_2}{\kappa}.
\end{equation}
If $\mathbf{y}\in\Omega_3$, then the trajectory first reaches $y_1=0$ at time $\tau(\mathbf{y})$, and at that time $z:=y_2e^{-m_2\tau(\mathbf{y})}>\kappa/\sqrt n$. The contribution up to time $\tau(\mathbf{y})$ is $\frac{y_2}{m_2}(1-e^{-m_2\tau(\mathbf{y})})-\frac{\kappa}{\sqrt n}\tau(\mathbf{y})$, whereas the boundary-sliding contribution is $\frac{\sqrt n}{2\beta}(z-\kappa/\sqrt n)^2$. Hence
\begin{equation}\label{eq:appF:fOmega3}
f^*(\mathbf{y})
=
\frac{y_2}{m_2}\bigl(1-e^{-m_2\tau(\mathbf{y})}\bigr)
-\frac{\kappa}{\sqrt n}\tau(\mathbf{y})
+\frac{\sqrt n}{2\beta}
\left(z-\frac{\kappa}{\sqrt n}\right)^2.
\end{equation}

We next compute the derivatives. On $\Omega_1$, $f_1^*(\mathbf{y})=f_2^*(\mathbf{y})=0$, while on $\Omega_2$,
\[
f_1^*(\mathbf{y})=0,
\qquad
f_2^*(\mathbf{y})=\frac{1}{m_2}-\frac{\kappa}{m_2\sqrt n\,y_2}.
\]
Now let $\mathbf{y}\in\Omega_3$, and write again $z:=y_2e^{-m_2\tau(\mathbf{y})}$. Differentiating implicitly the identity $v_1^{\mathbf{y}}(\tau(\mathbf{y}))=0$, we obtain
\begin{equation}\label{eq:tau12}
\tau_1
=
\frac{-e^{-m_1\tau}}{m_2z-\beta/\sqrt n}
\le 0,
\qquad
\tau_2
=
\frac{m_2\bigl(e^{-m_1\tau}-e^{-m_2\tau}\bigr)}
{d\bigl(m_2z-\beta/\sqrt n\bigr)}
\le 0.
\end{equation}
Here the denominator is strictly positive since $z\ge \kappa/\sqrt n>\beta/(m_2\sqrt n)$, and $\tau_2\le 0$ holds for either sign of $d$, because $e^{-m_1\tau}-e^{-m_2\tau}$ and $d$ always have opposite signs. Differentiating \eqref{eq:appF:fOmega3} with respect to $y_1$, and using $\partial z/\partial y_1=-m_2z\,\tau_1$, gives
\begin{equation}\label{eq:f1}
f_1^*(\mathbf{y})
=
\frac{\sqrt n}{\beta}e^{-m_1\tau}
\left(z-\frac{\kappa}{\sqrt n}\right)\ge 0.
\end{equation}
Now define
\[
\theta(\mathbf{y})
:=
\frac{m_1e^{-m_2\tau}-m_2e^{-m_1\tau}}{d}.
\]
Then $\theta(0)=1$ and $0<\theta(\mathbf{y})\le 1$ for all $\tau\ge 0$. Differentiating \eqref{eq:appF:fOmega3} with respect to $y_2$ and simplifying yields
\begin{equation}\label{eq:f2}
f_2^*(\mathbf{y})
=
\frac{1-e^{-m_2\tau}}{m_2}
+
\left(z-\frac{\kappa}{\sqrt n}\right)\frac{\sqrt n}{\beta}\theta(\mathbf{y})
\ge 0.
\end{equation}
At $y_1=0$, we have $\tau=0$, $z=y_2$, and $\theta=1$, so
\[
f_1^*(0,y_2)=f_2^*(0,y_2)
=
\frac{\sqrt n}{\beta}\left(y_2-\frac{\kappa}{\sqrt n}\right)_+,
\]
which verifies \eqref{eq:PDE_Lyapunov2}.

Next, using \eqref{eq:PDE_Lyapunov1}, we may write
\begin{equation}\label{eq:f1minusf2bound}
m_2y_2\bigl(f_1^*-f_2^*\bigr)
=
-\left(y_2-\frac{\kappa}{\sqrt n}\right)_+
+
\left(\frac{\beta}{\sqrt n}+m_1y_1\right)f_1^*.
\end{equation}
Since $y_1\le 0$ and $f_1^*\ge 0$,
\[
m_2y_2\bigl(f_1^*-f_2^*\bigr)
\le
-\left(y_2-\frac{\kappa}{\sqrt n}\right)_+
+
\frac{\beta}{\sqrt n}f_1^*.
\]
As we show below, $f_{11}^*\ge 0$. Hence, for each fixed $y_2$, the map $y_1\mapsto f_1^*(y_1,y_2)$ is increasing on $(-\infty,0]$, and therefore
\[
f_1^*(\mathbf{y})
\le
f_1^*(0,y_2)
=
\frac{\sqrt n}{\beta}\left(y_2-\frac{\kappa}{\sqrt n}\right)_+.
\]
Substituting this bound into \eqref{eq:f1minusf2bound} yields $m_2y_2(f_1^*-f_2^*)\le 0$.

On $\Omega_2$, $f_{11}^*=0$, since $f^*$ does not depend on $y_1$ there, and $f_{22}^*(\mathbf{y})=\kappa/(m_2\sqrt n\,y_2^2)$. On $\Omega_3$, differentiating \eqref{eq:f1} and \eqref{eq:f2} gives
\begin{align}
f_{11}^*
&=
-\frac{\sqrt n(m_1+m_2)}{\beta}\,
\tau_1e^{-m_1\tau}
\left(
z-\frac{m_1\kappa}{(m_1+m_2)\sqrt n}
\right),
\label{eq:f11}\\
f_{12}^*
&=
\frac{\sqrt n}{\beta}e^{-m_1\tau}
\left[
e^{-m_2\tau}
-\tau_2\left((m_1+m_2)z-\frac{m_1\kappa}{\sqrt n}\right)
\right].
\label{eq:f12}
\end{align}
We now verify the signs of these second derivatives. In \eqref{eq:f11}, the prefactor is negative, $\tau_1\le 0$ by \eqref{eq:tau12}, and $e^{-m_1\tau}>0$. Moreover, since $z\ge \kappa/\sqrt n$,
\[
z-\frac{m_1\kappa}{(m_1+m_2)\sqrt n}
\ge
\frac{\kappa}{\sqrt n}\frac{m_2}{m_1+m_2}>0.
\]
Hence $f_{11}^*\ge 0$. Similarly, in \eqref{eq:f12}, the prefactor is positive, the first term in brackets is strictly positive, and the second term is nonnegative because $\tau_2\le 0$ and $(m_1+m_2)z-m_1\kappa/\sqrt n\ge m_2\kappa/\sqrt n>0$. Therefore $f_{12}^*\ge 0$.

Finally, define
\[
\theta_2
:=
\frac{\partial \theta}{\partial y_2}
=
\frac{m_1m_2^2\bigl(e^{-m_2\tau}-e^{-m_1\tau}\bigr)^2}
{d^2\bigl(m_2z-\beta/\sqrt n\bigr)}
\ge 0.
\]
Differentiating \eqref{eq:f2} with respect to $y_2$, we obtain
\begin{equation}\label{eq:f22}
f_{22}^*
=
e^{-m_2\tau}\tau_2
\left(
1-\frac{\sqrt n\,m_2y_2\theta}{\beta}
\right)
+
\frac{\sqrt n}{\beta}e^{-m_2\tau}\theta
+
\frac{\sqrt n}{\beta}
\left(z-\frac{\kappa}{\sqrt n}\right)\theta_2.
\end{equation}
Since $\tau_2\le 0$, $\theta\ge e^{-m_2\tau}$, and $m_2y_2e^{-m_2\tau}=m_2z\ge \beta/\sqrt n$, we have $1-\sqrt n\,m_2y_2\theta/\beta\le 0$. Therefore the first term in \eqref{eq:f22} is nonnegative, and the second and third terms are also nonnegative because $\theta,\theta_2\ge 0$. We conclude that $f_{22}^*\ge 0$.

It remains to derive the upper bounds on the second derivatives. Substituting \eqref{eq:tau12} into \eqref{eq:f11}, we obtain
\[
f_{11}^*
=
\frac{\sqrt n(m_1+m_2)}{\beta}
\frac{e^{-2m_1\tau}}{m_2z-\beta/\sqrt n}
\left(
z-\frac{m_1\kappa}{(m_1+m_2)\sqrt n}
\right).
\]
Since $0<z-\frac{m_1\kappa}{(m_1+m_2)\sqrt n}\le z$ and $e^{-2m_1\tau}\le 1$, it follows that
\[
f_{11}^*
\le
\frac{\sqrt n(m_1+m_2)}{\beta}
\frac{z}{m_2z-\beta/\sqrt n}.
\]
Now the function $g(z):=z/(m_2z-\beta/\sqrt n)$ is decreasing on $(\beta/(m_2\sqrt n),\infty)$. Since $z\ge \kappa/\sqrt n$, we have
\[
g(z)\le g\!\left(\frac{\kappa}{\sqrt n}\right)
=
\frac{1}{m_2}\frac{\kappa}{\kappa-\beta/m_2}.
\]
Therefore
\[
f_{11}^*
\le
\frac{1}{\beta}
\frac{m_1+m_2}{m_2}
\frac{\kappa}{\kappa-\beta/m_2}\sqrt n
=:C_5(\kappa)\sqrt n.
\]

We next bound $f_{22}^*$. On $\Omega_2$, we already have $f_{22}^*(\mathbf{y})=\kappa/(m_2\sqrt n\,y_2^2)\le \sqrt n/(m_2\kappa)$, since $y_2\ge \kappa/\sqrt n$ on $\Omega_2$. Now let $\mathbf{y}\in\Omega_3$. We bound the three summands in \eqref{eq:f22} separately. For the second summand, $\theta(\mathbf{y})\le (m_1\vee m_2)/|d|$, and hence
\[
\frac{\sqrt n}{\beta}e^{-m_2\tau}\theta
\le
\frac{\sqrt n}{\beta}\frac{m_1\vee m_2}{|d|}.
\]
For the third summand, recall that
\[
\theta_2
=
\frac{m_1m_2^2\bigl(e^{-m_2\tau}-e^{-m_1\tau}\bigr)^2}
{d^2\bigl(m_2z-\beta/\sqrt n\bigr)}.
\]
Since $(e^{-m_2\tau}-e^{-m_1\tau})^2\le 1$ and $0\le (z-\kappa/\sqrt n)/(m_2z-\beta/\sqrt n)\le 1/m_2$, it follows that
\[
\frac{\sqrt n}{\beta}
\left(z-\frac{\kappa}{\sqrt n}\right)\theta_2
\le
\frac{\sqrt n}{\beta}\frac{m_1m_2}{d^2}.
\]
Consequently, the second and third summands together satisfy
\[
\frac{\sqrt n}{\beta}\frac{m_1\vee m_2}{|d|}
+
\frac{\sqrt n}{\beta}\frac{m_1m_2}{d^2}
=
\frac{\sqrt n}{\beta}\frac{(m_1\vee m_2)^2}{d^2},
\]
because $(m_1\vee m_2)|d|+m_1m_2=(m_1\vee m_2)^2$.

It remains to bound the first summand in \eqref{eq:f22}. We distinguish the two cases $m_2>m_1$ and $m_1>m_2$. If $m_2>m_1$, then $d<0$ and
\[
-\tau_2
=
\frac{m_2\bigl(e^{-m_1\tau}-e^{-m_2\tau}\bigr)}
{(m_2-m_1)\bigl(m_2z-\beta/\sqrt n\bigr)}.
\]
Since the first summand in \eqref{eq:f22} is nonnegative,
\[
0\le
e^{-m_2\tau}\tau_2
\left(
1-\frac{\sqrt n\,m_2y_2\theta}{\beta}
\right)
\le
\frac{\sqrt n}{\beta}
e^{-m_2\tau}(-\tau_2)m_2y_2\theta.
\]
Using $y_2e^{-m_2\tau}=z$, $\theta\le m_2/(m_2-m_1)$, and $e^{-m_1\tau}-e^{-m_2\tau}\le 1$, we obtain
\[
e^{-m_2\tau}(-\tau_2)m_2y_2\theta
\le
\left(\frac{m_2}{m_2-m_1}\right)^2
\frac{m_2z}{m_2z-\beta/\sqrt n}.
\]
Moreover, $r\mapsto m_2r/(m_2r-\beta/\sqrt n)$ is decreasing on $(\beta/(m_2\sqrt n),\infty)$, and $z\ge \kappa/\sqrt n$, so
\[
\frac{m_2z}{m_2z-\beta/\sqrt n}
\le
\frac{\kappa}{\kappa-\beta/m_2}.
\]
Therefore,
\[
0\le
e^{-m_2\tau}\tau_2
\left(
1-\frac{\sqrt n\,m_2y_2\theta}{\beta}
\right)
\le
\frac{\sqrt n}{\beta}
\left(\frac{m_2}{m_2-m_1}\right)^2
\frac{\kappa}{\kappa-\beta/m_2}.
\]

If $m_1>m_2$, then $d>0$ and
\[
-\tau_2
=
\frac{m_2\bigl(e^{-m_2\tau}-e^{-m_1\tau}\bigr)}
{(m_1-m_2)\bigl(m_2z-\beta/\sqrt n\bigr)}.
\]
Arguing as above,
\[
0\le
e^{-m_2\tau}\tau_2
\left(
1-\frac{\sqrt n\,m_2y_2\theta}{\beta}
\right)
\le
\frac{\sqrt n}{\beta}
e^{-m_2\tau}(-\tau_2)m_2y_2\theta.
\]
Using $y_2e^{-m_2\tau}=z$, $\theta\le m_1/(m_1-m_2)$, and $e^{-m_2\tau}-e^{-m_1\tau}\le 1$, we obtain
\[
e^{-m_2\tau}(-\tau_2)m_2y_2\theta
\le
\frac{m_1m_2^2}{(m_1-m_2)^2}
\frac{z}{m_2z-\beta/\sqrt n}.
\]
Since $r\mapsto r/(m_2r-\beta/\sqrt n)$ is decreasing and $z\ge \kappa/\sqrt n$, it follows that
\[
\frac{z}{m_2z-\beta/\sqrt n}
\le
\frac{1}{m_2}\frac{\kappa}{\kappa-\beta/m_2}.
\]
Hence
\begin{align*}
0\le
e^{-m_2\tau}\tau_2
\left(
1-\frac{\sqrt n\,m_2y_2\theta}{\beta}
\right)
&\le
\frac{\sqrt n}{\beta}
\frac{m_1m_2}{(m_1-m_2)^2}
\frac{\kappa}{\kappa-\beta/m_2}\\
&\le
\frac{\sqrt n}{\beta}
\frac{m_1^2}{(m_1-m_2)^2}
\frac{\kappa}{\kappa-\beta/m_2}.
\end{align*}

Thus, in both cases, the first summand in \eqref{eq:f22} is bounded by
\[
\frac{\sqrt n}{\beta}
\frac{(m_1\vee m_2)^2}{d^2}
\frac{\kappa}{\kappa-\beta/m_2}.
\]
Combining this estimate with the bound on $\Omega_2$ and the bounds for the remaining two summands, we conclude that
\[
f_{22}^*(\mathbf{y})
\le
C_6(\kappa)\sqrt n,
\qquad
\mathbf{y}\in\Omega,
\]
where
\[
C_6(\kappa)
:=
\frac{1}{m_2\kappa}
+
\frac{1}{\beta}\frac{(m_1\vee m_2)^2}{d^2}
\left(
1+\frac{\kappa}{\kappa-\beta/m_2}
\right).
\]
Since $\kappa>\beta/m_2$, all denominators are strictly positive, and therefore $C_6(\kappa)<\infty$.

\end{proof}
\section{Positive Recurrence of JSQ-SSL and JSQ-FSL Diffusion Limits}
\begin{proposition}[Geometric ergodicity of the JSQ-FSL and JSQ-SSL diffusion limits] \label{prop:geometric_ergodicity_FSL_SSL_BBM_style}
Suppose Assumption~\ref{asm:arrival_service_rates} holds and that $\beta:=\nu+\varsigma>0$. For each $\pi\in\{\FSL,\SSL\}$, the two-dimensional reflected diffusion $(\tilde{X}^\pi_1,\tilde{X}^\pi_2)$ associated with the limiting process $\tilde{\mathbf{X}}^{\pi}$ is geometrically ergodic. In particular, both $\tilde{\mathbf{X}}^{\FSL}$ and $\tilde{\mathbf{X}}^{\SSL}$ are positive recurrent.
\end{proposition}

\begin{proof}
The positive recurrence of the diffusion limit for JSQ-FSL is proven in Theorem~3 of \cite{bhambay2025asymptotic}. The case for JSQ-SSL differs only in that the order of the coefficients in the diffusion limit is reversed. We show that exchanging the coefficients in the Lyapunov function of~\cite{bhambay2025asymptotic} yields the same result for JSQ-SSL. 

Proposition~\ref{prop:stationary_bound_JSQ_SSL} implies that $\tilde{X}_{b,i}^{\SSL,n}(\infty)\Rightarrow 0$ for all $i\geq 3$ and hence, we take $\tilde{X}_{b,i}^{\SSL,n}(\infty)= 0$ which reduces the diffusion limit to 

\begin{align}
            \label{eq:JSQ_SSL_dif1app} X_1^{\SSL}(t) & = X_1(0) - \beta t + \sqrt{2} W(t) - \mu_{\max} \int_0^t X_1^{\SSL}(s) \, ds + \mu_{\min} \int_0^t X_2^{\SSL}(s) \, ds - U_1^{{\SSL}}(t), \\
            \label{eq:JSQ_SSL_dif2app}X_2^{\SSL}(t) & = X_2(0) + U_1^{{\SSL}}(t) - \mu_{\min} \int_0^t X_2^{\SSL}(s) \, ds,
\end{align}
where $U_1^{\SSL}(t)\in \DD_{\RR_+}[0, \infty)$ is the unique nondecreasing, nonnegative regulator with
\[\int_0^\infty\II(X_1^{\SSL}(t)<0)\,dU_1^{\SSL}(t)=0.\]

As in \cite{bhambay2025asymptotic} and \cite{Braverman2020}, using Ito's lemma, we obtain the extended generator for the diffusion limit of JSQ-SSL as 
\[
\cG_{\SSL}f(\mf{x}) = (-\beta - \mu_{\max}x_1 + \mu_{\min}x_2)f_1(\mf{x}) - \mu_{\min}x_2 f_2(\mf{x}) + f_{11}(\mf{x}) 
\]
for all $f:\RR_-\times \RR_+\to \RR$ where $f_1$ and $f_2$ correspond to the partial derivatives with respect to $x_1$ and $x_2$, respectively, such that $f_1(0, x_2) = f_2(0,x_2)$. We write 
$\mathcal{L}f(\mf{x}):=(-\beta-\mu_{\max}x_1+\mu_{\min}x_2)f_1(\mf{x}) -\mu_{\min}x_2f_2(\mf{x})$ 
for the first-order part of  $\cG_{\SSL}$, so that $\cG_{\SSL}f(\mf{x})=\mathcal{L}f(\mf{x})+f_{11}(\mf{x})$.

The geometric ergodicity of the diffusion follows, if we can find a Lyapunov function $V(\mf{x})$ satisfying the boundary condition above such that
\begin{align}
\cG_{\SSL}V(\mf{x}) \leq -c V(\mf{x}) + d \II(\mf{x}\in \KK),\label{eq:foster_lyapunov}
\end{align}
for some $c,d>0$ and compact $\KK\subset \RR_-\times \RR_+$. Setting 
\[
    \kappa_1:=\frac{\beta}{\mu_{\min}}+\epsilon,
    \qquad
    \kappa_2:=\frac{\beta}{\mu_{\min}}+2\epsilon,
\]
so that $\kappa_2>\kappa_1>\beta/\mu_{\min}\ge\beta/\mu_{\max}$, we can define the compact set as 
\[
    \KK:=[-\kappa_2,0]\times[0,\kappa_2].
\]
As in \cite{bhambay2025asymptotic} and \cite{Braverman2020}, we use the Lyapunov function $V:\Omega\to[1,\infty)$ of the form
\[
    V(\mf{x})=\exp\!\big(\alpha\,(f^{(1)}(\mf{x})+f^{(2)}(\mf{x}))\big),
\]
with appropriate choices of $f^{(1)},f^{(2)}$ and $\alpha>0$. To be able to define $f^{(1)}$ and $f^{(2)}$ properly, we need the following result from~\cite{bhambay2025asymptotic} to introduce additional notation.

\begin{lemma}[Lemmas 9 and 10 in \cite{bhambay2025asymptotic}]\label{lem:SSL_flow}
Recall $\mu_{\min}<\mu_{\max}$ by Assumption~\ref{asm:arrival_service_rates}. For $\mf{x}=(x_1,x_2)\in\Omega$, let $\tau(\mf{x})\in[0,\infty]$ be the
hitting time of the reflecting boundary $\{x_1=0\}$ by the deterministic flow generated by the first-order part of $\cG_{\SSL}$,
\begin{equation}
\label{eq:def_tau}
    \tau(\mf{x}):=\inf\Big\{s\ge0:\ -\frac{\beta}{\mu_{\max}}
    +\Big(x_1+\frac{\beta}{\mu_{\max}}\Big)e^{-\mu_{\max}s}
    +\frac{\mu_{\min}x_2}{\mu_{\max}-\mu_{\min}}
     \big(e^{-\mu_{\min}s}-e^{-\mu_{\max}s}\big)=0\Big\}
\end{equation}
Consider the following system of nonlinear equations
\begin{align}
    &-\frac{\beta}{\mu_{\max}}
    +\Big(x_1+\frac{\beta}{\mu_{\max}}\Big)e^{-\mu_{\max}\tau}
    +\frac{\mu_{\min}x_2}{\mu_{\max}-\mu_{\min}}
      \big(e^{-\mu_{\min}\tau}-e^{-\mu_{\max}\tau}\big)=0,\label{eq:SSL_Gamma1}\\
    &x_2e^{-\mu_{\min}\tau}=\kappa.
    \label{eq:SSL_Gamma2}
\end{align}

For any $\kappa\geq \beta/\mu_{\min}$ and $x_1\leq 0$, \eqref{eq:SSL_Gamma1} and \eqref{eq:SSL_Gamma2} has a unique solution, which we
denote $(x_2^\ast(x_1,\kappa),\tau)\in[\kappa,\infty) \times [0, \infty)$. Hence, the set 
\[\Gamma_\kappa = \{\mf{x}\in \Omega: x_2 = x_2^*(x_1, \kappa)\}\]
is well-defined, and we write $\mf{x}\ge\Gamma_\kappa$ when $x_2\ge x_2^\ast(x_1,\kappa)$, and $\mf{x}\le\Gamma_\kappa$ similarly. In addition, 
\begin{enumerate}
    \item[i)] for $\kappa\ge\beta/\mu_{\min}$, $\Gamma_\kappa$ exists and is unique, and $\Gamma_{\kappa'}>\Gamma_\kappa$ whenever $\kappa'>\kappa$;
    \item[ii)] $\tau(\mf{x})<\infty$ if and only if  $\mf{x}\ge\Gamma_{\beta/\mu_{\min}}$, in which case $x_2e^{-\mu_{\min}\tau(\mf{x})}\ge\beta/\mu_{\min}$;
    \item[iii)] for $\kappa>\beta/\mu_{\min}$, $\tau$ is differentiable on $\{\mf{x}\ge\Gamma_\kappa\}$ with
    \begin{align}
        \tau_1(\mf{x}) &= \frac{\partial \tau(\mf{x})}{\partial x_1}
        =-\frac{e^{-\mu_{\max}\tau(\mf{x})}}
               {\mu_{\min}x_2e^{-\mu_{\min}\tau(\mf{x})}-\beta}\le0,\\
        \tau_2(\mf{x})&= \frac{\partial \tau(\mf{x})}{\partial x_2}=\frac{\mu_{\min}}{\mu_{\min}-\mu_{\max}}\,
         \frac{e^{-\mu_{\min}\tau(\mf{x})}-e^{-\mu_{\max}\tau(\mf{x})}}
              {\mu_{\min}x_2e^{-\mu_{\min}\tau(\mf{x})}-\beta}\le0,
    \end{align}
    where $\tau_1$ is the left derivative at $x_1=0$.
\end{enumerate}
\end{lemma}

The indicator $\II(\mf{x}\notin\KK)$ in \eqref{eq:foster_lyapunov} is not smooth, and hence it cannot be used directly. Following \cite{Braverman2020}, we also define the $C^1$ cut-off function $\varphi:\RR_+\to[0,1]$ as
\[
\varphi(u)=
\begin{cases}
0, & u\le \kappa_1,\\
4\left(\dfrac{u-\kappa_1}{\kappa_2-\kappa_1}\right)^2
-
4\left(\dfrac{u-\kappa_1}{\kappa_2-\kappa_1}\right)^3,
& u\in\left[\kappa_1,\dfrac{\kappa_1+\kappa_2}{2}\right],\\
1
-
4\left(\dfrac{\kappa_2-u}{\kappa_2-\kappa_1}\right)^2
+
4\left(\dfrac{\kappa_2-u}{\kappa_2-\kappa_1}\right)^3,
& u\in\left[\dfrac{\kappa_1+\kappa_2}{2},\kappa_2\right],\\
1, & u\ge \kappa_2.
\end{cases}
\]
With $\varphi$ as above, we use the Lyapunov function $V(\mf{x})=\exp(\alpha(f^{(1)}(\mf{x})+f^{(2)}(\mf{x})))$ where $f^{(1)}$ and $f^{(2)}$ are associated with $-\varphi(-x_1)$ and $-\varphi(x_2)$, respectively. 
For $\kappa>\beta/\mu_{\max}$, let $\tau^{\kappa}(\mf{x})\in[0,\infty]$ be 
\begin{align*}
    \tau^{\kappa}(\mf x) &= \min\left\{t\geq 0: \frac{\beta}{\mu_M }-\kappa \leq \left(x_1 + \frac{\beta}{\mu_M} +\frac{\mu_1x_2}{\mu_1-\mu_M}\left(1-e^{-(\mu_1-\mu_M)t}\right)\right)e^{-\mu_M t}\right\}.
\end{align*}
Writing
\[
    \psi_{\mf{x}}(t):=\frac{\beta}{\mu_{\max}}
        -\Big(x_1+\frac{\beta}{\mu_{\max}}\Big)e^{-\mu_{\max}t}
        -\frac{\mu_{\min}x_2}{\mu_{\max}-\mu_{\min}}
         \big(e^{-\mu_{\min}t}-e^{-\mu_{\max}t}\big)
\]
for the negative of the first coordinate of the flow started at $\mf{x}$, and carrying out the integration using the crossing times $\tau^{\kappa_1},\tau^{\kappa_2}$ and the hitting time $\tau$ of Lemma~\ref{lem:SSL_flow}, $f^{(1)}$ admits the closed form
\begin{equation} \label{eq:SSL_f1}
    f^{(1)}(\mf{x})=
    \begin{cases}
        \displaystyle
        \tau^{\kappa_2}(\mf{x})
        +\int_{\tau^{\kappa_2}(\mf{x})}^{\tau^{\kappa_1}(\mf{x})}
          \varphi\big(\psi_{\mf{x}}(t)\big)\,dt,
        & x_1<-\kappa_2,\\
        \displaystyle
        \int_0^{\tau^{\kappa_1}(\mf{x})}\varphi\big(\psi_{\mf{x}}(t)\big)\,dt,
        & x_1\in[-\kappa_2,-\kappa_1],\\
        0, & x_1\in[-\kappa_1,0].
    \end{cases}
\end{equation}
The function $f^{(2)}$ is obtained in the same way, integrating $\varphi$ along the second coordinate $x_2e^{-\mu_{\min}t}$ of the flow. Here the structure of the trajectory depends on the starting point through the curves $\Gamma_{\kappa_1}, \Gamma_{\kappa_2}$ of Lemma~\ref{lem:SSL_flow}, so we partition $\Omega=S_0\cup S_1\cup S_2\cup S_3$ with
\[
    S_0=\{x_2\le\kappa_1\},\quad
    S_1=\{x_2>\kappa_1,\ \mf{x}\le\Gamma_{\kappa_1}\},\quad
    S_2=\{x_2>\kappa_1,\ \Gamma_{\kappa_1}\le\mf{x}\le\Gamma_{\kappa_2}\},\quad
    S_3=\{\mf{x}\ge\Gamma_{\kappa_2}\},
\]
where the ordering $\Gamma_{\kappa_1}\le\mf{x}\le\Gamma_{\kappa_2}$ in $S_2$ uses $\Gamma_{\kappa_2}>\Gamma_{\kappa_1}$ from Lemma~\ref{lem:SSL_flow}(i). Then
\begin{equation}
    f^{(2)}(\mf{x})=
    \begin{cases}
        0, & \mf{x}\in S_0,\\
        \displaystyle
        \int_0^{\frac{1}{\mu_{\min}}\log(x_2/\kappa_1)}
          \varphi\big(x_2e^{-\mu_{\min}t}\big)\,dt,
        & x_2\le\kappa_2,\ \mf{x}\in S_1,\\
        \displaystyle
        \frac{1}{\mu_{\min}}\log\frac{x_2}{\kappa_2}
        +\int_{\frac{1}{\mu_{\min}}\log(x_2/\kappa_2)}^{\infty}
          \varphi\big(x_2e^{-\mu_{\min}t}\big)\,dt,
        & x_2\ge\kappa_2,\ \mf{x}\in S_1,\\
        \displaystyle
        \int_0^{\tau(\mf{x})}\varphi\big(x_2e^{-\mu_{\min}t}\big)\,dt
        +\frac{1}{\beta}\int_{\kappa_1}^{x_2e^{-\mu_{\min}\tau(\mf{x})}}
          \varphi(u)\,du,
        & x_2\le\kappa_2,\ \mf{x}\in S_2,\\
        \displaystyle
        \frac{1}{\mu_{\min}}\log\frac{x_2}{\kappa_2}
        +  \int_{\frac{1}{\mu_{\min}}\log(x_2/\kappa_2)}^{\tau(\mf{x})}  
          \varphi\big(x_2e^{-\mu_{\min}t}\big)\,dt
        +\frac{1}{\beta} \int_{\kappa_1}^{x_2e^{-\mu_{\min}\tau(\mf{x})}}  
          \varphi(u)\,du,
        & x_2\ge\kappa_2,\ \mf{x}\in S_2,\\
        \displaystyle
        \tau(\mf{x})+\frac{x_2e^{-\mu_{\min}\tau(\mf{x})}-\kappa_2}{\beta}
        +\frac{1}{\beta}\int_{\kappa_1}^{\kappa_2}\varphi(u)\,du,
        & \mf{x}\in S_3.
    \end{cases}
    \label{eq:SSL_f2}
\end{equation}
Having defined $f^{(1)}$ and $f^{(2)}$, the following lemma collects the properties needed to verify the Foster-Lyapunov inequality \eqref{eq:foster_lyapunov}. Its proof follows that of Lemma~G.1 of \cite{bhambay2025asymptotic} under the substitution $\mu_M\mapsto\mu_{\max}$, $\mu_1\mapsto\mu_{\min}$ and is omitted.

\begin{lemma}(Lemma G.1 in \cite{bhambay2025asymptotic})\label{lem:SSL_DFL}
For any $\kappa_2>\kappa_1>\beta/\mu_{\min}$, the functions $f^{(1)},f^{(2)}$ defined in \eqref{eq:SSL_f1}-\eqref{eq:SSL_f2} have  absolutely continuous first-order derivatives and solve
\begin{align}
    \mathcal{L}f^{(1)}(\mf{x}) &= -\varphi(-x_1),
    & f^{(1)}_1(0,x_2) &= f^{(1)}_2(0,x_2),
    \label{eq:SSL_PDE_f1}\\
    \mathcal{L}f^{(2)}(\mf{x}) &= -\varphi(x_2),
    & f^{(2)}_1(0,x_2) &= f^{(2)}_2(0,x_2),
    \label{eq:SSL_PDE_f2}
\end{align}
for $\mf{x}\in\Omega$ and $x_2\ge 0$. Moreover, $\kappa_1$ and $\kappa_2$ can be chosen so that
\begin{align}
    |f^{(1)}(\mf{x})|
    &\le \frac{1}{\mu_{\max}}\log 2,
    & \mf{x}\in[-\kappa_2,0]\times[0,\kappa_2],
    \label{eq:SSL_G8}\\
    |f^{(2)}(\mf{x})|
    &\le \frac{\log 2}{\mu_{\min}}+\frac{\epsilon}{\beta},
    & \mf{x}\in[-\kappa_2,0]\times[0,\kappa_2],
    \label{eq:SSL_G9}
\end{align}
and, for all $\mf{x}\in\Omega$,
\begin{align}
    |f^{(1)}_1(\mf{x})|
    &\le \frac{4\log 2}{\mu_{\max}\,\epsilon},
    & |f^{(1)}_{11}(\mf{x})|
    &\le \frac{12\log 2}{\mu_{\max}\,\epsilon^2},
    \label{eq:SSL_G10}\\
    |f^{(2)}_1(\mf{x})|
    &\le \frac{1}{\beta},
    & |f^{(2)}_{11}(\mf{x})|
    &\le \frac{1}{\beta\,\epsilon}\Big(\frac{\mu_{\max}}{\mu_{\min}}+8\Big).
    \label{eq:SSL_G11}
\end{align}
\end{lemma}

We now verify that $V$ satisfies the Foster-Lyapunov inequality
\eqref{eq:foster_lyapunov}. For convenience we write $f^{(\Sigma)}(\mf{x}):=f^{(1)}(\mf{x})+f^{(2)}(\mf{x})$. Since
$f^{(r)}_1(0,x_2)=f^{(r)}_2(0,x_2)$ for $r=1,2$ by
Lemma~\ref{lem:SSL_DFL}, we have
$V_1(0,x_2)=V_2(0,x_2)$, so $V$ lies in the domain
of $\cG_{\SSL}$, and
\begin{align}
    \cG_{\SSL}V(\mf{x})
    &= \big(\mathcal{L}f^{(\Sigma)}(\mf{x})\big)\alpha V(\mf{x})
      +\Big(\alpha f^{(\Sigma)}_{11}(\mf{x})
            +\alpha^2\big(f^{(\Sigma)}_1(\mf{x})\big)^2\Big)V(\mf{x}) \notag\\
    &= -\big(\varphi(-x_1)+\varphi(x_2)\big)\alpha V(\mf{x})
      +\Big(\alpha f^{(\Sigma)}_{11}(\mf{x})
            +\alpha^2\big(f^{(\Sigma)}_1(\mf{x})\big)^2\Big)V(\mf{x})\label{eq:GV-explicit}.
\end{align}
where the second equality uses $\mathcal{L}f^{(\Sigma)}= \mathcal{L}f^{(1)}+\mathcal{L}f^{(2)}= -\varphi(-x_1)-\varphi(x_2)$ by the linearity of $\mathcal{L}$ and the equations \eqref{eq:SSL_PDE_f1}-\eqref{eq:SSL_PDE_f2}. By Lemma~\ref{lem:SSL_DFL}, on all of $\Omega$
\[
    |f^{(\Sigma)}_1(\mf{x})|\le C_1(\epsilon)
        :=\frac{4\log 2}{\mu_{\max}\,\epsilon}+\frac{1}{\beta},
    \qquad
    |f^{(\Sigma)}_{11}(\mf{x})|\le C_{11}(\epsilon)
        :=\frac{12\log 2}{\mu_{\max}\,\epsilon^2}
         +\frac{1}{\beta\,\epsilon}\Big(\frac{\mu_{\max}}{\mu_{\min}}+8\Big),
\]
and $C_{11}(\epsilon)\to0$ as $\epsilon\to\infty$, whereas $C_1(\epsilon)\to1/\beta$ remains bounded; this is all we shall need, since $C_{11}$ is made small by taking $\epsilon$ large and the term involving $C_1$ is then controlled by taking $\alpha$ small.

If $\mf{x}\notin\KK$, then $-x_1>\kappa_2$ or $x_2>\kappa_2$, so $\varphi(-x_1)+\varphi(x_2)\ge1$. Bounding the second term in \eqref{eq:GV-explicit} using $|f^{(\Sigma)}_1|\le C_1(\epsilon)$ and $|f^{(\Sigma)}_{11}|\le C_{11}(\epsilon)$,
\[
    \cG_{\SSL}V(\mf{x})
    \le\alpha\big(-1+C_{11}(\epsilon)+\alpha C_1(\epsilon)^2\big)V(\mf{x}).
\]
Choose $\epsilon$ large enough that $\epsilon\ge\beta/\mu_{\min}$ and $C_{11}(\epsilon)\le\tfrac12$, and then $\alpha\in(0,1)$ small enough that $\alpha C_1(\epsilon)^2\le\tfrac14$. With these choices $-1+C_{11}(\epsilon)+\alpha C_1(\epsilon)^2\le-\tfrac14$, so that
\begin{equation}
    \cG_{\SSL}V(\mf{x})\le -c\,V(\mf{x}),
    \qquad \mf{x}\notin\KK,
    \label{eq:drift_off_K}
\end{equation}
with $c:=\alpha/4>0$.

It remains to control the generator on $\KK$. By \eqref{eq:SSL_G8}-\eqref{eq:SSL_G9},
\[
    f^{(\Sigma)}(\mf{x})\le\frac{\log 2}{\mu_{\max}}
        +\frac{\log 2}{\mu_{\min}}+\frac{\epsilon}{\beta},
    \qquad \mf{x}\in\KK,
\]
so that $V(\mf{x})=\exp(\alpha f^{(\Sigma)}(\mf{x}))\le M_\KK$ on $\KK$, where
\[
    M_\KK:=\exp\!\Big(\alpha\Big(\frac{\log 2}{\mu_{\max}}
        +\frac{\log 2}{\mu_{\min}}+\frac{\epsilon}{\beta}\Big)\Big)<\infty.
\]
Discarding the nonpositive first term in \eqref{eq:GV-explicit} and using $|f^{(\Sigma)}_1|\le C_1(\epsilon)$, $|f^{(\Sigma)}_{11}|\le C_{11}(\epsilon)$, together with $V\le M_\KK$,
\[
    \cG_{\SSL}V(\mf{x})+cV(\mf{x})
    \le\big(\alpha C_{11}(\epsilon)+\alpha^2 C_1(\epsilon)^2+c\big)M_\KK
    =:d<\infty,
    \qquad \mf{x}\in\KK.
\]
Combining this with \eqref{eq:drift_off_K} gives the Foster-Lyapunov inequality \eqref{eq:foster_lyapunov},
\[
    \cG_{\SSL}V(\mf{x})\le -c\,V(\mf{x})+d\,\II(\mf{x}\in\KK),
    \qquad \mf{x}\in\Omega.
\]
It remains to check that $V(\mf{x})\to\infty$ as $|\mf{x}|\to\infty$ on $\Omega$. Since $V=\exp(\alpha f^{(\Sigma)})$ with $\alpha>0$, it is enough to show $f^{(\Sigma)}(\mf{x})\to\infty$; as $f^{(1)},f^{(2)}\ge0$, we bound $f^{(\Sigma)}$ below by either one. There are two ways for $|\mf{x}|\to\infty$ on $\Omega$.

First, suppose $x_2\to\infty$. If $\mf{x}\notin S_3$, then, for $x_2\ge\kappa_2$ the corresponding cases of \eqref{eq:SSL_f2} contain the term $\mu_{\min}^{-1}\log(x_2/\kappa_2)$. Since all the remaining terms are nonnegative,
\[
    f^{(\Sigma)}(\mf{x})\ge f^{(2)}(\mf{x})
    \ge\frac{1}{\mu_{\min}}\log\frac{x_2}{\kappa_2}
    \longrightarrow\infty
    \qquad\text{as } x_2\to\infty.
\]
It remains to consider the case $\mf{x}\in S_3$. The last case of \eqref{eq:SSL_f2} gives
\[
    f^{(2)}(\mf{x})
    =
    \tau(\mf{x})
    +\frac{x_2e^{-\mu_{\min}\tau(\mf{x})}-\kappa_2}{\beta}
    +\frac{1}{\beta}\int_{\kappa_1}^{\kappa_2}\varphi(u)\,du
    \ge
    \tau(\mf{x})+\frac{x_2e^{-\mu_{\min}\tau(\mf{x})}-\kappa_2}{\beta}.
\]
Now consider any sequence with $x_2 \to \infty$ and $\mf{x}\in S_3$. If $\tau(\mf{x})\to \infty$ along the sequence, then $f^{(2)}(\mf{x})\to \infty$. Otherwise, there exists a subsequence along which $\tau(\mf{x})$ is bounded. Along this subsequence, $x_2e^{-\mu_{\min}\tau(\mf{x})}\to \infty$, and hence again $f^{(2)}(\mf{x})\to \infty$. Therefore, every sequence with $x_2\to \infty$ satisfies $f^{(2)}(\mf{x})\to \infty$, and so 
\[f^{(\Sigma)}(\mf{x})\geq f^{(2)}(\mf{x})\to \infty.\]

Next, suppose $x_2$ stays bounded and $-x_1\to\infty$. Then, for all sufficiently large $-x_1$, we have $x_1<-\kappa_2$, and hence the first case of \eqref{eq:SSL_f1} applies. Since $\varphi= 1$ on $[\kappa_2, \infty)$ and $\varphi\geq 0$, this gives
\[f^{(1)}(\mf{x})\geq \tau^{\kappa_2}(\mf{x}).\]
We claim that $\tau^{\kappa_2}(\mf{x})\to\infty$ as $-x_1\to\infty$ with $x_2$ bounded. Fix any $T<\infty$. For $t\in[0, T]$, the first coordinate of the deterministic flow is 
\[
    -\frac{\beta}{\mu_{\max}}
    +
    \Big(x_1+\frac{\beta}{\mu_{\max}}\Big)e^{-\mu_{\max}t}
    +
    \frac{\mu_{\min}x_2}{\mu_{\max}-\mu_{\min}}
    \big(e^{-\mu_{\min}t}-e^{-\mu_{\max}t}\big).
\]
If $x_2$ is bounded, then the last term is uniformly bounded on $[0,T]$, whereas the term
\[
    \Big(x_1+\frac{\beta}{\mu_{\max}}\Big)e^{-\mu_{\max}t}
\]
tends to $-\infty$ uniformly on $[0,T]$ as $x_1\to-\infty$. Therefore, for all sufficiently large $-x_1$, the first coordinate of the flow remains below $-\kappa_2$ throughout $[0,T]$. Hence $\tau^{\kappa_2}(\mf{x})>T$. Since $T<\infty$ was arbitrary,
\[
    \tau^{\kappa_2}(\mf{x})\longrightarrow\infty .
\]
Consequently,
\[
    f^{(\Sigma)}(\mf{x})
    \ge f^{(1)}(\mf{x})
    \ge \tau^{\kappa_2}(\mf{x})
    \longrightarrow\infty .
\]
Combining the two alternatives, we obtain $f^{(\Sigma)}(\mf{x})\to\infty$ whenever $|\mf{x}|\to\infty$ in $\Omega$. Hence $V(\mf{x})\to\infty$ as $|\mf{x}|\to\infty$. The reflected diffusion is irreducible and its compact sets are petite, so by Theorem~5.2 of \cite{down1995exponential}, the Foster-Lyapunov inequality \eqref{eq:foster_lyapunov}, together with $V(\mf{x})\to\infty$, implies that $\tilde{\mathbf{X}}^{\SSL}$ is geometrically ergodic, and in particular positive recurrent. Together with the JSQ-FSL case from Theorem~3 of \cite{bhambay2025asymptotic}, both $\tilde{\mathbf{X}}^{\FSL}$ and $\tilde{\mathbf{X}}^{\SSL}$ are positive recurrent.

\end{proof}

The following property of the Lyapunov function is necessary in our proof of tightness of $\xi$.
\begin{lemma}\label{lem:superlinear_lyapunov}
    There exists $\alpha, \tilde{C}_1, \tilde{C}_2 \geq 0$ such that for all $\mf{x}\in \RR_-\times \RR_+$, we have 
    \begin{align}
        x_2 \leq \tilde{C}_1 V(\mf{x}) + \tilde{C}_2.
    \end{align}
\end{lemma}

\begin{proof}
By definition, $f^{(1)}(\mf{x})\geq 0$ for all $x\in \RR_-\times \RR_+$ and hence, for any $\alpha>0$,
\begin{equation}
 V(\mf{x})=\exp\!\big(\alpha\,(f^{(1)}(\mf{x})+f^{(2)}(\mf{x}))\big) \geq \exp\!\big(\alpha f^{(2)}(\mf{x})\big).
 \label{eq:V_bound}
\end{equation}
We now wish to show that we can find a $\tilde{C}_2>\kappa_2$ such that for any $\mf{x}$ where $x_2\geq \tilde{C}_2$, we have 
\begin{equation}
f^{(2)}(\mf{x})\geq \frac{1}{\mu_{\max}}\log\frac{x_2}{\kappa_2}.\label{eq:lyapunov_f2_bound}
\end{equation}
If $\mf{x}\in S_1\cup S_2$, positivity of $\varphi$ implies \eqref{eq:lyapunov_f2_bound}. Similarly, if $\mf{x}\in S_3$ 
and $\tau(\mf{x})\geq \frac{1}{\mu_{\max}}\log\frac{x_2}{\kappa_2} $, again positivity of $\varphi$ and the definition of $\tau(\mf{x})$ imply \eqref{eq:lyapunov_f2_bound} for any $\epsilon$. Finally, if $\mf{x}\in S_3$ and $\tau(x)<\frac{1}{\mu_{\max}}\log\frac{x_2}{\kappa_2}$,
\begin{align*}
    f^{(2)}(\mf{x})
    &\geq \frac{x_2e^{-\mu_{\min}\tau(\mf{x})}-\kappa_2}{\beta}\\
    &>
    \frac{x_2\left(\frac{\kappa_2}{x_2}\right)^{\mu_{\min}/\mu_{\max}}-\kappa_2}{\beta}\\
    &=
    \frac{\kappa_2^{\mu_{\min}/\mu_{\max}}
    x_2^{1-\mu_{\min}/\mu_{\max}}
    -\kappa_2}{\beta}.
\end{align*}
Since
\[
1-\frac{\mu_{\min}}{\mu_{\max}}>0,
\]
the right-hand side grows faster than logarithmically in $x_2$ and therefore dominates
\[
\frac{1}{\mu_{\max}}\log\frac{x_2}{\kappa_2}
\]
for all sufficiently large $x_2$. Hence, there exists $\tilde {C}_2>\kappa_2$ such that
\[
f^{(2)}(\mf{x})
\ge
\frac{1}{\mu_{\max}}\log\frac{x_2}{\kappa_2},
\mbox{ for all }
x_2\geq \tilde{C}_2.
\]
Plugging in \eqref{eq:V_bound}, for $\alpha=\mu_{\max}$ and $x_2\geq \tilde{C}_2$ we have 
\[
V(\mf{x})
\geq
\exp\left(
\mu_{\max}\frac{1}{\mu_{\max}}\log\frac{x_2}{\kappa_2}
\right)
=
\frac{x_2}{\kappa_2}.
\]
Equivalently,
\[
x_2\leq \kappa_2 V(\mf{x}),
\qquad x_2\geq \tilde{C}_2.
\]
For $x_2<\tilde{C}_2$, we trivially have
\[
x_2\leq \tilde{C}_2.
\]
Therefore, for all $\mf{x}\in\RR_-\times\RR_+$,
\[
x_2\leq \tilde{C}_2+\kappa_2 V(\mf{x}).
\]
This proves the lemma.

\end{proof}

\section{Dependency Map of the Main Results}
\label{app:dependency_map}
\begin{tikzpicture}[node distance=1.1cm and 0.8cm,
  scale=0.58, every node/.style={transform shape}]


\node[prelim] (L35) {
  \textbf{Lemma \ref*{lem:stochastic_boundedness}}\\[2pt]
  Stochastic boundedness of $\mf{X}^{\pi,n}$};

\node[assumption, above=0.7cm of L35] (A) {
  \textbf{Assumptions \ref*{asm:arrival_service_rates}--\ref*{asm:xi_initial}}\\[2pt]
  Model scaling and initial conditions
};

\node[prelim, right=10cm of L35] (L39) {
  \textbf{Lemma \ref*{lem:alpha_continuity}}\\[2pt]
  Tightness of $\hat\alpha^n_1$};

\draw[arrassumption] (A) -- (L35);

\node[fairness, imported, below=0.5cm of L35] (L36) {
  \textbf{Lemma \ref*{lem:fairness_tightness}}\\[2pt]
  Tightness of $\eta^n$};
\draw[arrprelim] (L35) -- (L36);
\node[fairness, imported, below left=0.7cm and -1.2cm of L36] (P37) {
  \textbf{Prop.\ref*{prop:priority_idleness}}\\[2pt]
  FSF\,/\,SSF fairness};
\node[fairness, imported, below right=0.7cm and -1.2cm of L36] (P38) {
  \textbf{Prop. \ref*{prop:totally_blind_idleness}}\\[2pt]
  Totally blind fairness};
\draw[arrfair] (L36.south) -- ++(0,-0.3) -| (P37.north);
\draw[arrfair] (L36.south) -- ++(0,-0.3) -| (P38.north);

\node[proof, below=4cm of L39, xshift= -8cm] (L63) {
  \textbf{Lemma \ref*{lem:modified_reflection_equations}}\\[2pt]
  Contraction mapping};
\node[proof, left=1cm of L63] (L61) {
  \textbf{Lemma \ref*{lem:reflection_equations}}\\[2pt]
  E/U for $(x_1,u_1,\xi_2)$};
\node[proof, right=1cm of L63] (L62) {
  \textbf{Lemma \ref*{lem:reflection_equations2}}\\[2pt]
  E/U for $(\xi_i)_{i\ge3}$};
\draw[arrproof] (L63.west) -- (L61.east);

\node[proof, right=1cm of L62] (L64) {
  \textbf{Lemma \ref*{lem:martingale_convergence}}\\[2pt]
  Martingale CLT limits};

\node[mainresult,
  below=1.5cm of $(L61)!0.5!(L64)$] (T31) {
  Theorem \ref*{thm:process_level_convergence}\\[2pt]
  \normalfont General process-level convergence};

\draw[arrproof] (L61.south) -- ++(0,-0.4) -| ([xshift=-2.4cm]T31.north);
\draw[arrproof] (L62.south) -- ++(0,-0.4) -| ([xshift=0.0cm]T31.north);
\draw[arrproof] (L64.south) -- ++(0,-0.4) -| ([xshift=2.4cm]T31.north);

\draw[arrprelim] (L35.east) -- ++(0,0) -| ([xshift=-0.35cm]T31.north);

\coordinate (L39toT31Corner) at ([xshift=0.4cm]L64.east);

\draw[arrprelim] ([yshift= -0.4cm]L39.east)
  -- ([yshift= -0.4cm]L39.east -| L39toT31Corner)
  -- (L39toT31Corner)
  |- (T31.east);

\node[proof, below=3.5cm of T31, xshift=-4.8cm] (L65) {
  \textbf{Lemma \ref*{lem:A_1_bound}}\\[2pt]
  Stochastic boundedness of $\tau_s^n$
};
\coordinate (L35toL64Corner) at ([yshift=3cm]L64.north);
\draw[arrprelim]
  (L35.east)
  -- ++(5.58cm,0)
  |- (L35toL64Corner)
  -- (L64.north);

\node[corollary, below=0.8cm of T31, xshift=-9.0cm] (C42) {
  \textbf{Corollary \ref*{cor:SA-JSQ}}\\[2pt]
  JFSQ};
\node[corollary, below=0.8cm of T31, xshift=-5cm] (C43) {
  \textbf{Corollary \ref*{cor:JSQ_random}}\\[2pt]
  Random JSQ};
\node[corollary, below=0.8cm of T31, xshift=-1cm]  (C45) {
  \textbf{Corollary \ref*{cor:JFIQ}}\\[2pt]
  JFIQ };
\node[corollary, below=0.8cm of T31, xshift=3.0cm]  (C44) {
  \textbf{Corollary \ref*{cor:JSQ_random_pool}}\\[2pt]
  Pool-level}; 

\node[corollary, below=0.8cm of T31, xshift=7.0cm]  (C41) {
  \textbf{Corollary  \ref*{cor:tau_0_continuity}}\\[2pt]
   $\tau^{\pi}_0=0$};   

\draw[arrmain] (T31.south) -- ++(0,-0.4) -| (C42.north);
\draw[arrmain] (T31.south) -- ++(0,-0.4) -| (C43.north);
\draw[arrmain] (T31.south) -- ++(0,-0.4) -| (C45.north);
\draw[arrmain] (T31.south) -- ++(0,-0.4) -| (C41.north);

\draw[arrfair] ([xshift=-0.74cm]P37.south) |- ([xshift=-0.3cm]C42.west) -- (C42.west);
\draw[arrfair] ([xshift=1.6cm]P38.south) -- ++(0,0) -| ([xshift=0.5cm]C43.north);

\draw[arrcorollary] (C43.south) -- ++(0,-0.4) -|  (C44.south);


\draw[arrfair] ([xshift=-1.2cm]P37.south)
  |- ([yshift=-3cm]C45.south)
  -- (C45.south);

\coordinate (busMid) at ([yshift=0.65cm, xshift= -0.5cm]L65.north);

\coordinate (bus42) at (C42.south |- busMid);
\coordinate (bus43) at (C43.south |- busMid);
\coordinate (bus45) at (C45.south |- busMid);

\draw[cBlueBorder,  line width=0.5pt] ([xshift= -0.5cm]L65.north) -- (busMid);

\draw[cBlueBorder, line width=0.5pt] (bus42) -- ([xshift= -0.3cm]bus45);

\draw[arrproof] (bus42) -- (C42.south);
\draw[arrproof] ([xshift= -0.3cm]bus43) -- ([xshift= -0.3cm]C43.south);
\draw[arrproof] ([xshift= -0.3cm]bus45) -- ([xshift= -0.3cm]C45.south);

\coordinate (L39toBusCorner) at ([xshift=0.7cm]L64.east);

\draw[arrprelim] (L39.east)
  -- (L39.east -| L39toBusCorner)
  -- (L39toBusCorner)
  |- ([yshift= -0.25cm]busMid);

\node[
  circle,
  fill=cGrayBorder,
  inner sep=1.3pt
] at ([yshift=-0.25cm]busMid) {};

\draw[darrcorollary] (C43.east) -- (C45.west);


\coordinate (DIVC) at
  ([yshift=-2.4cm]$(L65.south)!0.5!(C44.south)$);

\node[font=\Large\bfseries] at (DIVC) {
  Stationary Analysis
};


\node[
  assumption,
  below=1.0cm of DIVC,
  align=center
] (SA) {
  \textbf{Stationary analysis assumptions}\\[2pt]
  Assumption~\ref*{asm:arrival_service_rates}\\
  and $\varsigma^n+\nu^n>\varepsilon>0$
};


\node[
  stationary,
  below=1.0cm of SA,
  xshift=-12cm
] (L39stat) {
  \textbf{Lemma \ref*{lem:positive_recurrent_finite}}\\[2pt]
  Positive recurrence of the\\
  prelimit JSQ system
};

\node[
  stationary,
  below=1.0cm of SA,
  xshift=-8cm
] (P51) {
  \textbf{Prop.\ \ref*{prop:upper_bound_system}}\\[2pt]
  FSL/SSL\\ sample-path bounds
};

\node[
  stationary,
  below=1.0cm of SA,
  xshift=-4.2cm
] (L72) {
  \textbf{Lemma \ref*{lem:x_1_lower_bound}}\\[2pt]
  Idle-server bound
};

\node[
  stationary,
  below=1.0cm of SA,
  xshift=-0.2cm
] (L53) {
  \textbf{Lemma \ref*{lem:jsq_ssl_monotonicity}}\\[2pt]
  JSQ--SSL \\positive recurrence\\
  and buffer monotonicity
};

\node[
  stationary,
  below=1.0cm of SA,
  xshift=3.8cm
] (LD2) {
  \textbf{Lemma \ref*{lem:solution_PDE}}\\[2pt]
  Lyapunov \\PDE solution
};

\node[
  stationary,
  below=1.0cm of SA,
  xshift=7.8cm
] (P52) {
  \textbf{Prop.\ \ref*{prop:diffusion_limit_FSL_SSL}}\\[2pt]
  FSL/SSL diffusion limits
};


\node[
  stationary,
  below=1cm of L53
] (LD1) {
  \textbf{Lemma \ref*{lem:rate_balance_q}}\\[2pt]
  Rate-balance equations
};

\node[
  stationary,
  below=1cm of LD1,
  xshift=2cm
] (P71) {
  \textbf{Prop.\ \ref*{prop:stationary_bound_JSQ_SSL}}\\[2pt]
  Stationary bounds for JSQ--SSL
};

\node[
  stationary,
  below=1cm of P52
] (PD1) {
  \textbf{Prop.\ \ref*{prop:geometric_ergodicity_FSL_SSL_BBM_style}}\\[2pt]
  Geometric ergodicity of\\
  FSL/SSL diffusion limits
};


\node[
  mainresult,
  below=1cm of P71,
  xshift=-7.1cm
] (T47) {
  Theorem \ref*{thm:stationary_bounds}\\[2pt]
  \normalfont Uniform steady-state bounds
};

\node[
  corollary,
  right=4.2cm of T47
] (C48) {
  \textbf{Corollary \ref*{cor:interchange_limits}}\\[2pt]
  Stationary tightness and\\
  subsequential limit identification
};

\node[
  mainresult,
  below=1cm of C48,
  xshift=-5.5cm
] (T49) {
  Theorem \ref*{thm:asymptotical_optimality}\\[2pt]
  \normalfont Asymptotic optimality of JFSQ
};

\node[
  assumption,
  left=1.2cm of T47,
  align=center
] (FB) {
  \textbf{Finite-buffer}\\[-1pt]
  condition
};


\coordinate (SABus) at ([yshift=-0.45cm]SA.south);

\draw[arrassumption,-]
  (SA.south) -- (SABus);

\draw[arrassumption]
  (SABus) -| (L39stat.north);

\draw[arrassumption]
  (SABus) -| (L72.north);

\draw[arrassumption]
  (SABus) -| (L53.north);

\coordinate (SAP71Drop) at (P71.north |- SABus);

\draw[arrassumption,-]
  (SABus) -- (SAP71Drop);

\draw[arrassumption]
  (SAP71Drop) -- (P71.north);


\draw[arrstationary]
  (L53.south) -- (LD1.north);

\draw[arrstationary]
  (LD1.south)
  -- ++(0,-0.35cm)
  -| ([xshift=-0.8cm]P71.north);

\draw[arrstationary]
  (LD2.south)
  -- ++(0,-0.35cm)
  -| ([xshift=0.8cm]P71.north);


\draw[arrstationary]
  (P52.south) -- (PD1.north);

\draw[arrstationary]
  (P71.east)
  -- ++(0.55cm,0)
  |- (PD1.west);


\draw[arrstationary]
  ([xshift=0.9cm]P51.south)
  -- ++(0,-0.4cm)
  -| ([xshift=-1.8cm]T47.north);

\draw[arrstationary]
  ([xshift=-0.6cm]L72.south)
  -- ++(0,-0.55cm)
  -| ([xshift=0.5cm]T47.north);


\coordinate (L53Out) at
  ([yshift=-0.58cm]L53.west);

\coordinate (T47In) at
  ([xshift=1cm]T47.north);

\coordinate (T47Turn) at
  (L53Out -| T47In);

\coordinate (T47Approach) at
  ([yshift=0.45cm]T47In);

\draw[arrstationary]
  (L53Out)
  -- (T47Turn)
  -- (T47Approach)
  -- (T47In);

\draw[arrstationary]
  (P71.south)
  -- ++(0,-0.4cm)
  -| ([xshift=1.7cm]T47.north);

\draw[arrassumption]
  (FB.east)
  --(T47.west);


\draw[arrmain]
  (T47.east) -- (C48.west);


\coordinate (T31Out) at
  ([yshift=-0.18cm]T31.east);

\coordinate (C48In) at
  ([yshift=0.18cm]C48.east);

\coordinate (RightLaneT31) at
  ([xshift=0.7cm]P52.east);

\draw[arrmain]
  (T31Out)
  -- (T31Out -| RightLaneT31)
  -- (C48In -| RightLaneT31)
  -- (C48In);


\coordinate (T49BusCenter) at
  ([yshift=0.75cm]T49.north);

\coordinate (T49BusP51) at
  (P51.south |- T49BusCenter);

\coordinate (T49BusPD1) at
  (PD1.south |- T49BusCenter);

\coordinate (RightLaneStationary) at
  ([xshift=0.35cm]P52.east);

\coordinate (RightLaneStationaryBus) at
  (T49BusCenter -| RightLaneStationary);

\draw[arrstationary,-]
  (T49BusP51) -- (RightLaneStationaryBus);


\coordinate (P51toT49Out) at
  ($(P51.south west)!0.18!(P51.south east)$);

\coordinate (T49BusP51) at
  (P51toT49Out |- T49BusCenter);

\draw[arrstationary,-]
  (P51toT49Out) -- (T49BusP51);

\draw[arrstationary,-]
  (T49BusP51) -- (RightLaneStationaryBus);
  -- (RightLaneStationaryBus);

\coordinate (L53TopLane) at
  ([xshift=0.65cm,yshift=0.35cm]L53.north);

\coordinate (L53RightCorner) at
  (L53TopLane -| RightLaneStationary);

\draw[arrstationary,-]
  ([xshift=0.65cm]L53.north)
  -- (L53TopLane)
  -- (L53RightCorner)
  -- (RightLaneStationaryBus);

\draw[arrstationary,-]
  (P52.east)
  -- (P52.east -| RightLaneStationary)
  -- (RightLaneStationaryBus);

\draw[arrstationary,-]
  (PD1.south) -- (T49BusPD1);

\draw[arrstationary]
  (T49BusCenter) -- (T49.north);


\coordinate (RightLaneC48) at
  ([xshift=-5.3cm]P52.east);

\draw[arrcorollary]
  (C48.south)
  -- (C48.south -| RightLaneC48)
  -- (T49.east -| RightLaneC48)
  -- (T49.east);


\coordinate (T49WestUpper) at
  ($(T49.north west)!0.35!(T49.south west)$);

\coordinate (L39StatLane) at
  ([xshift=-0.35cm]L39stat.west);

\coordinate (L39StatTurn) at
  (L39StatLane |- T49WestUpper);

\draw[arrstationary]
  (L39stat.west)
  -- (L39StatLane)
  -- (L39StatTurn)
  -- (T49WestUpper);


\coordinate (T49WestLower) at
  ($(T49.north west)!0.72!(T49.south west)$);

\coordinate (C42OuterLane) at
  ([xshift=-0.7cm]L39stat.west);

\coordinate (C42OuterTop) at
  (C42.west -| C42OuterLane);

\coordinate (C42OuterTurn) at
  (C42OuterLane |- T49WestLower);

\draw[arrcorollary]
  (C42.west)
  -- (C42OuterTop)
  -- (C42OuterTurn)
  -- (T49WestLower);


\node[
  legendbox,
  below=1cm of T49,
  anchor=north,
  minimum width=22.6cm,
  minimum height=2.65cm
] (legend) {};


\node[anchor=west, font=\small] at
  ([xshift=0.35cm,yshift=0.50cm]legend.west) {

  \tikz[baseline=-0.6ex]
    \node[
      assumption,
      minimum height=0.48cm,
      minimum width=1.0cm
    ] {};
  assumptions
  \hspace{0.28cm}

  \tikz[baseline=-0.6ex]
    \node[
      prelim,
      minimum height=0.48cm,
      minimum width=1.0cm
    ] {};
  prelim.
  \hspace{0.28cm}

  \tikz[baseline=-0.6ex]
    \node[
      fairness,
      minimum height=0.48cm,
      minimum width=1.0cm
    ] {};
  fairness
  \hspace{0.28cm}

  \tikz[baseline=-0.6ex]
    \node[
      proof,
      minimum height=0.48cm,
      minimum width=1.0cm
    ] {};
  transient tools
  \hspace{0.28cm}

  \tikz[baseline=-0.6ex]
    \node[
      stationary,
      minimum height=0.48cm,
      minimum width=1.0cm
    ] {};
  stationary tools
  \hspace{0.28cm}

  \tikz[baseline=-0.6ex]
    \node[
      mainresult,
      minimum height=0.48cm,
      minimum width=1.0cm
    ] {};
  main result
  \hspace{0.28cm}

  \tikz[baseline=-0.6ex]
    \node[
      corollary,
      minimum height=0.48cm,
      minimum width=1.0cm
    ] {};
  corollary
};


\node[anchor=west, font=\small] at
  ([xshift=0.35cm,yshift=-0.40cm]legend.west) {

  \tikz[baseline=-0.6ex]
    \node[
      fairness,
      imported,
      minimum height=0.48cm,
      minimum width=1.0cm
    ] {};
  adapted from B\"uke--Qin 23
  \hspace{0.55cm}

  \tikz[baseline=-0.45ex]
    \draw[arrcorollary] (0,0)--(1.0,0);
  direct dependency
  \hspace{0.55cm}

  \tikz[baseline=-0.45ex]
    \draw[darrcorollary] (0,0)--(1.0,0);
  indirect dependency
  \hspace{0.55cm}

  \itshape Arrow color indicates the source-node category
};

\end{tikzpicture}


\newpage
\bibliographystyle{chicago}
\bibliography{bibliography}

@techreport{CloudNativeSilicon2023,
  author      = {Abbott, John},
  title       = {Cloud-Native Silicon Advances a Heterogeneous Approach to Datacenter Processing for AI and Other New Workloads},
  institution = {S\&P Global Market Intelligence},
  year        = {2023},
  month       = sep,
  note        = {Vanguard Report, commissioned by AMD. Available at: \url{https://www.amd.com/en/solutions/data-center/insights/cloud-native-silicon-advances-heterogeneous-approach.html} (accessed June 25, 2026).}
}

@article{Armony2005,
  author  = {Armony, Mor},
  title   = {Dynamic Routing in Large-Scale Service Systems with Heterogeneous Servers},
  journal = {Queueing Systems},
  volume  = {51},
  number  = {3},
  pages   = {287-329},
  year    = {2005},
  doi     = {10.1007/s11134-005-3760-7},
  url     = {https://doi.org/10.1007/s11134-005-3760-7}
}

@article{AW2010,
  author  = {Armony, Mor and Ward, Amy R.},
  title   = {Fair Dynamic Routing in Large-Scale Heterogeneous-Server Systems},
  journal = {Operations Research},
  volume  = {58},
  number  = {3},
  pages   = {624-637},
  year    = {2010},
  doi     = {10.1287/opre.1090.0777},
  url     = {https://doi.org/10.1287/opre.1090.0777}
}

@article{Atar2008,
  author  = {Atar, Rami},
  title   = {Central limit theorem for a many-server queue with random service rates},
  journal = {The Annals of Applied Probability},
  volume  = {18},
  number  = {4},
  pages   = {1548-1568},
  year    = {2008},
  doi     = {10.1214/07-AAP497},
  url     = {https://doi.org/10.1214/07-AAP497}
}

@article{AtarShwartz2008,
  author  = {Atar, Rami and Shwartz, Adam},
  title   = {Efficient routing in heavy traffic under partial sampling of service times},
  journal = {Mathematics of Operations Research},
  volume  = {33},
  number  = {4},
  pages   = {899-909},
  year    = {2008},
  doi     = {10.1287/moor.1080.0325},
  url     = {https://doi.org/10.1287/moor.1080.0325}
}

@article{AtarShakiShwartz2011,
  author  = {Atar, Rami and Shaki, Yair Y. and Shwartz, Adam},
  title   = {A blind policy for equalizing cumulative idleness},
  journal = {Queueing Systems},
  volume  = {67},
  number  = {4},
  pages   = {275-293},
  year    = {2011},
  doi     = {10.1007/s11134-011-9212-7},
  url     = {https://doi.org/10.1007/s11134-011-9212-7}
}

@article{BM19,
  author  = {Banerjee, Sayan and Mukherjee, Debankur},
  title   = {Join-the-shortest queue diffusion limit in {Halfin-Whitt} regime: Tail asymptotics and scaling of extrema},
  journal = {The Annals of Applied Probability},
  volume  = {29},
  number  = {2},
  pages   = {1262-1309},
  year    = {2019},
  doi     = {10.1214/18-AAP1436},
  url     = {https://doi.org/10.1214/18-AAP1436}
}

@article{BM20,
  author  = {Banerjee, Sayan and Mukherjee, Debankur},
  title   = {Join-the-shortest queue diffusion limit in {Halfin-Whitt} regime: Sensitivity on the heavy-traffic parameter},
  journal = {The Annals of Applied Probability},
  volume  = {30},
  number  = {1},
  pages   = {80-144},
  year    = {2020},
  doi     = {10.1214/19-AAP1496},
  url     = {https://doi.org/10.1214/19-AAP1496}
}

@article{BM2022,
  author  = {Bhambay, Sanidhay and Mukhopadhyay, Arpan},
  title   = {Asymptotic optimality of speed-aware {JSQ} for heterogeneous service systems},
  journal = {Performance Evaluation},
  volume  = {157-158},
  pages   = {102320},
  year    = {2022},
  doi     = {10.1016/j.peva.2022.102320},
  url     = {https://doi.org/10.1016/j.peva.2022.102320}
}

@article{bhambay2025asymptotic,
author = {Bhambay, Sanidhay and B\"{u}ke, Burak and Mukhopadhyay, Arpan},
title = {Asymptotic Optimality of the Speed-Aware Join-the-Shortest-Queue in the {Halfin-Whitt} Regime for Heterogeneous Systems},
journal = {Stochastic Systems},
volume = {15},
number = {2},
pages = {147-193},
year = {2025}
}

@book{billingsley1999convergence,
  title={Convergence of Probability Measures},
  author={Billingsley, P},
  journal={Wiley Series in Probability and Statistics},
  year={1999},
  publisher={John Wiley \& Sons, Inc.}
}

@book{bogachev2007measure,
  title     = {Measure Theory. Vol. II},
  author    = {Bogachev, Vladimir I.},
  year      = {2007},
  publisher = {Springer},
  address   = {Berlin},
  series    = {Fundamental Principles of Mathematical Sciences},
  volume    = {333},
  isbn      = {978-3-540-34515-9},
  doi       = {10.1007/978-3-540-34514-2}
}

@article{Braverman2020,
  title={Steady-state analysis of the join-the-shortest-queue model in the {Halfin-Whitt} regime},
  author={Braverman, Anton},
  journal={Mathematics of Operations Research},
  volume={45},
  number={3},
  pages={1069-1103},
  year={2020},
  publisher={INFORMS}
}

@article{BraSSY2023,
  author  = {Braverman, Anton},
  title   = {The Join-the-Shortest-Queue System in the {Halfin-Whitt} Regime: Rates of Convergence to the Diffusion Limit},
  journal = {Stochastic Systems},
  volume  = {13},
  number  = {1},
  pages   = {1-39},
  year    = {2023},
  doi     = {10.1287/stsy.2022.0102},
  url     = {https://doi.org/10.1287/stsy.2022.0102}
}

@article{Buke2022,
  author  = {B{\"u}ke, Burak},
  title   = {Modelling heterogeneity in many-server queueing systems},
  journal = {Queueing Systems},
  volume  = {100},
  number  = {3-4},
  pages   = {401-403},
  year    = {2022},
  doi     = {10.1007/s11134-022-09788-1},
  url     = {https://doi.org/10.1007/s11134-022-09788-1}
}

@article{buke2023many,
author = {B\"{u}ke, Burak and Qin, Wenyi},
title = {Many-Server Queues with Random Service Rates: A Unified Framework Based on Measure-Valued Processes},
journal = {Mathematics of Operations Research},
volume = {48},
number = {2},
pages = {748-783},
year = {2023}
}

@article{BdRP2025,
  author  = {B{\"u}ke, Burak and dos Reis, Gon{\c c}alo and Platonov, Vadim},
  title   = {Many-Server Queueing Systems with Heterogeneous Strategic Servers in Heavy Traffic},
  journal = {Operations Research},
  year    = {2025},
  doi     = {10.1287/opre.2022.0608},
  url     = {https://doi.org/10.1287/opre.2022.0608}
}

@article{CaoZhong2025,
  author  = {Cao, Ping and Zhong, Zhiheng},
  title   = {Asymptotically optimal routing of a many-server parallel queueing system with long-run average criterion},
  journal = {European Journal of Operational Research},
  volume  = {321},
  number  = {2},
  pages   = {462-475},
  year    = {2025},
  doi     = {10.1016/j.ejor.2024.09.044},
  url     = {https://doi.org/10.1016/j.ejor.2024.09.044}
}

@book{chen2001fundamentals,
  title     = {Fundamentals of Queueing Networks: Performance, Asymptotics, and Optimization},
  author    = {Chen, H. and Yao, D. D.},
  year      = {2001},
  publisher = {Springer},
  address   = {New York}
}

@article{ChenMoinzadehSongZhong2023,
  author  = {Chen, S. and Moinzadeh, K. and Song, J. S. and Zhong, Y.},
  title   = {Cloud Computing Value Chains: Research from the Operations Management Perspective},
  journal = {Manufacturing \& Service Operations Management},
  volume  = {25},
  pages   = {1338--1356},
  year    = {2023},
  doi     = {10.1287/msom.2022.1178},
  url     = {https://doi.org/10.1287/msom.2022.1178}
}

@article{down1995exponential,
  author  = {Down, D. and Meyn, S. P. and Tweedie, R. L.},
  title   = {Exponential and Uniform Ergodicity of {M}arkov Processes},
  journal = {Annals of Probability},
  volume  = {23},
  number  = {4},
  pages   = {1671-1691}, 
  year    = {1995},
  doi     = {10.1214/aop/1176987798},
  url     = {https://doi.org/10.1214/aop/1176987798}
}

@inproceedings{Duato2010,
  author = {Duato, Jos{\'e} and Pe{\~n}a, Antonio J. and Silla, Federico and Mayo, Rafael and Quintana-{Ort{\'i}}, Enrique S.},
  title     = {{rCUDA}: Reducing the number of {GPU}-based accelerators in high performance clusters},
  booktitle = {{2010 International Conference on High Performance Computing \& Simulation}},
  pages     = {224-231},
  year      = {2010},
  publisher = {IEEE},
  doi       = {10.1109/HPCS.2010.5547126},
  url       = {https://doi.org/10.1109/HPCS.2010.5547126}
}

@article{eschenfeldt2018join,
  title={Join the shortest queue with many servers. The heavy-traffic asymptotics},
  author={Eschenfeldt, Patrick and Gamarnik, David},
  journal={Mathematics of Operations Research},
  volume={43},
  number={3},
  pages={867-886},
  year={2018},
  publisher={INFORMS}
}

@book{feller1968introduction,
  author     = {Feller, William},
  title      = {An Introduction to Probability Theory and Its Applications},
  volume     = {1},
  edition    = {3rd},
  year       = {1968},
  publisher = {John Wiley \& Sons},
  address   = {New York}
}

@incollection{Gans2010,
  author    = {Gans, Noah and Liu, Nan and Mandelbaum, Avishai and Shen, Haipeng and Ye, Han},
  title     = {Service Times in Call Centers: Agent Heterogeneity and Learning with Some Operational Consequences},
  booktitle = {Borrowing Strength: Theory Powering Applications--A Festschrift for Lawrence D. Brown},
  series    = {Institute of Mathematical Statistics Collections},
  volume    = {6},
  pages     = {99-123},
  year      = {2010},
  publisher = {Institute of Mathematical Statistics},
  doi       = {10.1214/10-IMSCOLL608},
  url       = {https://doi.org/10.1214/10-IMSCOLL608}
}

@article{GHJM2021,
  author  = {Gardner, Kristen and Abdul Jaleel, Jazeem and Wickeham, Alexander and Doroudi, Sherwin},
  title   = {Scalable load balancing in the presence of heterogeneous servers},
  journal = {Performance Evaluation},
  volume  = {145},
  pages   = {102151},
  year    = {2021},
  doi     = {10.1016/j.peva.2020.102151},
  url     = {https://doi.org/10.1016/j.peva.2020.102151}
}

@article{HalfinWhitt1981,
  author  = {Halfin, Shlomo and Whitt, Ward},
  title   = {Heavy-traffic limits for queues with many exponential servers},
  journal = {Operations Research},
  volume  = {29},
  number  = {3},
  pages   = {567-588},
  year    = {1981},
  doi     = {10.1287/opre.29.3.567},
  url     = {https://doi.org/10.1287/opre.29.3.567}
}

@inproceedings{Huang2016,
  author    = {Huang, Muhuan and Wu, Di and Yu, Cody Hao and Fang, Zhenman and Interlandi, Matteo and Condie, Tyson and Cong, Jason},
  title     = {Programming and runtime support to blaze {FPGA} accelerator deployment at datacenter scale},
  booktitle = {Proceedings of the Seventh ACM Symposium on Cloud Computing},
  series    = {SoCC '16},
  pages     = {456-469},
  year      = {2016},
  publisher = {Association for Computing Machinery},
  isbn      = {9781450345255},
  doi       = {10.1145/2987550.2987569},
  url       = {https://doi.org/10.1145/2987550.2987569}
}

@article{HLM2022,
  author  = {Hurtado-Lange, Daniela Andrea and Maguluri, Siva Theja},
  title   = {A load balancing system in the many-server heavy-traffic asymptotics},
  journal = {Queueing Systems},
  volume  = {101},
  number  = {3-4},
  pages   = {353-391},
  year    = {2022},
  doi     = {10.1007/s11134-022-09847-7},
  url     = {https://doi.org/10.1007/s11134-022-09847-7}
}

@misc{IntelTBM3_2025,
  author  = {{Intel}},
  title   = {Frequently Asked Questions About {Intel}{\textregistered} Turbo Boost Max Technology 3.0},
  year    = {2025},
  url     = {https://www.intel.com/content/www/us/en/support/articles/000021587/processors.html},
  note    = {Accessed: 2026-04-15}
}

@misc{IntelAVX_2025,
  author  = {{Intel}},
  title   = {Can {Intel}{\textregistered} {Xeon}{\textregistered} W Processors Run AVX Loads Without Impacting Processor Frequency?},
  year    = {2025},
  url     = {https://www.intel.com/content/www/us/en/support/articles/000090746/processors.html},
  note    = {Accessed: 2026-04-15}
}

@misc{IntelRAPLGuidance2025,
  author  = {{Intel}},
  title   = {Frequency Throttling Side Channel Guidance},
  year    = {2025},
  month   = may,
  url     = {https://www.intel.com/content/www/us/en/developer/articles/technical/software-security-guidance/best-practices/frequency-throttling-side-channel-guidance.html},
  note    = {Published: 12 May 2025. Accessed: 2026-04-15}
}

@article{Jakubowski1986,
  author  = {Jakubowski, Adam},
  title   = {On the {Skorokhod} topology},
  journal = {Annales de l'I.H.P. Probabilités et Statistiques},
  year    = {1986},
  volume  = {22},
  number  = {3},
  pages   = {263-285},
  publisher = {Gauthier-Villars},
  url     = {https://www.numdam.org/item/AIHPB_1986__22_3_263_0/},
  language = {en}
}

@book{kallenberg2017random,
  author    = {Olav Kallenberg},
  title     = {Random Measures, Theory and Applications},
  series     = {Probability Theory and Stochastic Modelling},
  volume     = {77},
  publisher  = {Springer},
  year       = {2017},
  doi        = {10.1007/978-3-319-41598-7}
}

@book{LevinPeres2017,
  author    = {Levin, David A. and Peres, Yuval},
  title     = {Markov Chains and Mixing Times},
  edition   = {2},
  year      = {2017},
  publisher = {American Mathematical Society},
  isbn      = {978-1470429621}
}

@article{LiuYing2025,
  author        = {Liu, Xin and Ying, Lei},
  title         = {Zero-Waiting Load Balancing with Heterogeneous Servers in Heavy Traffic},
  journal       = {arXiv preprint arXiv:2509.23918},
  year          = {2025},
  eprint        = {2509.23918},
  archivePrefix = {arXiv},
  primaryClass  = {math.PR},
  url           = {https://arxiv.org/abs/2509.23918}
}

@article{Lu2011,
  author  = {Lu, Yi and Xie, Qiaomin and Kliot, Gabriel and Geller, Alan and Larus, James R. and Greenberg, Albert},
  title   = {Join-Idle-Queue: A novel load balancing algorithm for dynamically scalable web services},
  journal = {Performance Evaluation},
  volume  = {68},
  number  = {11},
  pages   = {1056-1071},
  year    = {2011},
  doi     = {10.1016/j.peva.2011.07.015},
  url     = {https://doi.org/10.1016/j.peva.2011.07.015}
}

@inproceedings{Marathe2017,
  author    = {Marathe, Aniruddha and Bailey, Peter E. and Lowenthal, David K. and Rountree, Barry and Schulz, Martin and de Supinski, Bronis R. and Gamblin, Todd},
  title     = {An Empirical Survey of Performance and Energy Efficiency Variation on Intel Processors},
  booktitle = {Proceedings of the 5th International Workshop on Energy Efficient Supercomputing},
  series    = {E2SC '17},
  year      = {2017},
  publisher = {Association for Computing Machinery},
  doi       = {10.1145/3149412.3149421},
  url       = {https://doi.org/10.1145/3149412.3149421}
}

@misc{McKinseyDataCenters2025,
  author       = {{McKinsey \& Company}},
  title        = {The Cost of Compute: A \$7 Trillion Race to Scale Data Centers},
  year         = {2025},
  month        = apr,
  note          = {https://www.mckinsey.com/industries/technology-media-and-telecommunications/our-insights/the-cost-of-compute-a-7-trillion-dollar-race-to-scale-data-centers (accessed June 25, 2026)},
  urldate      = {2026-04-21}
}

@phdthesis{Mitz96,
  author  = {Mitzenmacher, Michael David},
  title   = {The Power of Two Choices in Randomized Load Balancing},
  school  = {University of California, Berkeley},
  year    = {1996},
  type    = {PhD thesis},
  address = {Berkeley, CA},
  url     = {https://www.eecs.harvard.edu/~michaelm/postscripts/mythesis.pdf}
}

@article{mukherjee2016universality,
  title   = {Universality of load balancing schemes on the diffusion scale},
  author  = {Mukherjee, Debankur and Borst, Sem C. and van Leeuwaarden, Johan S. H. and Whiting, Philip A.},
  journal = {Journal of Applied Probability},
  volume  = {53},
  number  = {4},
  pages   = {1111--1124},
  year    = {2016},
  doi     = {10.1017/jpr.2016.68}
}

@article{MukhopadhyayMazumdar2016,
  author  = {Mukhopadhyay, Arpan and Mazumdar, Ravi R.},
  title   = {Analysis of randomized join-the-shortest-queue ({JSQ}) schemes in large heterogeneous processor-sharing systems},
  journal = {IEEE Transactions on Control of Network Systems},
  volume  = {3},
  number  = {2},
  pages   = {116-126},
  year    = {2016},
  doi     = {10.1109/TCNS.2015.2428331},
  url     = {https://doi.org/10.1109/TCNS.2015.2428331}
}

@ARTICLE{ptw07,
    author = {G. Pang and R. Talreja and W. Whitt},
    title = {Martingale proofs of many-server heavy-traffic limits for {M}arkovian queues},
    journal = {Probability Surveys},
    year = {2007},
  volume =    {4},
  number =    {},
  pages =     {193-267},
}

@inproceedings{Romanescu2008,
  author    = {Romanescu, Bogdan F. and Bauer, Michael E. and Ozev, Sule and Sorin, Daniel J.},
  title     = {Reducing the Impact of Intra-Core Process Variability with Criticality-Based Resource Allocation and Prefetching},
  booktitle = {Proceedings of the 5th International Conference on Computing Frontiers},
  series    = {CF '08},
  pages     = {129--138},
  year      = {2008},
  publisher = {Association for Computing Machinery},
  doi       = {10.1145/1366230.1366257},
  url       = {https://doi.org/10.1145/1366230.1366257}
}

@book{rudin1976principles,
  title     = {Principles of Mathematical Analysis},
  author    = {Rudin, Walter},
  edition   = {3rd},
  year      = {1976},
  publisher = {McGraw-Hill}
}

@article{sko56,
  author  = {A. V. Skorokhod},
  title   = {Limit Theorems for Stochastic Processes},
  journal = {Theory of Probability and Its Applications},
  volume  = {1},
  number  = {3},
  pages   = {261--290},
  year    = {1956}
}

@article{Stolyar2015,
  author  = {Stolyar, Alexander L.},
  title   = {Pull-based load distribution in large-scale heterogeneous service systems},
  journal = {Queueing Systems},
  volume  = {80},
  number  = {4},
  pages   = {341-361},
  year    = {2015},
  doi     = {10.1007/s11134-015-9448-8},
  url     = {https://doi.org/10.1007/s11134-015-9448-8}
}

@inproceedings{Subramanya2023Sia,
  author    = {Subramanya, Suhas Jayaram and Arfeen, Daiyaan and Lin, Shouxu and Qiao, Aurick and Jia, Zhihao and Ganger, Gregory R.},
  title     = {Sia: Heterogeneity-aware, goodput-optimized ML-cluster scheduling},
  booktitle = {Proceedings of the {ACM SIGOPS} 29th Symposium on Operating Systems Principles},
  series    = {SOSP '23},
  pages     = {642--657},
  year      = {2023},
  publisher = {Association for Computing Machinery},
  address   = {New York, NY, USA},
  isbn      = {979-8-4007-0229-7},
  doi       = {10.1145/3600006.3613175},
  url       = {https://doi.org/10.1145/3600006.3613175}
}

@inproceedings{Um2024Metis,
  author    = {Um, Taegeon and Oh, Byungsoo and Kang, Minyoung and Lee, Woo-Yeon and Kim, Goeun and Kim, Dongseob and Kim, Youngtaek and Muzzammil, Mohd and Jeon, Myeongjae},
  title     = {Metis: Fast Automatic Distributed Training on Heterogeneous GPUs},
  booktitle = {Proceedings of the 2024 USENIX Annual Technical Conference},
  pages     = {563--575},
  year      = {2024},
  publisher = {USENIX Association},
  isbn      = {978-1-939133-41-0},
  url       = {https://www.usenix.org/conference/atc24/presentation/um}
}

@article{vdB2022,
  author  = {Van der Boor, Mark and Borst, Sem C. and Van Leeuwaarden, Johan S. H. and Mukherjee, Debankur},
  title   = {Scalable load balancing in networked systems: A survey of recent advances},
  journal = {SIAM Review},
  volume  = {64},
  number  = {3},
  pages   = {554-622},
  year    = {2022},
  doi     = {10.1137/20M1323746},
  url     = {https://doi.org/10.1137/20M1323746}
}

@article{VDK96,
  author  = {Vvedenskaya, N. D. and Dobrushin, R. L. and Karpelevich, F. I.},
  title   = {A queueing system with a choice of the shorter of two queues: an asymptotic approach},
  journal = {Problems of Information Transmission},
  volume  = {32},
  number  = {1},
  pages   = {15-27},
  year    = {1996},
  url     = {https://www.mathnet.ru/eng/ppi298}
}

@inproceedings{Verma2015Borg,
  author    = {Verma, Abhishek and Pedrosa, Luis and Korupolu, Madhukar R. and Oppenheimer, David and Tune, Eric and Wilkes, John},
  title     = {Large-scale cluster management at Google with Borg},
  booktitle = {Proceedings of the Tenth European Conference on Computer Systems},
  series    = {EuroSys '15},
  articleno = {18},
  year      = {2015},
  publisher = {Association for Computing Machinery},
  doi       = {10.1145/2741948.2741964},
  url       = {https://doi.org/10.1145/2741948.2741964}
}

@article{WA2013,
  author  = {Ward, Amy R. and Armony, Mor},
  title   = {Blind Fair Routing in Large-Scale Service Systems with Heterogeneous Customers and Servers},
  journal = {Operations Research},
  volume  = {61},
  number  = {1},
  pages   = {228-243},
  year    = {2013},
  doi     = {10.1287/opre.1120.1129},
  url     = {https://doi.org/10.1287/opre.1120.1129}
}

@article{Weber1978,
  author  = {Weber, Richard R.},
  title   = {On the optimal assignment of customers to parallel servers},
  journal = {Journal of Applied Probability},
  volume  = {15},
  number  = {2},
  pages   = {406-413},
  year    = {1978},
  doi     = {10.2307/3213411},
  url     = {https://doi.org/10.2307/3213411}
}

@book{whitt2002stochastic,
  author    = {Ward Whitt},
  title     = {Stochastic‐Process Limits: An Introduction to Stochastic‐Process Limits and Their Application to Queues},
  series    = {Springer Series in Operations Research and Financial Engineering},
  year      = {2002},
  publisher = {Springer},
  address   = {New York, NY},
  isbn      = {9780387953588},
  doi       = {10.1007/b97479},
  pages     = {602}
}

@article{Winston1977,
  author  = {Winston, Wayne},
  title   = {Optimality of the shortest line discipline},
  journal = {Journal of Applied Probability},
  volume  = {14},
  number  = {1},
  pages   = {181-189},
  year    = {1977},
  doi     = {10.1017/S0021900200104772},
  url     = {https://doi.org/10.1017/S0021900200104772}
}

\end{document}